\documentclass[11pt]{amsart}
\usepackage{amsmath,amssymb,mathrsfs,color}
\def\L{{\mathcal{L}}}
\def\d{{\rm{d}}}
\def\eps{\varepsilon}

\def\N{{\mathbb N}}

\def\R{{\mathbb R}}

\def\P{{\mathbb P}}
\newtheorem{lemma}{Lemma}[section]
\newtheorem{theorem}[lemma]{Theorem}
\newtheorem{remark}[lemma]{Remark}
\newtheorem{prop}[lemma]{Proposition}
\newtheorem{coro}[lemma]{Corollary}
\newtheorem{definition}[lemma]{Definition}
\newtheorem{example}[lemma]{Example}
\allowdisplaybreaks
\numberwithin{equation}{section}
\begin{document}

\title[The second Bogolyubov theorem and global averaging principle]
{The second Bogolyubov theorem and global averaging principle
for SPDEs with monotone coefficients}

%%First author
\author{mengyu Cheng}
\address{M. Cheng: School of Mathematical Sciences,
Dalian University of Technology, Dalian 116024, P. R. China}
\email{mengyucheng@mail.dlut.edu.cn; mengyu.cheng@hotmail.com}

%% Second author
\author{zhenxin Liu}
\address{Z. Liu (Corresponding author): School of Mathematical Sciences,
Dalian University of Technology, Dalian 116024, P. R. China}
\email{zxliu@dlut.edu.cn}

%%%%%%%%%%%%%%%%%%%%%%%%%%%%%%%%%%%
%\date{\today}
\date{June 27, 2022}

\begin{abstract}
In this paper, we establish the second Bogolyubov theorem
and global averaging principle for stochastic partial
differential equations (in short, SPDEs) with
monotone coefficients. Firstly, we prove that there
exists a unique $L^{2}$-bounded solution to SPDEs with
monotone coefficients and this bounded solution
is globally asymptotically stable in square-mean sense.
Then we show that the $L^{2}$-bounded solution possesses the
same recurrent properties (e.g. periodic, quasi-periodic,
almost periodic, almost automorphic, Birkhoff recurrent,
Levitan almost periodic, etc.) in distribution sense as the
coefficients. Thirdly, we prove that the recurrent solution of
the original equation converges to the stationary solution of
averaged equation under the compact-open topology as the time
scale goes to zero --- in other words, there exists a unique
recurrent solution to the original equation in a neighborhood
of the stationary solution of averaged equation when the time
scale is small. Finally, we establish the global averaging
principle in weak sense, i.e. we show that the
attractor of original system tends to that of the averaged
equation in probability measure space as the time scale
goes to zero. For illustration of our results, we give two
applications, including stochastic reaction diffusion equations
and stochastic generalized porous media equations.
\end{abstract}

\maketitle

\section{Introduction}
Averaging principle is  an effective method for studying
dynamical systems with highly oscillating components.
Under suitable conditions, the highly oscillating components
can be ``averaged out" to produce an averaged system.
The averaged system is easier for analysis and governs the
evolution of the original system over long time scales.

Consider the following deterministic systems in
$\R^{n},n\in\mathbb N$:
\begin{equation}\label{dfode}
\dot{X}^{\varepsilon}=F\left(\frac{t}{\varepsilon},
X^{\varepsilon}\right)
\end{equation}
and
\begin{equation}\label{dfodeave}
\dot{X}=\bar{F}(X)
\end{equation}
for small parameter $0<\varepsilon\ll 1$,
where $F\in C(\R\times\R^{n},\R^{n})$, and
$
\bar{F}(x)=\lim\limits_{T\rightarrow\infty}\frac{1}{T}
\int_{0}^{T}F(t,x)\d t.
$

It is a basic problem of averaging principle to determine in
what sense the behavior of solutions to the averaged
system \eqref{dfodeave} approximates the behavior of solutions
to the non-autonomous system \eqref{dfode} as the time scale
$\varepsilon$ goes to zero. For the connotation of
approximation, there are three natural types of interpretation.
One is the so-called {\em first Bogolyubov theorem},
i.e. the convergence of the solution of the original Cauchy problem
\eqref{dfode} to that of the averaged equation \eqref{dfodeave} on
a finite interval $[0,T]$ when the initial data are such that
$X^{\varepsilon}(0)=X(0)$. And another one is to request that the
approximation be valid on the entire real axis, which is the
so-called {\em second Bogolyubov theorem} (sometimes called ``theorem for
periodic solution by averaging"). In addition, it is meaningful
to determine whether the attractor of the averaged equation
\eqref{dfodeave} approximates the attractor of the original
equation \eqref{dfode}. One calls this result the {\em global
averaging principle}.

The idea of averaging dates back to the perturbation theory
which was proposed by Clairaut, Laplace and Lagrange in the 18th
century. Then fairly rigorous averaging method for nonlinear
oscillations was presented by Krylov, Bogolyubov and
Mitropolsky \cite{KB1943, BM1961}, which is called the
Krylov-Bogolyubov method nowadays. After that, there is a lot
of works on averaging for deterministic finite and infinite
dimensional systems, which we will not mention here.

Meanwhile, Stratonovich firstly proposed the stochastic averaging
method on the basis of physical considerations, which was later
proved mathematically by Khasminskii. Then extensive
investigations concerning averaging
principle for stochastic differential equations were conducted,
following Khasminskii's mathematically pioneering work \cite{Khas1968};
see, e.g. \cite{BK2004, Cerr2009, Cerr2011, CF2009, CL2017, DW2014, FW2006,
FW2012, Gao, Gao2021, Kifer2004,  LRSX2020, MSV1991, RX2021,
Skor1989, SXX2021, Vere1990, Vrko1995, WR2012} and the
references therein. Note that the above existing results are
concerned with the first Bogolyubov theorem.

Despite considerable advances in this direction,
there are  few works on stochastic averaging concerning with
the second Bogolyubov theorem, which states:
there exists a unique periodic solution to the original
equation in a neighborhood of the stationary solution of
averaged equation. Consider the following
SPDEs on a separable Hilbert space
$\left(H,\langle\cdot,\cdot\rangle\right)$
\begin{equation}\label{origeq}
\d X_{\varepsilon}(t)=\left(A(X_{\varepsilon}(t))
+F\left(\frac{t}{\varepsilon},X_{\varepsilon}(t)\right)\right)\d t
+G\left(\frac{t}{\varepsilon},X_{\varepsilon}(t)\right)\d W(t),
\end{equation}
where  $A$ satisfies some monotone
condition, $F$ and $G$ are Lipschtiz in second variable.
Here $W$ is a two-sided cylindrical Wiener process defined on
another separable Hilbert space $U$  and $0<\varepsilon\leq1$.
Recall that the second Bogolyubov theorem for stochastic
differential equations with almost periodic coefficients was
studied in \cite{KMF_2015},
and \cite{CL2020} investigated the averaging principle
for stochastic ordinary differential equations with general recurrent coefficients. As discussed in
\cite{KMF_2015} and \cite{CL2020}, equations are semilinear with
globally Lipschitz and linear growth nonlinear terms,
which cannot cover the monotone case.
However, the coefficients of many interesting
practical models just satisfy monotone conditions.
Some typical examples are reaction diffusion equations
and porous media equations.

Generally, reaction diffusion equations can be used to describe
the growth of biological population and the spatial spread of
epidemic diseases, which are largely affected by time-varying
environment. In particular, the recurrent phenomenon has been found
in the growth of population and the spread of diseases,
since some regular environmental changes such as seasonal changes.
And porous media equations appear in the description of different
natural phenomenon related to diffusion, filtration or heat
propagation. Since the noise models the small irregular fluctuations
generated by microscopic effects, it is more practical to
consider the above systems perturbed by white noise.

From the perspective of theoretical and practical value,
we establish the second Bogolyubov theorem for SPDEs
with monotone coefficients in this paper.
More specifically, consider equation \eqref{origeq}, we assume
that $A$ is strongly monotone. Compared with the assumption that
$A$ is a linear bounded operator in \cite{CL2020}, this condition
admits wider applications. It includes unbounded linear operators
and quasi-linear operators.

Denoting by
$F_{\varepsilon}(t,x):=F(\frac{t}{\varepsilon},x)$ and
$G_{\varepsilon}(t,x):=G(\frac{t}{\varepsilon},x)$, we transform
equation \eqref{origeq} to
\begin{equation}\label{origeq2}
\d X_{\varepsilon}(t)=\left(A(X_{\varepsilon}(t))
+F_{\varepsilon}(t,X_{\varepsilon}(t))\right)\d t
+G_{\varepsilon}(t,X_{\varepsilon}(t))\d W(t).
\end{equation}
In this paper, we firstly show that
there exists a unique $L^{2}$-bounded solution
$X_{\varepsilon}(t),t\in\R$ of \eqref{origeq2} which
shares the same recurrent properties (in particular, periodic,
quasi-periodic, almost periodic, almost automorphic, Birkhoff
recurrent, Levitan almost periodic, almost recurrent,
pseudo-periodic, pseudo-recurrent, Poisson stable)
in distribution sense as the coefficients for each
$0<\varepsilon\leq1$. And we prove that this $L^{2}$-bounded
solution $X_{\varepsilon}(t),t\in\R$ to \eqref{origeq2}
is globally asymptotically stable in square-mean sense.
Without loss of generality, we assume
$\varepsilon=1$ in this part, then the $L^{2}$-bounded solution
is denoted by $X(t),t\in\R$. Note that coefficient $F$ in
\eqref{origeq2} need not to be Lipschitz in this part.
This result is interesting on its own rights. To our knowledge,
there are only a few works on general recurrent solutions to SPDEs in a unified framework, see
\cite{CL_2017, LL2020}. As discussed in \cite{CL_2017}
and \cite{LL2020}, they dealt with recurrent solutions
to semilinear SPDEs with Lipschitz continuous and
globally linear growth nonlinearities according to Shcherbakov's
comparability method by character of recurrence.
B. A. Shcherbakov gave the existence condition of at least one
(or exactly one) solution to deterministic equation with the
same character of recurrence as the coefficient.
This solution is said to be compatible (respectively,
uniformly compatible). Comparing to \cite{CL_2017}
and \cite{LL2020}, we consider SPDEs with monotone
coefficients.

Let $X_{\varepsilon}$ be the recurrent solution to equation
\eqref{origeq2}. Then one of the major aims of this paper
is to prove that
\begin{equation}\label{result1}
\lim_{\varepsilon\rightarrow0}d_{BL}(\mathcal L
      (X_{\varepsilon}),\mathcal L(\bar{X}))=0 \qquad
      {\rm in}~Pr(C(\R,H))
\end{equation}
(see Theorem \ref{averth} and Corollary \ref{averthcoro}),
where $d_{BL}$ is the bounded Lipschitz distance
(also called Fortet-Mourier distance); see Subsection \ref{varapproach} for details.
And $\bar{X}$ is the unique stationary solution
of the following averaged equation
\begin{equation}\label{avereq}
\d X(t)=\left(A(X(t))+\bar{F}(X(t))\right)\d t
+\bar{G}(X(t))\d W(t).
\end{equation}
Here $\bar{F}\in C(H,H),
\bar{G}\in C(H,L_{2}(U,H))$, $\bar{F}$ and $\bar{G}$ satisfy
$$\lim\limits_{T\rightarrow\infty}\frac{1}{T}\int_{t}^{t+T}F(s,x)
\d s=\bar{F}(x),\quad
\lim\limits_{T\rightarrow\infty}\frac{1}{T}\int_{t}^{t+T}
\|G(s,x)-\bar{G}(x)\|_{L_{2}(U,H)}^{2}\d s=0$$
uniformly with respect to $t\in\R$.

This averaging principle is also applicable to
the following system
\begin{equation}\label{smeq}
\d X(t)=\varepsilon\left(A(X(t))+F(t,X(t))\right)\d t
+\sqrt{\varepsilon}G(t,X(t))\d W(t),
\end{equation}
where $0<\varepsilon\leq1$. With the time scaling
$t\mapsto \frac{t}{\varepsilon}$, denote by
$\Phi_{\varepsilon}(t):=X(\frac{t}{\varepsilon})$ and
$W_{\varepsilon}(t):=\sqrt{\varepsilon}W(\frac{t}{\varepsilon})$
for all $t\in\R$, we transform equation \eqref{smeq} to
\begin{equation}\label{smteq1}
\d \Phi_{\varepsilon}(t)=\left(A(\Phi_{\varepsilon}(t))
+F_{\varepsilon}(t,\Phi_{\varepsilon}(t))\right)\d t
+G_{\varepsilon}(t,\Phi_{\varepsilon}(t))\d W_{\varepsilon}(t).
\end{equation}
Then we can consider the following equation
\begin{equation}\label{smteq2}
\d \tilde{X}_{\varepsilon}(t)=\left(A(\tilde{X}_{\varepsilon}(t))
+F_{\varepsilon}(t,\tilde{X}_{\varepsilon}(t))\right)\d t
+G_{\varepsilon}(t,\tilde{X}_{\varepsilon}(t))\d W(t).
\end{equation}
It is obvious that
$\L(\tilde{X}_{\varepsilon}(t))=\L(\Phi_{\varepsilon}(t))$
for any $t\in\R$.

In contrast to the first Bogolyubov averaging principle
for Cauchy problem of stochastic differential equations
on finite intervals, we prove that there exists a
unique recurrent solution in a small neighborhood of
the stationary solution to the averaged equation
when the time scale is small. Note that it is
non-initial value problem, and this recurrent solution is more
general than the classical second Bogolyubov theorem which only
treats the periodic case.

Since the SPDEs we concern in this paper
are not semilinear, the semigroup framework in
\cite{KMF_2015} and \cite{CL2020} is not applicable to our
problem. Therefore, a difficulty that we face is how to
deal with the monotone SPDEs. Firstly, under some suitable
conditions, employing the technique of truncation
which is used in
\cite{Cerr2009, CF2009, CL2017, LRSX2020}, we show that
\[
\lim\limits_{\varepsilon\rightarrow0} E
\sup_{s\leq t\leq s+T}\|X_{\varepsilon}(t,s,\zeta_s^\varepsilon)-\bar{X}(t,s,\zeta_s)\|^{2}=0
\]
for all $s\in\R$ and $T>0$ provided
$\lim\limits_{\varepsilon\rightarrow0} E
\|\zeta^{\varepsilon}_{s}-\zeta_{s}\|^{2}=0$,
where $X_{\varepsilon}(t,s,\zeta_s^\varepsilon)$ is the solution
of \eqref{origeq2} with the initial condition
$X_{\varepsilon}(s,s,\zeta_s^\varepsilon)=\zeta^{\varepsilon}_{s}$
and $\bar{X}(t,s,\zeta_s)$ is the solution of \eqref{avereq}
with the initial condition $\bar{X}(s,s,\zeta_s)=\zeta_{s}$
(see Theorem \ref{avethf}). In fact,
this is the first Bogolyubov theorem, which is new despite that
there have already been many results in this direction mentioned above.

In view of Theorem \ref{avethf}, tightness of family of measures
$\{\mathbb P\circ[X_{\varepsilon}(t)]^{-1}
\}_{\varepsilon\in(0,1]}$ for any $t\in\mathbb R$
plays an important  role in establishing the second
Bogolyubov theorem. Although the tightness of
$\{\mathbb P\circ[X_{\varepsilon}(t)]^{-1}\}_{\varepsilon\in(0,1]}$
on $C([0,T];H)$ was proved by using Ascoli-Arzel\`a theorem
and the Garcia-Rademich-Rumsey theorem in \cite{Cerr2009,
CF2009, CL2017}, it is different from the technique used in our
paper. We find that
\begin{equation}\label{Ineq3}
\sup\limits_{t\in\R} E
\|X_{\varepsilon}(t)\|_{S}^{2}<\infty
\end{equation}
uniformly with respect to $\varepsilon\in(0,1]$.
Therefore, for any $t\in\R$, the tightness of
$\{\mathbb P\circ[X_{\varepsilon}(t)]^{-1}\}_{\varepsilon\in(0,1]}$
is a consequence of the compactness of the inclusion $S\subset H$.

Another major result in present paper is to establish the global
averaging principle in weak sense. Namely, we prove that uniform
attractor of original system tends to uniform attractor of
averaged equation in probability measure space $Pr(H)$.
Global averaging of deterministic systems was conducted, see e.g.
\cite{HV1990, Ily1996, Ily1998, Zel2006} among others.
But to our knowledge, there is no work so far on global averaging
for stochastic equations. Notice that it is still an open problem
whether solutions of general SPDEs can generate random dynamical
systems (in short, RDS). So we consider attractors in the probability
measure space $Pr(H)$ instead of pullback attractors in the framework
of RDS. That is why we call it ``in weak sense".

Let $Pr_{2}(H)$ be a subspace of $Pr(H)$ such that
\[
\int_{H}\|z\|^{2}\mu(\d z)<+\infty
\]
for any $\mu\in Pr_{2}(H)$.
With the transition probability $P_{\mathbb F}
(s,x,t,\d y):=\mathbb P\circ\left(X(t,s,x)\right)^{-1}
(\d y)$ to equation \eqref{origeq2} when $\varepsilon=1$,
we associate a mapping
$P^{*}(t,\mathbb F,\cdot):Pr(H)\rightarrow Pr(H)$
defined by
\[
P^{*}(t,\mathbb F,\mu)(B):=
\int_{H}P_{\mathbb F}(0,x,t,B)\mu(\d x)
\]
for all $\mu\in Pr(H)$, $B\in\mathcal B(H)$ and $\mathbb F:=(F,G)$.
Firstly, we show that $P^{*}$ is a cocycle over
$\left(H(\mathbb F),\R,\sigma\right)$ with fiber $Pr_{2}(H)$,
where $\left(H(\mathbb F),\R,\sigma\right)$ is a shift
dynamical system (see Section 5 for details).
Suppose that $H(\mathbb F)$ is compact.
Then we prove that $P_{\varepsilon}^{*}$ associated with
\eqref{origeq2} has a uniform attractor
$\mathcal{A}^{\varepsilon}$ in $Pr_{2}(H)$
for any $0<\varepsilon\leq1$, and
\[
 \lim_{\varepsilon\rightarrow0}{\rm dist}_{Pr_{2}(H)}\left(
        \mathcal{A}^{\varepsilon},\bar{\mathcal{A}}\right)=0
\]
(see Theorem \ref{gath}), where dist$_{Pr_{2}(H)}$ is the Hausdorff
semi-metric and $\bar{\mathcal{A}}:=\{\L(\bar{X}(0))\}$ is the
attractor of $\bar{P}^{*}$ to the averaged equation \eqref{avereq}.
Note that $H(\mathbb F)$ is compact provided $\mathbb F$
is Birkhoff recurrent.

The remainder of this paper is organized as follows. In the next
section, we recall some definitions and facts concerning
dynamical systems, Poisson stable (or recurrent) functions,
Shcherbakov's comparability method by character of
recurrence and variational approach.
In the third section, we show that there exists a unique
$L^{2}$-bounded solution which possesses the same recurrent
properties in distribution sense as the coefficients
and this bounded solution is globally
asymptotically stable in square-mean sense.
In section 4, we establish the second Bogolyubov theorem
for SPDEs with monotone coefficients.
In section 5, we prove the global averaging principle for
these SPDEs.
In the last section, we illustrate our theoretical results by
stochastic reaction diffusion equations and
stochastic generalized porous media equations.

\section{Preliminaries}

In this section, we introduce some useful preliminaries,
including dynamical systems, poisson stable functions,
Shcherbakov's comparability method by character of
recurrence, varational approach.

\subsection{Shift dynamical systems}

In this subsection, let $(\mathcal X,\rho)$ be a
complete metric space and $(\mathcal X,\mathbb R,\pi)$
be a dynamical system (flow) on $\mathcal X$, i.e. the
mapping $\pi :\mathbb R\times \mathcal X\to \mathcal X$
is continuous, $\pi(0,x)=x$ and
$\pi(t+s,x)=\pi(t,\pi(s,x))$ for any $x\in \mathcal X$ and
$t,s\in\mathbb R$. We write $C(\mathbb R,\mathcal X)$
to mean the space of all continuous functions
$\varphi :\mathbb R \to \mathcal X$ equipped with the distance
\begin{equation*}\label{eqD1}
d(\varphi_{1},\varphi_{2}):=\sum_{k=1}^{\infty}\frac{1}{2^k}
\frac{d_{k}(\varphi_1,\varphi_{2})}{1+d_{k}(\varphi_1,\varphi_{2})},
\end{equation*}
where
$$
d_{k}(\varphi_1,\varphi_{2})
:=\sup\limits_{|t|\le k}\rho(\varphi_1(t),\varphi_{2}(t)),
$$
which generates the compact-open topology on $C(\mathbb R,\mathcal X)$.
The space $(C(\mathbb R,\mathcal X),d)$ is a complete metric space
(see, e.g. \cite{Sel, Sch72, Sch85, sib}).

\begin{remark}\label{remCh}\rm
Let $\{\varphi_{n}\}_{n=1}^{\infty}, \varphi\in C(\mathbb R,\mathcal X)$.
Then the following statements are equivalent.
\begin{enumerate}
  \item $\lim\limits_{n\to \infty}d(\varphi_{n},\varphi)=0$.
  \item $\lim\limits_{n\to\infty}\max\limits_{|t|\le l}
        \rho(\varphi_{n}(t),\varphi(t))=0$
         for any $l>0$.
  \item There exists a sequence $l_n\to +\infty$ such that
        $\lim\limits_{n\to \infty}\max\limits_{|t|\le l_n}
        \rho(\varphi_{n}(t),\varphi(t))=0$.
\end{enumerate}
\end{remark}

Let us now consider two examples of shift dynamical systems
which we will use in this paper.

\begin{example}\label{ex1}\rm
We say $\varphi^{\tau}$ is the {\em $\tau$-translation} of $\varphi$
if $\varphi^{\tau}(t):=\varphi(t+\tau)$ for any $t\in\R$ and
$\varphi\in C(\R,\mathcal X)$. For any
$(\tau,\varphi)\in\R\times C(\R,\mathcal X)$, the mapping
$\sigma:\R\times C(\R,\mathcal X)\rightarrow C(\R,\mathcal X)$ is defined
by $\sigma(\tau,\varphi):=\varphi^{\tau}$. Then the triplet
$\left(C(\R,\mathcal X),\R,\sigma\right)$ is a dynamical system which is
called {\em shift dynamical system} or {\em Bebutov's dynamical system}.
Indeed, it is easy to check that $\sigma(0,\varphi)=\varphi$ and
$\sigma(\tau_{1}+\tau_{2},\varphi)
=\sigma(\tau_{2},\sigma(\tau_{1},\varphi))$ for any
$\varphi\in C(\R,\mathcal X)$ and $\tau_{1},\tau_{2}\in\R$. And it can be
proved that the mapping
$\sigma:\R\times C(\R,\mathcal X)\rightarrow C(\R,\mathcal X)$
is continuous, see, e.g. \cite{Ch2015, Sel, Sch72, sib}.
\end{example}
In what follows, let $(\mathcal Y,\rho_{1})$ be a complete metric space.
We employ $H(\varphi)$ to denote the hull of $\varphi$, which is the set
of all the limits of $\varphi^{\tau_{n}}$ in $C(\R,\mathcal X)$, i.e.
\[
H(\varphi):=\{\psi\in C(\R,\mathcal X):\psi=\lim_{n\rightarrow\infty}
\varphi^{\tau_{n}} ~{\rm for~ some~ sequence~}
\{\tau_{n}\}\subset \R \}.
\]
Notice that the set $H(\varphi)\subset C(\R,\mathcal X)$ is closed and
translation invariant. Consequently, it naturally defines on $H(\varphi)$
a shift dynamical system $\left(H(\varphi),\R,\sigma\right)$. Now we give
the second example, which is similar to Section 2.4 in \cite{CL_2017}.

\begin{example}\label{ex2}\rm
We write $BUC(\R\times\mathcal X,\mathcal Y)$ to mean the space of
all continuous functions $f:\R\times\mathcal X\rightarrow\mathcal Y$
which satisfy the following conditions:
\begin{enumerate}
  \item $f$ is bounded on every bounded subset from $\R\times\mathcal X$;
  \item $f$ is continuous in $t\in\R$ uniformly with respect to $x$
  on each bounded subset $Q\subset\mathcal X$.
\end{enumerate}
We endow $BUC(\R\times\mathcal X,\mathcal Y)$ with the following $d$ metric
\begin{equation}\label{dBUC}
d(f,g):=\sum_{k=1}^{\infty}\frac{1}{2^{k}}\frac{d_{k}(f,g)}{1+d_{k}(f,g)},
\end{equation}
where $d_{k}(f,g):=\sup\limits_{|t|\leq k,x\in Q_{k}}
\rho_{1}(f(t,x),g(t,x))$.
Here $Q_{k}\subset\mathcal X$ is bounded,
$Q_{k}\subset Q_{k+1}$ and $\cup_{k\in\N}Q_{k}=\mathcal X$.
Note that $d$ generates the topology of uniform convergence
on  bounded subsets on $BUC(\R\times\mathcal X,\mathcal Y)$
and $\left(BUC(\R\times\mathcal X,\mathcal Y),d\right)$
is a complete metric space.

Given $f\in BUC(\R\times\mathcal X,\mathcal Y)$ and
$\tau\in\R$. We write $f^{\tau}$ to mean the
$\tau$-translation of $f$ if $f^{\tau}(t,x):=f(t+\tau,x)$
for all $(t,x)\in\R\times\mathcal X$. It is proved that
$BUC(\R\times\mathcal X,\mathcal Y)$ is invariant with
respect to translations. Like in Example \ref{ex1},
we define a mapping
$\sigma:\R\times BUC(\R\times\mathcal X,\mathcal Y)
\rightarrow BUC(\R\times\mathcal X,\mathcal Y)$,
$(\tau,f)\mapsto f^{\tau}$.
Then it can be proved that the triplet
$\left(BUC(\R\times\mathcal X,\mathcal Y),\R,\sigma\right)$
is a dynamical system. See Chapter I in \cite{Ch2015}
for details. Given $f\in BUC(\R\times\mathcal X,\mathcal Y)$,
$H(f)\subset BUC(\R\times\mathcal X,\mathcal Y)$ is closed
and translation invariant. Consequently, it naturally defines
on $H(f)$ a shift dynamical system $\left(H(f),\R,\sigma\right)$.

We employ $BC(\mathcal X,\mathcal Y)$ to denote the space of
all continuous functions $f:\mathcal X\rightarrow\mathcal Y$
which are bounded on every bounded subset of  $\mathcal X$
and equipped with the distance
\[
d(f,g):=\sum_{k=1}^{\infty}\frac{1}{2^{k}}
\frac{d_{k}(f,g)}{1+d_{k}(f,g)},
\]
where $d_{k}(f,g):=\sup\limits_{x\in Q_{k}}\rho_{1}(f(x),g(x))$,
$Q_{k}$ is similar to that in Example \ref{ex2}.
Note that $\left(BC(\mathcal X,\mathcal Y),d\right)$
is a complete metric space. For any
$F\in BUC(\R\times\mathcal X,\mathcal Y)$, the mapping
$\mathcal F:\R\rightarrow BC(\mathcal X,\mathcal Y)$ is
defined by
$\mathcal F(t):=F(t,\cdot):\mathcal X\rightarrow\mathcal Y$.
Clearly, $\mathcal F\in C(\R,BC(\mathcal X,\mathcal Y))$.
\end{example}

\begin{remark}\label{remF1}\rm
It can be proved that the following statements are true.
\begin{enumerate}
\item
The mapping $h: BUC(\mathbb R\times \mathcal X,\mathcal Y)
\rightarrow C(\mathbb R,BC(\mathcal X,\mathcal Y))$
defined by equality $h(F):=\mathcal F$ establishes an isometry between
$ BUC(\mathbb R\times \mathcal X,\mathcal Y)$
and $C(\mathbb R, BC(\mathcal X, \mathcal Y))$.
\item
$h(F^{\tau})=\mathcal F^{\tau}$ for any $\tau\in \mathbb R$ and
$F\in BUC(\mathbb R\times \mathcal X,\mathcal Y)$,
i.e. the shift dynamical systems
$(BUC(\mathbb R\times \mathcal X,\mathcal Y),\mathbb R,\sigma)$ and
$(C(\mathbb R,BC(\mathcal X,\mathcal Y)),\mathbb R,\sigma)$
are (dynamically) homeomorphic.
\end{enumerate}
\end{remark}

\subsection{Poisson stable functions}

Let us recall the types of Poisson stable (or recurrent) functions to be studied in
this paper; For further details and the relations among these types
of functions, see \cite{Sel,Sch72,Sch85,sib}.

\begin{definition} \rm
We say that a function $\varphi\in C(\mathbb R,\mathcal X)$ is
{\em $T$-periodic}, if there exists a constant $T\in\R$ such that
$\varphi(t+T)=\varphi(t)$ for all $t\in\R$. In particular, $\varphi$
is called {\em stationary} provided $\varphi(t)=\varphi(0)$
for all $t\in \mathbb R$.
\end{definition}

\begin{definition} \rm
A function $\varphi\in C(\R,\mathcal X)$ is called
{\em Bohr almost periodic} if the set $\mathcal T(\varphi,\varepsilon)$ of
$\varepsilon$-almost periods of $\varphi$ is relatively dense for each
$\varepsilon>0$, i.e. for each $\varepsilon>0$ there exists a constant
$l=l(\varepsilon)>0$ such that
$\mathcal T(\varphi,\varepsilon)\cap[a,a+l]\not=\emptyset$
for all $a\in\R$, where
\[
\mathcal T(\varphi,\varepsilon):=\left\{\tau\in\R:
\sup_{t\in\R}\rho(\varphi(t+\tau),\varphi(t))<\varepsilon\right\},
\]
and $\tau\in\mathcal T(\varphi,\varepsilon)$ is called
{\em $\varepsilon$-almost period} of $\varphi$.
\end{definition}

\begin{definition} \label{ppf}\rm %(\cite[p.32]{Bohr_I1947})
We say that $\varphi \in C(\mathbb R,\mathcal X)$ is
{\em pseudo-periodic in the positive} (respectively, {\em negative})
{\em direction} if for each $\varepsilon >0$ and $l>0$ there exists a
$\varepsilon$-almost period $\tau >l$ (respectively, $\tau <-l$)
of the function $\varphi$. The function $\varphi$ is called
{\em pseudo-periodic} if it is pseudo-periodic in both directions.
\end{definition}

\begin{definition} \rm
A function $\varphi \in C(\mathbb R,\mathcal X)$ is called {\em almost
recurrent (in the sense of Bebutov)} if the set
$\mathfrak T(\varphi,\varepsilon)$ is relatively dense for every
$\varepsilon>0$, where
$\mathfrak T(\varphi,\varepsilon)
:=\{\tau\in\R:d(\varphi^{\tau},\varphi)<\varepsilon\}$.
And $\tau\in\mathfrak T(\varphi,\varepsilon)$ is said to be
{\em $\varepsilon$-shift} for $\varphi$.
\end{definition}

\begin{definition} \rm
\begin{enumerate}
\item We say that a function $\varphi\in C(\mathbb R,\mathcal X)$ is
      {\em Lagrange stable} provided $\{\varphi^{h}:\ h\in \mathbb R\}$
      is a relatively compact subset of $C(\mathbb R,\mathcal X)$.
\item We say that a function $\varphi \in C(\mathbb R,\mathcal X)$ is
      {\em Birkhoff recurrent} if it is almost recurrent
      and Lagrange stable.
\end{enumerate}

\end{definition}

\begin{definition} \rm
We say that $\varphi \in C(\mathbb R,\mathcal X)$ is
{\em Poisson stable in the positive }
(respectively, {\em negative}) {\em direction} if for
every $\varepsilon >0$ and $l>0$ there exists $\tau >l$
(respectively, $\tau <-l$) such that
$d(\varphi^{\tau},\varphi)<\varepsilon$. The function $\varphi$ is
called {\em Poisson stable} if it is Poisson stable in both directions.
\end{definition}

\begin{definition} \rm
We say that $\varphi\in C(\mathbb R,\mathcal X)$ is
{\em Levitan almost periodic}
if there exists an almost periodic function
$\psi \in C(\mathbb R,\mathcal Y)$ such that for any
$\varepsilon >0$ there exists $\delta =\delta (\varepsilon)>0$
such that $\mathcal T(\psi,\delta)
\subseteq \mathfrak T(\varphi,\varepsilon)$.
\end{definition}

\begin{definition}\label{defAA1} \rm
We say that a function $\varphi \in C(\mathbb R,\mathcal X)$ is
{\em almost automorphic} if it is Levitan almost periodic and
Lagrange stable.
\end{definition}

\begin{definition} \rm
We say that $\varphi \in C(\mathbb R,\mathcal X)$ is
{\em quasi-periodic with the spectrum of frequencies
$\nu_1,\nu_2,\ldots,\nu_k$}
if it satisfies the following conditions:
\begin{enumerate}
\item the numbers $\nu_1,\nu_2,\ldots,\nu_k$ are rationally
independent; \item there exists a continuous function $\Phi
:\mathbb R^{k}\to \mathcal X$ such that
$\Phi(t_1+2\pi,t_2+2\pi,\ldots,t_k+2\pi)=\Phi(t_1,t_2,\ldots,t_k)$
for all $(t_1,t_2,\ldots,t_k)\in \mathbb R^{k}$; \item
$\varphi(t)=\Phi(\nu_1 t,\nu_2 t,\ldots,\nu_k t)$ for $t\in
\mathbb R$.
\end{enumerate}
\end{definition}

Let $\varphi \in C(\mathbb R,\mathcal X)$. We employ
$\mathfrak N_{\varphi}$ (respectively, $\mathfrak M_{\varphi}$)
to denote the family of all sequences
$\{t_n\}\subset \mathbb R$ such that $\varphi^{t_n}
\to \varphi$ (respectively, $\{\varphi^{t_n}\}$ converges) in
$C(\mathbb R,\mathcal X)$ as $n\to \infty$.
%We write $\mathfrak N_{\varphi}^{u}$ (respectively,
%$\mathfrak M_{\varphi}^{u}$) to mean the set of sequences
%$\{t_n\}\subset\R$ such that $\varphi^{t_n}$ converges to
%$\varphi$ (respectively,  $\varphi^{t_n}$ converges) uniformly
%with respect to $t\in\mathbb R$ as $n\to \infty$.

\begin{definition}\label{defPR}\rm (\cite{Sch68,Sch72,Sch85})
A function $\varphi \in C(\mathbb R,\mathcal X)$ is called  {\em
pseudo-recurrent}
%\cite{Sch72}
if for any $\varepsilon >0$ and
$l\in\mathbb R$ there exists $L\ge l$ such that for any $\tau_0\in
\mathbb R$ we can find a number $\tau \in [l,L]$ satisfying
$$
\sup\limits_{|t|\le 1/\varepsilon}\rho(\varphi(t+\tau_0
+\tau),\varphi(t+\tau_0))\le \varepsilon.
$$
\end{definition}

\begin{remark}\label{remPR} \rm (\cite{Sch68,Sch72,Sch85,sib})
\begin{enumerate}
\item Every Birkhoff recurrent function is pseudo-recurrent, but
      not vice versa.
\item Suppose that $\varphi \in C(\mathbb R,\mathcal X)$ is
      pseudo-recurrent, then every function $\psi\in H(\varphi)$ is
      pseudo-recurrent.
\item Suppose that $\varphi \in C(\mathbb R,\mathcal X)$  is Lagrange
      stable and  every function $\psi\in H(\varphi)$ is Poisson stable,
      then $\varphi$ is pseudo-recurrent.
\end{enumerate}
\end{remark}

Finally, we remark that a Lagrange stable function is not Poisson
stable in general, but all other types of functions introduced
above are Poisson stable.

\begin{definition}\label{defF1} \rm
\begin{enumerate}
\item We say that a function $\varphi\in C(\R,\mathcal X)$ {\em possesses
  the property A} if the motion $\sigma(\cdot,\varphi)$ through $\varphi$
  with respect to the Bebutov dynamical system
  $(C(\R\times\mathcal X),\R,\sigma)$ possesses the property A.
\item Similarly, we say that $F\in BUC(\R\times\mathcal X,\mathcal Y)$
  {\em possesses the property A in $t\in\R$ uniformly with respect to
  $x$ on each bounded subset $Q\subset\mathcal X$}, if the motion
  $\sigma(\cdot,F):\R\rightarrow BUC(\R\times\mathcal X,\mathcal Y)$
  through $F$ with respect to the Bebutov dynamical system
  $\left(BUC(\R\times\mathcal X,\mathcal Y),\R,\sigma\right)$
  possesses the property A.
\end{enumerate}
Here the property A may be stationary, periodic, Bohr/Levitan almost
periodic, etc.
\end{definition}

\subsection{Shcherbakov's comparability method by character of recurrence}

\begin{definition} \rm
A function $\varphi \in C(\mathbb R,\mathcal X)$ is called
{\em comparable}
%\cite{Sch72,scher75,Sch85}
(respectively, {\em strongly comparable})
%\cite{Che_2009})
{\em by character of recurrence} with $\psi\in C(\mathbb R,\mathcal Y)$
provided $\mathfrak N_{\psi}\subseteq \mathfrak N_{\varphi}$
(respectively, $\mathfrak M_{\psi}\subseteq \mathfrak M_{\varphi}$).
\end{definition}

\begin{theorem}\label{th1}{\rm(\cite[ChII]{Sch72}, \cite{scher75})}
\begin{enumerate}
\item $\mathfrak M_{\psi}\subseteq \mathfrak M_{\varphi}$ implies
$\mathfrak N_{\psi}\subseteq \mathfrak N_{\varphi}$, and hence
      strong comparability implies comparability.

\item  Assume that $\varphi \in C(\mathbb R,\mathcal X)$ is comparable by
       character of recurrence with $\psi\in C(\mathbb R,\mathcal Y)$.
       If the function $\psi$ is stationary (respectively, $T$-periodic,
       Levitan almost periodic, almost recurrent, Poisson stable),
       then so is $\varphi$.
\item  Assume that $\varphi \in C(\mathbb R,\mathcal X)$ is strongly
       comparable by character of recurrence with
       $\psi\in C(\mathbb R,\mathcal Y)$.
       If the function $\psi$ is quasi-periodic with the spectrum of
       frequencies $\nu_1,\nu_2,\dots,\nu_k$ (respectively, almost
       periodic, almost automorphic, Birkhoff recurrent, Lagrange
       stable), then so is $\varphi$.
\item Assume that $\varphi \in C(\mathbb R,\mathcal X)$ is strongly
      comparable by character of recurrence with
      $\psi\in C(\mathbb R,\mathcal Y)$. And suppose further that
      $\psi$ is Lagrange stable. If $\psi$ is pseudo-periodic
      (respectively, pseudo-recurrent), then so is $\varphi$.
\end{enumerate}
\end{theorem}

\subsection{Variational approach}\label{varapproach}

Recall that $H$ is a separable Hilbert space with norm
$\|\cdot\|_{H}$ and inner product $\langle~,~\rangle_{H}$,
and that $H^{*}$ is the dual space of $H$.
Let $(V,\|\cdot\|_{V})$ be a reflexive Banach space
such that $V\subset H$ continuously and densely.
So we have $H^{*}\subset V^{*}$ continuously and densely.
Identifying $H$ with its dual $H^{*}$ via the Riesz
isomorphism, then we have
$$V\subset H\subset V^{*}$$
continuously and densely. We write $_{V^{*}}\langle~,~\rangle_{V}$
to denote the pairing between $V^{*}$ and $V$. It follows that
$$ _{V^{*}} \langle h,v\rangle_{V}=\langle h,v\rangle_{H}$$
for all $h\in H$, $v\in V$. $(V,H,V^{*})$ is called
{\em Gelfand triple}. Since $H\subset V^{*}$ continuously
and densely, we deduce that $V^{*}$ is separable, hence so is $V$.

Assume that $(V_{1},\|\cdot\|_{V_{1}})$ and
$(V_{2},\|\cdot\|_{V_{2}})$ are reflexive Banach spaces
and embedded in $H$ continuously and densely.
Then we get two triples:
\[
V_{1}\subset H\simeq H^{*}\subset V_{1}^{*}\quad {\rm{and}}\quad
V_{2}\subset H\simeq H^{*}\subset V_{2}^{*}.
\]
We define the norm $\|v\|_{V}:=\|v\|_{V_{1}}+\|v\|_{V_{2}}$
on the space $V:=V_{1}\cap V_{2}$.
Note that $(V,\|\cdot\|_{V})$ is also a Banach space.
Since $V_{1}^{*}$ and $V_{2}^{*}$ can be thought as subspaces
of $V^{*}$, we get a Banach space
$W:=V_{1}^{*}+V_{2}^{*}\subset V^{*}$ with norm
\[
\|f\|_{W}:=\inf\left\{\|f_{1}\|_{V_{1}^{*}}+\|f_{2}\|_{V_{2}^{*}}:
f=f_{1}+f_{2}, ~f_{i}\in V_{i}^{*}, ~i=1,2\right\}.
\]
Similarly, we write $_{V_{i}^{*}}\langle~,~\rangle_{V_{i}}$
to denote the pairing between $V_{i}^{*}$ and $V_{i}$, $i=1,2$.
Then, for all $v\in V$ and $f=f_{1}+f_{2}\in W\subset V^{*}$
we have
\[
_{V^{*}}\langle f,v \rangle_{V}= ~_{V_{1}^{*}}\langle f_{1},
v \rangle_{V_{1}}+~_{V_{2}^{*}}\langle f_{2},v \rangle_{V_{2}}.
\]
Note carefully that if $f\in H$ and $v\in V$, then we obtain
\[
_{V^{*}}\langle f,v \rangle_{V}= ~_{V_{1}^{*}}\langle f,
v \rangle_{V_{1}}=~_{V_{2}^{*}}\langle f,v \rangle_{V_{2}}
=\langle f,v \rangle_{H}.
\]

We write $Pr(H)$ to mean the set of all Borel probability measures
on $H$. Denote by $C_{b}(H)$ the space of all continuous functions
$\varphi:H\rightarrow \R$ for which the norm
$\|\varphi\|_{\infty}:=\sup\limits_{x\in H}|\varphi(x)|$ is finite.
Let $\{\mu_{n}\}:=\{\mu_{n}\}_{n=1}^{\infty}\subset Pr(H)$ and
$\mu\in Pr(H)$. We say $\mu_{n}$ {\em converges weakly} to $\mu$ in
$Pr(H)$ provided $\int \varphi\d\mu_{n}$ converges to
$\int \varphi\d\mu$ for all $\varphi\in C_{b}(H)$.
Let $\varphi\in C_{b}(H)$ be Lipschitz continuous, we define
\begin{equation*}
  \|\varphi\|_{BL}:=Lip(\varphi)+\|\varphi\|_{\infty},
\end{equation*}
where $Lip(\varphi)=\sup\limits_{x\neq y}
\frac{|\varphi(x)-\varphi(y)|}{\|x-y\|_{H}}$. Then $Pr(H)$ is a
separable complete metric space with the following
bounded Lipschitz distance (also called Fortet-Mourier distance)
\[
d_{BL}(\mu,\nu):=\sup\left\{\left|\int \varphi\d\mu
-\int \varphi\d\nu\right|:\|\varphi\|_{BL}\leq1\right\}
\]
for all $\mu,~\nu\in Pr(H)$. It is well known that $d_{BL}$ generates
the weak topology on $Pr(H)$, i.e. $\mu_{n}\rightarrow\mu$ weakly in
$Pr(H)$ if and only if
$d_{BL}(\mu_{n},\mu)\rightarrow0$ as $n\rightarrow\infty$.
See Chapter 11 in \cite{Dudley} for this metric $d_{BL}$
(denoted by $\beta$ there) and its related properties.

We assume in the following exposition that $(\Omega,\mathcal{F},\P)$
is a complete probability space. The space $L^{2}(\Omega,\P;H)$
consists of all $H$-valued random variables $\zeta$ such that
$E\|\zeta\|_{H}^{2}=\int_{\Omega}\|\zeta\|_{H}^{2} \d P<\infty$.
An $H$-valued stochastic process
$X=X(t),~t\in\R$ is called {\em $L^{2}$-bounded} provided
$\sup\limits_{t\in\R}E\|X(t)\|_{H}^{2}<\infty$.
Throughout the paper, we denote by $\L(\zeta)\in Pr(H)$ the law or
distribution of $H$-valued random variable $\zeta$. A sequence of
$H$-valued continuous stochastic processes $\{X_{n}\}$ is said to
{\em converge in distribution to $X$ (on $C(\R,H)$)} provided
$\L(X_{n})$ weakly converges to $\L(X)$ in $Pr(C(\R,H))$,
where $\L(X)$ is the law or distribution of $X$ on $C(\R,H)$.
If $d_{BL} (\L (X_n(t)), \L(X(t)))\to 0$ as $n\to\infty$ for each
$t\in\R$, we simply say that $X_n$
{\em converges in distribution to $X$ on $H$}.

\section{Compatible solutions}
Let $W(t)$, $t\in\R$ be a two-sided cylindrical $Q$-Wiener
process with $Q=I$ on a separable Hilbert space
$(U,\langle~,~\rangle_{U})$ with respect to a complete
filtered probability space
$(\Omega,\mathcal{F},\mathcal{F}_{t},\P)$. Denote by
$L_{2}(U,H)$ the space of all Hilbert-Schmidt operators
from $U$ into $H$.

In this section, coefficient $F$ in \eqref{origeq} need not to be
Lipschitz. Therefore, instead of explicitly writting
$F$ in \eqref{origeq}, we consider the following
stochastic partial differential equation on $H$
\begin{equation}\label{eqSPDE1}
\ \d X(t)=A(t,X(t))\d t+G(t,X(t))\d W(t),
\end{equation}
where $A(t,x)=A_{1}(x)+A_{2}(t,x)$,
$A_{i}:\R\times V_{i}\rightarrow V_{i}^{*}$, $i=1,2$ and
$G:\R\times V\rightarrow L_{2}(U,H)$.

Consider equation \eqref{eqSPDE1}.
Let us introduce the following conditions.
\begin{enumerate}

\item[\textbf{(H1)}](Continuity) For all $u$, $v$, $w \in V$
and $t\in \mathbb{R}$ the map
\begin{equation}\label{hcdef}
\mathbb{R}\ni \theta \mapsto~_{V_{1}^{*}}\langle
A_{1}(u+\theta v),w\rangle_{V_{1}}
\end{equation}
is continuous. $A_{2}:\R\times V_{2}\rightarrow V_{2}^{*}$
and $G:\R\times V\rightarrow L_{2}(U,H)$ are continuous.
Here $A_{1}$ is called {\em hemicontinuity} provided
\eqref{hcdef} hold.
\item[\textbf{(H2)}] (Strong monotonicity) There exist  constants
    $\lambda\geq0$, $r>2$ and $\lambda'\geq0$ such that for
    all $u$, $v\in V$, $t\in \mathbb{R}$
\begin{equation*}
  ~_{V^{*}}\langle A(t,u)-A(t,v),u-v\rangle_{V}
  \leq-\lambda\|u-v\|^{2}_{H}-\lambda'\|u-v\|^{r}_{H}
\end{equation*}
  and
  \begin{equation*}
  \|G(t,u)-G(t,v)\|^{2}_{L_{2}(U,H)}\leq L_{G}^{2}\|u-v\|^{2}_{H}.
  \end{equation*}
\item[\textbf{(H3)}](Coercivity) There  exist constants
$\alpha_{1}, \alpha_{2} \in(1,\infty)$,
$c_{1}\in \mathbb{R}$, $c_{2}, c_{2}'\in(0,\infty)$
and $M_{0}\in(0,\infty)$ such that for all
$v\in V$, $t\in\mathbb{R}$
 \begin{equation*}
   ~_{V^{*}}\langle A(t,v),v\rangle_{V}\leq c_{1}\|v\|^{2}_{H}
   -c_{2}\|v\|^{\alpha_{1}}_{V_{1}}-c_{2}'\|v\|^{\alpha_{2}}_{V_{2}}
   +M_{0}.
 \end{equation*}
\item[\textbf{(H4)}](Boundedness) There exist constants $c_{3},
    c_{3}'\in(0,\infty)$ such that for all
    $v\in V$, $t\in\mathbb{R}$
    $$\|A_{1}(v)\|_{V_{1}^{*}}\leq c_{3}\|v\|^{\alpha_{1}-1}_{V_{1}}+M_{0},\quad
    \|A_{2}(t,v)\|_{V_{2}^{*}}\leq c_{3}'\|v\|^{\alpha_{2}-1}_{V_{2}}+M_{0}$$
    and
 \begin{equation*}
   \|G(t,0)\|_{L_{2}(U,H)}\leq M_{0},
 \end{equation*}
    where $\alpha_{i}$ and $M_{0}$ are as in (H3).
\item[\textbf{(H5)}] $A_{2}$ and $G$ are continuous in
        $t\in\R$ uniformly with respect to $v$ on each
        bounded subset $Q\subset V$.
\end{enumerate}

\begin{remark}\rm
Since we consider compatible solutions (see Definition
\ref{compatible}) by the method of dynamical systems,
we assume that $A_{2}$ and $G$ satisfy (H1) and (H5)
that are different from the usual situation (i.e. we request stronger continuity conditions here). Under
conditions of (H1)--(H2) and (H4)--(H5), $(H(A_{2}),\R,\sigma)$ and
$(H(G),\R,\sigma)$ are dynamical systems, where
$\sigma:\R\times H(A_{2})\rightarrow H(A_{2}),
~(\tau,\tilde{A}_{2})\mapsto \tilde{A}_{2}^{\tau}$ and similarly for the action $\sigma$ on $H(G)$.
Note that we only need hemicontinuity of $A_{2}$ and do not need (H5), as usual, when we consider estimates of solutions,
such as Lemmas \ref{pmes}--\ref{pe}, Theorems \ref{Boundedth} and \ref{gasms}, Proposition \ref{tightprop},
Lemma \ref{conlemma}.
\end{remark}

\begin{definition}[see, e.g. \cite{PR2007, Zhang2009}]
\label{vsolution}\rm
We say  continuous $H$-valued $(\mathcal{F}_{t})$-adapted
process $X(t)$, $t\in[0,T]$ is a
{\em solution} to equation \eqref{eqSPDE1},
if $X\in \cap_{i=1,2}L^{\alpha_{i}}([0,T]\times \Omega,\d t\otimes
\P;V_{i})\cap L^{2}([0,T]\times \Omega, \d t\otimes \P;H)$
with $\alpha_{i}$ as in (H3) and $\P$-a.s.
\begin{equation}\label{defsol}
\ X(t)=X(s)+\int^{t}_{s}A(\sigma,X(\sigma))\d\sigma
+\int^{t}_{s}G(\sigma,X(\sigma))\d W(\sigma),
\quad 0\leq s\leq t\leq T.
\end{equation}
Moreover, we say $X(t),t\in\R$ is a {\em solution} to equation \eqref{eqSPDE1} provided \eqref{defsol} holds for all $t\geq s$ and each $s\in\R$.
\end{definition}

Fix $s\in\R$. Under conditions (H1)--(H4), for any $\zeta\in L^{2}(\Omega, \mathcal{F}_{s}, \P;H)$ and $T>0$ there exists a unique solution $X(t,s,\zeta),s\leq t\leq s+T$ to \eqref{eqSPDE1} with initial condition $X(s,s,\zeta)=\zeta$
(see, e.g. \cite{PR2007}).
In this paper, we write $C_{\alpha}$ to mean some positive
constant which depends on $\alpha$. Here $\alpha$ is one or more
than one parameter and $C_{\alpha}$ may change from line to line.
Now we discuss the $L^{2}$-bounded solution to equation
\eqref{eqSPDE1} by employing the classical pullback attraction
method in random and non-autonomous dynamics
(see, e.g. \cite{CML2020, DT1995} etc).
For this we need three lemmas.

\begin{lemma}\label{pmes}
Assume that {\rm{(H1)}}--{\rm{(H4)}} hold.
Let $\zeta_{s}\in L^{2}(\Omega, \mathcal{F}_{s}, \P;H)$
and $X(t,s,\zeta_{s})$, $t\geq s$ be the solution to
the following Cauchy problem
\begin{equation*}
  \left\{
   \begin{aligned}
   &\ \d X(t)=A(t,X(t))\d t+G(t,X(t))\d W(t)\\
  &\ X(s)=\zeta_{s}.
   \end{aligned}
   \right.
  \end{equation*}
\begin{enumerate}
\item If $2\lambda>L_{G}^{2}$, let $\eta\in(0,2\lambda-L_{G}^{2})$.
Then there exist constants $1\leq p<\frac{\eta}{2L_{G}^{2}}+1$
and $\kappa,M_{1}>0$ such that
\begin{equation}\label{ineq2}
E\|X(t,s,\zeta_{s})\|^{2p}_{H}\leq {\rm{e}}^{-\kappa(t-s)}
E\|\zeta_{s}\|^{2p}_{H}+M_{1},
\end{equation}
where $M_{1}$ depends only on $\eta,c_{2},c_{3},c_{2}',
c_{3}',\alpha_{1},\alpha_{2},\kappa,p,r$.
\item If $\lambda'>0$ then estimate \eqref{ineq2} hold
for any $p\in[1,+\infty)$ and $\kappa>0$.
\end{enumerate}
\end{lemma}
\begin{proof}
By (H2)--(H4) and Young's inequality, we have
\begin{align}\label{pmeseq3}
&
  2_{V^{*}}\langle A(t,u),u\rangle_{V}
  +\|G(t,u)\|^{2}_{L_{2}(U,H)}\\\nonumber
&
  \leq
  \begin{cases}
    -\eta\|u\|_{H}^{2}+C_{\alpha_{1},\alpha_{2},c_{2},c_{2}',M_{0}},
    &\text{if  $2\lambda>L_{G}^{2}$}\\
    -\lambda'\|u\|_{H}^{r}+(c_{1}+2L_{G}^{2}-\lambda)\|u\|_{H}^{2}
    +C_{\alpha_{1},\alpha_{2},c_{2},c_{2}',M_{0}},
    &\text{if $\lambda'>0$}.
  \end{cases}
\end{align}
Given $\kappa>0$ and $p\geq1$,
in view of It\^o's formula (see, e.g. \cite[Theorem 4.2.5]{PR2007}), we get
\begin{align}\label{pmeseq4}
&
E\left({\rm{e}}^{\kappa(t-s)}\|X(t,s,\zeta_{s})\|_{H}^{2p}\right)
\\\nonumber
&
=E\|\zeta_{s}\|_{H}^{2p}+\int_{s}^{t}\kappa{\rm{e}}^{\kappa(\sigma-s)}
E\|X(\sigma,s,\zeta_{s})\|_{H}^{2p}\d\sigma\\\nonumber
&\quad
+pE\int_{s}^{t}\|X(\sigma,s,\zeta_{s})\|_{H}^{2p-2}
{\rm{e}}^{\kappa(\sigma-s)}\bigg(2_{V^{*}}\langle
A(\sigma,X(\sigma,s,\zeta_{s})),X(\sigma,s,\zeta_{s})\rangle_{V}
\\\nonumber
&\qquad
+\|G(\sigma,X(\sigma,s,\zeta_{s}))\|^{2}_{L_{2}(U,H)}\bigg)\d\sigma
\\\nonumber
&\quad
+2p(p-1)E\int_{s}^{t}{\rm{e}}^{\kappa(\sigma-s)}
\|X(\sigma,s,\zeta_{s})\|_{H}^{2p-4}\|\left(
G(\sigma,X(\sigma,s,\zeta_{s}))\right)^{*}
X(\sigma,s,\zeta_{s})\|_{U}^{2}\d\sigma.
\end{align}

If  $2\lambda>L_{G}^{2}$, according to
\eqref{pmeseq3}--\eqref{pmeseq4}, (H2) and Young's inequality,
we obtain
\begin{align*}
&
E\left({\rm{e}}^{\kappa(t-s)}\|X(t,s,\zeta_{s})\|_{H}^{2p}\right) \\
&
\leq E\|\zeta_{s}\|_{H}^{2p}
+\int_{s}^{t}\kappa{\rm{e}}^{\kappa(\sigma-s)}
E\|X(\sigma,s,\zeta_{s})\|_{H}^{2p}\d\sigma\\
&\quad
+pE\int_{s}^{t}\|X(\sigma,s,\zeta_{s})\|_{H}^{2p-2}
{\rm{e}}^{\kappa(\sigma-s)}\left(-\eta\|X(\sigma,s,\zeta_{s})\|_{H}^{2}
+C_{\alpha_{1},\alpha_{2},c_{2},c_{2}',M_{0}}\right)\d\sigma\\
&\quad
+2p(p-1)E\int_{s}^{t}{\rm{e}}^{\kappa(\sigma-s)}
\|X(\sigma,s,\zeta_{s})\|_{H}^{2p-2}
\left(L_{G}^{2}\|X(\sigma,s,\zeta_{s})\|_{H}^{2}+\varepsilon
L_{G}^{2}\|X(\sigma,s,\zeta_{s})\|_{H}^{2}
+C_{\varepsilon,M_{0}}\right)\d\sigma\\
&
\leq E\|\zeta_{s}\|_{H}^{2p}+E\int_{s}^{t}{\rm{e}}^{\kappa(\sigma-s)}
\left(\kappa-\eta p+2p(p-1)\varepsilon L_{G}^{2}+2p(p-1) L_{G}^{2}\right)\|X(\sigma,s,\zeta_{s})\|_{H}^{2p}\d\sigma\\
&\quad
+E\int_{s}^{t}{\rm{e}}^{\kappa(\sigma-s)}\left(
C_{\alpha_{1},\alpha_{2},c_{2},c_{2}',M_{0},\varepsilon}p
+C_{\varepsilon}p(p-1)\right)\|X(\sigma,s,\zeta_{s})\|_{H}^{2p-2}
\d\sigma.
\end{align*}
Let $1\leq p<\frac{\eta}{2L_{G}^{2}}+1$ and $\kappa\in(0,\eta p
-2p(p-1)L_{G}^{2})$. Employing Young's inequality and taking
$\varepsilon$ small enough, we obtain
\begin{equation*}
E\|X(t,s,\zeta_{s})\|^{2p}_{H}\leq {\rm{e}}^{-\kappa(t-s)}
E\|\zeta_{s}\|^{2p}_{H}+M_{1}.
\end{equation*}

If $\lambda'>0$, for any $\kappa>0$, by
\eqref{pmeseq3}--\eqref{pmeseq4}, (H2) and Young's inequality we have
\begin{align*}
&
E\left({\rm{e}}^{\kappa(t-s)}\|X(t,s,\zeta_{s})\|_{H}^{2p}\right) \\
&
\leq E\|\zeta_{s}\|_{H}^{2p}
+\int_{s}^{t}\kappa{\rm{e}}^{\kappa(\sigma-s)}
E\|X(\sigma,s,\zeta_{s})\|_{H}^{2p}\d\sigma
+pE\int_{s}^{t}\|X(\sigma,s,\zeta_{s})\|_{H}^{2p-2}
{\rm{e}}^{\kappa(\sigma-s)}\bigg(
C_{\alpha_{1},\alpha_{2},c_{2},c_{2}',M_{0}}\\
&\qquad
+\left(c_{1}+2L_{G}^{2}-\lambda\right)
\|X(\sigma,s,\zeta_{s})\|_{H}^{2}
-\lambda'\|X(\sigma,s,\zeta_{s})\|_{H}^{r}\bigg)\d\sigma\\
&\quad
+2p(p-1)E\int_{s}^{t}{\rm{e}}^{\kappa(\sigma-s)}
\|X(\sigma,s,\zeta_{s})\|_{H}^{2p-2}
\left(2L_{G}^{2}\|X(\sigma,s,\zeta_{s})\|_{H}^{2}+2M_{0}^{2}\right)
\d\sigma\\
&
\leq E\|\zeta_{s}\|_{H}^{2p}
+E\int_{s}^{t}{\rm{e}}^{\kappa(\sigma-s)}\bigg[-\lambda'p
\|X(\sigma,s,\zeta_{s})\|_{H}^{r+2p-2}+\Big(\kappa+p\left(c_{1}
+2L_{G}^{2}-\lambda\right)\\
&\qquad
+4p(p-1)L_{G}^{2}\Big)
\|X(\sigma,s,\zeta_{s})\|_{H}^{2p}
+\left(pC_{\alpha_{1},\alpha_{2},c_{2},c_{2}',M_{0}}
+4p(p-1)M_{0}^{2}\right)\|X(\sigma,s,\zeta_{s})\|_{H}^{2p-2}\bigg]
\d\sigma\\
&
\leq E\|\zeta_{s}\|_{H}^{2p}
+E\int_{s}^{t}{\rm{e}}^{\kappa(\sigma-s)}
C_{\alpha_{1},\alpha_{2},c_{2},c_{2}',M_{0},p,r}\d\sigma.
\end{align*}
Therefore, for any $p\in[1,+\infty)$ and $\kappa>0$
\begin{equation*}
E\|X(t,s,\zeta_{s})\|^{2p}_{H}\leq {\rm{e}}^{-\kappa(t-s)}
E\|\zeta_{s}\|^{2p}_{H}+M_{1}.
\end{equation*}
\end{proof}

\begin{lemma}\label{ivcont}
Consider equation \eqref{eqSPDE1}.
Assume that $2\lambda-L_{G}^{2}\geq0$ and
{\rm{(H1)}}--{\rm{(H4)}} hold.
Let $X$ and $Y$ be two solutions of equation \eqref{eqSPDE1}.
If $\lambda'>0$ or $2\lambda>L_{G}^{2}$,
then for any $s\leq t$ we have
\begin{align}\label{ANSCP}
   & E\|X(t,s,X(s))-Y(t,s,Y(s))\|^{2}_{H} \\\nonumber
   & \leq \begin{cases}
    E\|X(s)-Y(s)\|_{H}^{2}\wedge \left\{\lambda'(r-2)
   (t-s)\right\}^{-\frac{2}{r-2}}, &\text{if $\lambda'>0$}\\
	{\rm{e}}^{-(2\lambda-L_{G}^{2})(t-s)}E\|X(s)-Y(s)\|_{H}^{2},
   &\text{if $2\lambda>L_{G}^{2}$.}
  \end{cases}
\end{align}

In particular, for any $t\in\R$ there exists some random
variable $X(t)$ such that
\begin{equation}\label{L2lim2}
  X(t,-n,0)\rightarrow X(t)\quad {\rm{in}}~ L^{2}(\Omega,\P;H)
  ~ {\rm{as}}~ n\rightarrow\infty.
\end{equation}
\end{lemma}
\begin{proof}
If $2\lambda>L_{G}^{2}$, by It\^o's formula and (H2)  we get
\begin{align*}
&
E\|X(t,s,X(s))-Y(t,s,Y(s))\|_{H}^{2}\\
&
\leq E\|X(s)-Y(s)\|_{H}^{2}+E\int_{s}^{t}\left(-2\lambda+L_{G}^{2}\right)
\|X(\sigma,s,X(s))-Y(\sigma,s,Y(s))\|_{H}^{2}\d\sigma.
\end{align*}
It follows from Gronwall's lemma that
\begin{equation*}
E\|X(t,s,X(s))-Y(t,s,Y(s))\|_{H}^{2}
\leq {\rm{e}}^{-(2\lambda-L_{G}^{2})(t-s)}E\|X(s)-Y(s)\|_{H}^{2}.
\end{equation*}

If $\lambda'>0$ and $2\lambda\geq L_{G}^{2}$, in view of
It\^o's formula and (H2), we have
\begin{align*}
&
E\|X(t,s,X(s))-Y(t,s,Y(s))\|_{H}^{2}\\
&
\leq E\|X(s)-Y(s)\|_{H}^{2}+E\int_{s}^{t}-2\lambda'
\|X(\sigma,s,X(s))-Y(\sigma,s,Y(s))\|_{H}^{r}\d\sigma\\
&
\leq E\|X(s)-Y(s)\|_{H}^{2}-2\lambda'\int_{s}^{t}\left(E
\|X(\sigma,s,X(s))-Y(\sigma,s,Y(s))\|_{H}^{2}\right)^{\frac{r}{2}}
\d\sigma.
\end{align*}
Employing comparison theorem, we obtain
\begin{equation*}
E\|X(t,s,X(s))-Y(t,s,Y(s))\|^{2}_{H}
\leq E\|X(s)-Y(s)\|_{H}^{2}\wedge \left\{\lambda'(r-2)
   (t-s)\right\}^{-\frac{2}{r-2}}.
\end{equation*}
\end{proof}

\begin{lemma}\label{pe}
Suppose that {\rm{(H1)--(H4)}} hold. Let $X(t,s,\zeta_{s})$ be a
solution to equation \eqref{eqSPDE1} with initial value
$X(s,s,\zeta_{s})=\zeta_{s}$. We have
\begin{align}\label{peq}
&
E\left(\sup_{t\in[s,s+T]}\|X(t,s,\zeta_{s})\|_{H}^{2}\right)
+E\int_{s}^{s+T}\left(\|X(t,s,\zeta_{s})\|_{V_{1}}^{\alpha_{1}}
+\|X(t,s,\zeta_{s})\|_{V_{2}}^{\alpha_{2}}\right)\d t\\\nonumber
&\quad
+E\int_{s}^{s+T}\left(\|A_{1}(X(t,s,\zeta_{s}))\|_{V_{1}^{*}}
^{\frac{\alpha_{1}}{\alpha_{1}-1}}
+\|A_{2}(t,X(t,s,\zeta_{s}))\|_{V_{2}^{*}}^{\frac{\alpha_{2}}
{\alpha_{2}-1}}\right)\d t\\\nonumber
&
\leq C_{c_{1},L_{G},M_{0},T}\left(1
+E\|\zeta_{s}\|_{H}^{2}\right)
\end{align}
for any $s\in\R, T>0$.
\end{lemma}

\begin{proof}
By It\^o's formula, (H2) and (H3), we have
\begin{align}\label{peq1}
&
\|X(t,s,\zeta_{s})\|_{H}^{2}\\\nonumber
&
=\|\zeta_{s}\|_{H}^{2}+\int_{s}^{t}\left(2_{V^{*}}\langle
A(\sigma,X(\sigma,s,\zeta_{s})),X(\sigma,s,\zeta_{s})\rangle_{V}+
\|G(\sigma,X(\sigma,s,\zeta_{s}))\|_{L_{2}(U,H)}^{2}\right)\d\sigma\\\nonumber
&\quad
+2\int_{s}^{t}\langle X(\sigma,s,\zeta_{s}),
G(\sigma,X(\sigma,s,\zeta_{s}))\d W(\sigma)\rangle_{H}\\\nonumber
&
\leq\|\zeta_{s}\|_{H}^{2}+\int_{s}^{t}\Big(2c_{1}
\|X(\sigma,s,\zeta_{s})\|_{H}^{2}
-2c_{2}\|X(\sigma,s,\zeta_{s})\|_{V_{1}}^{\alpha_{1}}
-2c'_{2}\|X(\sigma,s,\zeta_{s})\|_{V_{2}}^{\alpha_{2}}+2M_{0}\\\nonumber
&\qquad
+2L_{G}^{2}\|X(\sigma,s,\zeta_{s})\|_{H}^{2}+2M_{0}^{2}\Big)\d\sigma
+2\int_{s}^{t}\langle X(\sigma,s,\zeta_{s}),
G(\sigma,X(\sigma,s,\zeta_{s}))\d W(\sigma)\rangle_{H}.\\\nonumber
\end{align}
Dropping negative terms on the right of the above inequality,
according to Burkholder-Davis-Gundy inequality
(see, e.g. \cite{PR2007}) and Young's inequality, we get
\begin{align}\label{peq2}
&
E\sup_{t\in[s,s+T]}\|X(t,s,\zeta_{s})\|_{H}^{2}\\\nonumber
&
\leq E\|\zeta_{s}\|_{H}^{2}+E\int_{s}^{s+T}\left(\left(2c_{1}
+2L_{G}^{2}\right)\|X(\sigma,s,\zeta_{s})\|_{H}^{2}+2M_{0}^{2}
+2M_{0}\right)\d\sigma\\\nonumber
&\quad
+6E\left(\int_{s}^{s+T}
\|G(\sigma,X(\sigma,s,\zeta_{s}))\|_{L_{2}(U,H)}^{2}
\|X(\sigma,s,\zeta_{s})\|_{H}^{2}\d\sigma\right)^{\frac{1}{2}}\\\nonumber
&
\leq E\|\zeta_{s}\|_{H}^{2}+E\int_{s}^{s+T}\left(C_{c_{1},L_{G}}
\|X(\sigma,s,\zeta_{s})\|_{H}^{2}+C_{M_{0}}\right)\d\sigma\\\nonumber
&\quad
+\frac{1}{2}E\sup_{t\in[s,s+T]}\|X(t,s,\zeta_{s})\|_{H}^{2}.
\end{align}
By Gronwall's lemma, we obtain
\begin{align}\label{peq4}
E\sup_{t\in[s,s+T]}\|X(t,s,\zeta_{s})\|_{H}^{2}
\leq C_{c_{1},L_{G},T,M_{0}}\left(1+E\|\zeta_{s}\|_{H}^{2}\right).
\end{align}

Take expectations on both sides of \eqref{peq1} and let $t=s+T$,
then by \eqref{peq4} we have
\begin{align*}
E\int_{s}^{s+T}\left(\|X(t,s,\zeta_{s})\|_{V_{1}}^{\alpha_{1}}
+\|X(t,s,\zeta_{s})\|_{V_{2}}^{\alpha_{2}}\right)\d t
\leq C_{c_{1},L_{G},T,M_{0}}\left(1+E\|\zeta_{s}\|_{H}^{2}\right).
\end{align*}
In view of (H4), we complete the proof.
\end{proof}

\begin{theorem}\label{Boundedth}
Consider equation \eqref{eqSPDE1}.
Suppose that $2\lambda-L_{G}^{2}\geq0$ and
{\rm{(H1)}}--{\rm{(H4)}} hold. If $\lambda'>0$
or $2\lambda>L_{G}^{2}$,
then there exists a unique $L^{2}$-bounded
continuous H-valued solution $X(t)$, $t\in\mathbb{R}$ to
equation \eqref{eqSPDE1}. Moreover, the mapping
$\widehat{\mu}:\mathbb{R}\rightarrow Pr(H)$,
defined by $\widehat{\mu}(t):=\P\circ[X(t)]^{-1}$,
is unique with the following properties:
\begin{enumerate}
\item $L^{2}$-boundedness: ~$\sup\limits_{t\in \mathbb{R}}
\int_{H}\|x\|^{2}_{H}\widehat{\mu}(t)(\d x)<+\infty$;
\item Flow property: ~$\mu(t,s,\widehat{\mu}(s))
=\widehat{\mu}(t)$ for all $t\geq s$.
\end{enumerate}
Here $\mu(t,s,\mu_{0})$ denotes the distribution of
$X(t,s,\zeta_{s})$ on $H$, with $\mu_{0}=\P\circ\zeta^{-1}_{s}$.
\end{theorem}
\begin{proof}
For any fixed interval $[a,b]\subset \mathbb{R}$, we denote
$$J:=L^{2}([a,b]\times\Omega,\d t\otimes\P;L_{2}(U,H)),
\quad K_{i}:=L^{\alpha_{i}}([a,b]\times\Omega,\d t\otimes\P;V_{i}),$$
$$K_{i}^{*}:=L^{\frac{\alpha_{i}}{\alpha_{i}-1}}([a,b]\times\Omega,\d t\otimes\P;V_{i}^{*}),~i=1,2.$$
According to the reflexivity of $K_{i}$, $i=1,2$, \eqref{ANSCP} and
\eqref{peq}, we may assume, going if necessary to a subsequence, that
\begin{enumerate}
\item[(1)] $X(\cdot,-n,0)\rightarrow X(\cdot)$ in $L^{2}([a,b]\times\Omega,\d t\otimes \P;H)$ and
      $X(\cdot,-n,0)\rightarrow X(\cdot) $ weakly in $K_{1}$ and $K_{2}$;
\item[(2)] $A_{i}(\cdot,X(\cdot,-n,0))\rightarrow Y_{i}(\cdot)$ weakly in $K_{i}^{*}$, $i=1,2$;
\item[(3)] $G(\cdot,X(\cdot,-n,0))\rightarrow Z(\cdot)$ weakly in $J$ and hence
      $$\int^{t}_{a}G(\sigma,X(\sigma,-n,0))\d W(\sigma)\rightarrow\int^{t}_{a}Z(\sigma)\d W(\sigma)$$
      weakly* in $L^{\infty}([a,b],\d t;L^{2}(\Omega,\P;H))$.
\end{enumerate}

Thus for all $v\in V$, $\varphi\in L^{\infty}([a,b]\times\Omega)$ by Fubini's theorem we get
\begin{align*}
&
E\int^{b}_{a}~_{V^{*}}\langle X(t),\varphi(t)v\rangle_{V}\d t\\
&
=\lim_{n\rightarrow\infty}E\int^{b}_{a}~_{V^{*}}
\langle X(t,-n,0),\varphi(t)v\rangle_{V}\d t \\
&
=\lim_{n\rightarrow\infty}E\int^{b}_{a}~_{V^{*}}
\langle X(a,-n,0)+\int_{a}^{t}A(\sigma,X(\sigma,-n,0))\d\sigma,
\varphi(t)v\rangle_{V}\d t\\
& \quad
+ \lim_{n\rightarrow\infty}E\left(\int^{b}_{a}\langle
\int^{t}_{a}B(\sigma,X(\sigma,-n,0))
\d W(\sigma),\varphi(t)v\rangle_{H}\d t\right)\\
&
=E\int^{b}_{a}~_{V^{*}}\langle X(a)+\int_{a}^{t}\left(Y_{1}(\sigma)
+Y_{2}(\sigma)\right)\d\sigma,\varphi(t)v\rangle_{V}\d t\\
& \quad
+E\left(\int^{b}_{a}\langle \int^{t}_{a}Z(\sigma)
\d W(\sigma),\varphi(t)v\rangle_{H}\d t\right).
\end{align*}
Let $Y(\sigma):=Y_{1}(\sigma)+Y_{2}(\sigma)\in W\subset V^{*}$, we have
$$X(t)=X(a)+\int^{t}_{a}Y(\sigma)\d\sigma+\int^{t}_{a}Z(\sigma)\d W(\sigma),
\quad \d t\otimes \P\rm{\mbox{-}a.e.}$$
Thus, it remains to verify that
$$Y=A(\cdot,X),\quad Z=G(\cdot,X),\quad
\d t\otimes \P\rm{\mbox{-}a.e.}$$
To this end, for any $\phi\in K_{1}\cap K_{2}\cap L^{2}([a,b]\times\Omega,\d t\otimes \P;H)$, we have
\begin{align}\label{bseq}
&
E\|X(t,-n,0)\|^{2}_{H}-E\|X(a,-n,0)\|^{2}_{H} \\\nonumber
&
=E\int^{t}_{a}\left(2_{V^{*}}\langle A(\sigma,X(\sigma,-n,0)),
X(\sigma,-n,0)\rangle
+\|G(\sigma,X(\sigma,-n,0))\|^{2}_{L_{2}(U,H)}\right)\d\sigma\\\nonumber
&
\leq E\int^{t}_{a}\Big[2_{V^{*}}\langle
A(\sigma,X(\sigma,-n,0))-A(\sigma,\phi(\sigma)),
X(\sigma,-n,0)-\phi(\sigma)\rangle_{V}\\\nonumber
&\qquad
+\|G(\sigma,X(\sigma,-n,0))-G(\sigma,\phi(\sigma))\|^{2}_{L_{2}(U,H)}
+2_{V^{*}}\langle A(\sigma,\phi(\sigma)),X(\sigma,-n,0)
\rangle_{V}\\\nonumber
& \qquad
+2_{V^{*}}\langle A(\sigma,X(\sigma,-n,0))
-A(\sigma,\phi(\sigma)),\phi(\sigma)\rangle_{V}\\\nonumber
& \qquad
+2\langle G(\sigma,X(\sigma,-n,0)),G(\sigma,\phi(\sigma))
\rangle_{L_{2}(U,H)}-\|G(\sigma,\phi(\sigma))\|_{L_{2}(U,H)}^{2}
\Big]\d\sigma.
\end{align}

For every nonnegative $\psi\in L^{\infty}([a,b],\d t;\mathbb{R})$,
first multiplying $\psi(t)$ on both sides of \eqref{bseq}, then
integrating with respect to $t$ from $a$ to $b$ and letting
$n\rightarrow\infty$, it follows from (H2) and
$2\lambda-L_{G}^{2}>0$ that
\begin{align}\label{Bountheq2}
&
E\int^{b}_{a}\psi(t)\left(\|X(t)\|^{2}_{H}-\|X(a)\|^{2}_{H}
\right)\d t \\\nonumber
&
\leq E\Bigg(\int^{b}_{a}\psi(t)\int^{t}_{a}\Big(2_{V_{1}^{*}}
\langle Y_{1}(\sigma)-A_{1}(\phi(\sigma)),\phi(\sigma)
\rangle_{V_{1}}+2_{V_{1}^{*}}\langle A_{1}(\phi(\sigma)),
X(\sigma)\rangle_{V_{1}}\\\nonumber
& \qquad
+2_{V_{2}^{*}}\langle Y_{2}(\sigma)-A_{2}(\sigma,\phi(\sigma)),
\phi(\sigma)\rangle_{V_{2}}+2_{V_{2}^{*}}\langle
A_{2}(\sigma,\phi(\sigma)),X(\sigma)\rangle_{V_{2}}\\\nonumber
& \qquad
+2\langle Z(\sigma),G(\sigma,\phi(\sigma))\rangle_{L_{2}(U,H)}
-\|G(\sigma,\phi(\sigma))\|^{2}_{L_{2}(U,H)}\Big)\d\sigma\d t\Bigg).
\end{align}
Applying It\^o's formula to $\|X(t)\|_{H}^{2}-\|X(a)\|_{H}^{2}$
in \eqref{Bountheq2}, we get
%\begin{align}\label{Bountheq3}
%&E\int^{b}_{a}\psi(t)\left(\|X(t)\|^{2}_{H}-\|X(a)\|^{2}_{H}
%\right)\d t \\\nonumber
%&=E\Bigg(\int^{b}_{a}\psi(t)\int^{t}_{a}\Big(2_{V_{1}^{*}}
%\langle Y_{1}(\sigma),X(\sigma)\rangle_{V_{1}}+2_{V_{2}^{*}}\langle
%Y_{2}(\sigma),X(\sigma)\rangle_{V_{2}}
%+\|Z(\sigma)\|^{2}_{L_{2}(U,H)}\Big)\d\sigma \d t\Bigg).
%\end{align}
%Therefore, \eqref{Bountheq2} and \eqref{Bountheq3} imply
\begin{align}\label{ineq3}
0\geq
&
E\Bigg(\int^{b}_{a}\psi(t)\int^{t}_{a}\Big(2_{V^{*}}
\langle Y(\sigma)-A(\sigma,\phi(\sigma)),
X(\sigma)-\phi(\sigma)\rangle_{V}\\\nonumber
&\qquad
+\|G(\sigma,\phi(\sigma))-Z(\sigma)\|_{L_{2}(U,H)}^{2}\Big)
\d\sigma \d t\Bigg).
\end{align}
Taking $\phi=X$ in \eqref{ineq3},  we have
$Z=G(\cdot,X)$, $\d t\otimes \P$-a.e. Then, applying \eqref{ineq3} to
$\phi=X-\epsilon\widetilde{\phi}v$ for $\epsilon>0$ and
$\widetilde{\phi}\in L^{\infty}([a,b]\times\Omega,\d t\otimes\P;
\mathbb{R})$, $v\in V$,
we have
\begin{equation*}
E\Bigg(\int^{b}_{a}\psi(t)\int^{t}_{a}
2_{V^{*}}\langle Y(\sigma)-A(\sigma,X(\sigma)
-\epsilon\widetilde{\phi}(\sigma)v),
\epsilon\widetilde{\phi}(\sigma)v\rangle_{V}\d\sigma\d t\Bigg)\leq 0.
\end{equation*}
Dividing both sides by $\epsilon$ and letting $\epsilon\rightarrow0$,
according to Lebesgue's dominated convergence theorem, (H1) and (H4), we obtain
\begin{equation*}
E\left(\int^{b}_{a}\psi(t)\int^{t}_{a}\widetilde{\phi}(\sigma)
_{V^{*}}\langle Y(\sigma)-A(\sigma,X(\sigma)),
v\rangle_{V}\d\sigma\d t\right) \leq 0.
\end{equation*}
In view of the arbitrariness of $\psi$, $\widetilde{\phi}$ and $v$,
we conclude that
$Y=A(\cdot,X)$, $\d t\otimes \P$-a.e. This completes the existence proof, i.e.
$$
X(t)=X(a)+\int^{t}_{a}A(\sigma,X(\sigma))\d\sigma
+\int^{t}_{a}G(\sigma,X(\sigma))\d W(\sigma),\quad
\d t\otimes \P\rm{\mbox{-}a.e.}
$$
By the arbitrariness of interval $[a,b]\subset \mathbb{R}$, we
conclude that $X(\cdot)$ is a solution on $\mathbb{R}$.
It follows from \eqref{ineq2} that $\sup\limits_{t\in\mathbb{R}}E\|X(t)\|_{H}^{2}<\infty$.
The uniqueness of $L^{2}$-bounded solution is a consequence of
\eqref{ANSCP}.

The goal next is to prove that $\widehat{\mu}$ is unique with the properties (i) and (ii).
Note that
$$\sup_{t\in\mathbb{R}}\int_{H}\|x\|^{2}_{H}\widehat{\mu}(t)(\d x)=\sup_{t\in\mathbb{R}}E\|X(t)\|^{2}_{H}<\infty.$$
In view of the Chapman-Kolmogorov equation, we have
$\mu(t,s,\L(X(s,-n,0)))=\L(X(t,-n,0))$.
Then according to the Feller property, we get
$$\mu(t,s,\widehat{\mu}(s))=\widehat{\mu}(t).$$
Suppose that $\mu_{1}$ and $\mu_{2}$ satisfy properties (i) and (ii),
let $\zeta_{n,1}$ and $\zeta_{n,2}$ be random variables with the distributions $\mu_{1}(-n)$ and $\mu_{2}(-n)$ respectively.
Then consider the solutions $X(t,-n,\zeta_{n,1})$ and $X(t,-n,\zeta_{n,2})$ on $[-n,\infty)$,
we have
\begin{align*}
d_{BL}(\mu_{1}(t),\mu_{2}(t))
&
=\sup_{\|f\|_{BL}\leq1} \left|\int_{H}f(x)\d
\left(\mu(t,-n,\mu_{1}(-n))-\mu(t,-n,\mu_{2}(-n))\right)\right| \\
&
\leq \left(E\|X(t,-n,\zeta_{n,1})-X(t,-n,\zeta_{n,2})\|^{2}_{H}
\right)^{1/2}.
\end{align*}
Thus \eqref{ANSCP} yields that $\mu_{1}(t)=\mu_{2}(t)$
for all $t\in\R$.
\end{proof}

\begin{remark}\rm
Note that we call $X(t),t\in\R$ a solution to \eqref{eqSPDE1}
if for any $[s,r]\subset\R$, $X(t),t\in[s,r]$ is a solution to \eqref{eqSPDE1}.
Here we cannot obtain the existence and uniqueness of solutions to \eqref{eqSPDE1}
for $t\in\R$ for any given initial data because backward orbits through the
initial data are not necessarily unique.
But in Theorem \ref{Boundedth}, we prove that there exists a unique $L^2$-bounded solution
$X(t),t\in\R$ by the pullback attraction method.
And we will also show that this bounded solution $X$ is globally asymptotically stable
in square-mean sense below (see Theorem \ref{gasms}).
%satisfies the integral equation, i.e. $\mathbb P$-a.s.
%\begin{equation*}
%X(t)=X(s)+\int_s^t \left(A(X(\tau))+F(\tau,X(\tau))\right)\d\tau+\int_s^t G(\tau,X(\tau))\d W(\tau),\quad s\leq t\leq r.
%\end{equation*}
Therefore, if $Y(t),t\in\R$ is another solution to \eqref{eqSPDE1} and there exists $s\in\R$ such that $E\|Y(s)\|_H^2<\infty$, then we have
\begin{equation*}
\sup_{t\geq s}E\|Y(t)\|_H^2<\infty.
\end{equation*}
But on the other hand, we necessarily have
\begin{equation*}
\limsup_{t\rightarrow-\infty}E\|Y(t)\|_H^2=+\infty.
\end{equation*}
Indeed, if this is false, then $Y$ is also an $L^2$-bounded solution to \eqref{eqSPDE1}, which contradicts
the uniqueness of $L^2$-bounded solution.
\end{remark}

\begin{definition}[See \cite{FL}]\rm
We say that a solution $X(\cdot)$ of equation \eqref{eqSPDE1} is
{\em stable in square-mean sense}, if for each $\epsilon>0$, there
exists $\delta>0$ such that for all $t\geq0$
\begin{equation*}
  E\|X(t,0,\zeta_{0})-X(t)\|_{H}^{2}<\epsilon,
\end{equation*}
whenever $E\|\zeta_{0}-X(0)\|_{H}^{2}<\delta$. The solution
$X(\cdot)$ is said to be {\em asymptotically stable in
square-mean sense} if it is stable in square-mean sense and
\begin{equation}\label{stab}
  \lim_{t\rightarrow\infty}E\|X(t,0,\zeta_{0})-X(t)\|_{H}^{2}=0.
\end{equation}
We say $X(\cdot)$ is {\em globally asymptotically stable
in square-mean sense} provided  \eqref{stab} holds
for any $\zeta_{0}\in L^{2}(\Omega,\mathcal{F}_{0},\P;H)$.
\end{definition}

Applying Lemma \ref{ivcont} we obtain the following result:
\begin{theorem}\label{gasms}
Consider equation \eqref{eqSPDE1}. Suppose that
$2\lambda-L_{G}^{2}\geq0$ and {\rm{(H1)}}--{\rm{(H4)}} hold.
If $\lambda'>0$ or $2\lambda>L_{G}^{2}$,
then the unique $L^{2}$-bounded solution of
equation \eqref{eqSPDE1} is globally asymptotically stable in
square-mean sense. Moreover,
\begin{equation}\label{gaseq}
    E\|X(t,s,\zeta_{s})-X(t)\|^{2}_{H}
   \leq \begin{cases}
    E\|\zeta_{s}-X(s)\|_{H}^{2}\wedge \left\{\lambda'(r-2)
   (t-s)\right\}^{-\frac{2}{r-2}}, &\text{if $\lambda'>0$}\\
	{\rm{e}}^{-(2\lambda-L_{G}^{2})(t-s)}E\|\zeta_{s}-X(s)\|_{H}^{2},
   &\text{if $2\lambda>L_{G}^{2}$}
  \end{cases}
\end{equation}
for any $t\geq s$ and
$\zeta_{s}\in L^{2}(\Omega,\mathcal{F}_{s},\P;H)$.
\end{theorem}

\begin{remark}\rm\label{hulllem}
\begin{enumerate}
\item If $A_{2}$ and $G$ satisfy (H2) and (H3), then every pair of
    functions $\left(\tilde{A_{2}},\tilde{G}\right)\in H(A_{2},G)$
    possess the same property with the same constants, where
      \[
      H(A_{2},G):=\overline{\left\{\left(A_{2}^{\tau},G^{\tau}\right):
      \tau\in\R\right\}}.
      \]
      Here $\overline{\left\{\left(A_{2}^{\tau},G^{\tau}\right):
      \tau\in\R\right\}}$ means the closure of $\left\{\left(
      A_{2}^{\tau},G^{\tau}\right):\tau\in\R\right\}$.
\item If $A_{2}$ and $G$ satisfy the conditions (H1), (H2), (H4)
     and (H5), then $A_{2}\in BUC(\R\times V, V_{2}^{*})$,
     $G\in BUC(\R\times V,L_{2}(U,H))$ and $H(A_{2},G)\subset
     BUC(\R\times V, V_{2}^{*})\times BUC(\R\times V,L_{2}(U,H))$.
\end{enumerate}
\end{remark}

\begin{definition}\label{compatible} \rm
Let $\{\varphi (t)\}_{t\in\mathbb R}$ be a solution of
equation \eqref{eqSPDE1}. Then $\varphi$ is called {\em
compatible} (respectively, {\em strongly compatible}) {\em in
distribution} if the following conditions are fulfilled:
\begin{enumerate}
\item
there exists a bounded closed subset
$\mathcal Q\subset L^{2}(\Omega,\mathbb P;H$) such that
$\varphi(\mathbb R)\subseteq \mathcal Q$;
\item
$\mathfrak N_{(F,G)}\subseteq \tilde{\mathfrak N}_{\varphi}$
(respectively,
$\mathfrak M_{(F,G)}\subseteq \tilde{\mathfrak M}_{\varphi}$),
where $\tilde{\mathfrak N}_{\varphi}$
(respectively, $\tilde{\mathfrak M}_{\varphi}$) means the set of
all sequences $\{t_n\}\subset\mathbb R$ such that the sequence
$\{\varphi(\cdot+t_n)\}$ converges to $\varphi(\cdot)$
(respectively, $\{\varphi(\cdot+t_n)\}$ converges) in distribution
uniformly on any compact interval.
\end{enumerate}
\end{definition}

Now we show that the $L^{2}$-bounded solution $X(t),t\in\R$
for equation \eqref{eqSPDE1} is strongly compatible in distribution.
To this end, we need the tightness of the family of distributions
$\{\P\circ[X(t)]^{-1}\}_{t\in\mathbb{R}}$. Therefore,
we need the following condition (H6) which is used by many
works (see, e.g. \cite{Liu2010}).

$\textbf{(H6)}$ Assume that there exists a closed subset $S\subset H$
equipped with the norm $\|\cdot\|_{S}$ such that $V\subset S$
is continuous and $S\subset H$ is compact.
Let $T_{n}$ be a sequence of positive definite self-adjoint
operators on $H$ such that for each $n\geq1$,
$$\langle x,y\rangle_{n}:=\langle x,T_{n}y\rangle_{H},
\quad x,y\in H,$$
defines a new inner product on $H$. Assume further that
the norms $\|\cdot\|_{n}$ generated by $\langle ~,~\rangle_{n}$
are all equivalent to $\|\cdot\|_{H}$ and for all $x\in S$ we have
$$\|x\|_{n} \uparrow \|x\|_{S} \quad {\rm{as}}~n\rightarrow\infty.$$
Furthermore, we suppose that for each $n\geq1$,
$T_{n}: V\rightarrow V$ is continuous and there exist constants
$c_{4}>0$, $M_{0}>0$ such that for all $v\in V$, $t\in\mathbb{R}$
\begin{equation*}
  2_{V^{*}}\langle A(t,v),T_{n}v\rangle_{V}
  +\|G(t,v)\|^{2}_{L_{2}(U,H_{n})}\leq-c_{4}\|v\|_{n}^{2}+M_{0}.
\end{equation*}

\begin{prop}\label{tightprop}
Consider equation \eqref{eqSPDE1}.
Suppose that conditions of Theorem \ref{gasms} hold.
If {\rm{(H6)}} hold
then the $L^{2}$-bounded solution $X(\cdot)$ satisfies
\begin{equation}\label{compeq}
\sup_{t\in\mathbb{R}}E\|X(t)\|_{S}^{2}<\infty.
\end{equation}

In particular, the family of distributions
$\{{P\circ[X(t)]^{-1}}\}_{t\in\mathbb{R}}$ is tight.
\end{prop}
\begin{proof}
Similar to the proof of Proposition 1 in \cite{CML2020},
\eqref{compeq} can be obtained by It\^o's formula, (H6)
and Gronwall's lemma.
\end{proof}

The following lemma is a direct corollary of
Theorem 3.1 in \cite{CML2020}.
\begin{lemma}\label{conlemma}
Suppose that $A_{n}$, $A$, $G_{n}$, $G$ satisfy
{\rm{(H1)--(H4)}} with the same constants
$c$, $c_{1}$, $c_{2}$, $c_{3}$, $c_{2}'$, $c_{3}'$,
$M_{0}$, $\alpha_{i}$, $i=1,2$ and $L_{G}$.
Let $X_{n}$ be the solution of the Cauchy problem
\begin{equation}\label{system1}
  \left\{
   \begin{aligned}
   &\ \d X(t)=A_{n}(t,X(t))\d t+G_{n}(t,X(t))\d W(t)\\
  &\ X(s)=\zeta_{n}^{s}
   \end{aligned}
   \right.
  \end{equation}
and $X$ be the solution to the Cauchy problem
\begin{equation}\label{system2}
  \left\{
   \begin{aligned}
   &\ \d X(t)=A(t,X(t))\d t+G(t,X(t))\d W(t)\\
  &\ X(s)=\zeta^{s}.
   \end{aligned}
   \right.
  \end{equation}
Assume further that
\begin{enumerate}
  \item[(1)] $\lim\limits_{n\rightarrow\infty}A_{i,n}(t,x)
  =A_{i}(t,x)~in~V_{i}^{*}~for~all~t\in\mathbb{R},~x\in V,~i=1,2$;
  \item[(2)] $\lim\limits_{n\rightarrow\infty}G_{n}(t,x)
  =G(t,x)~in~L_{2}(U,H)~for~all~t\in\mathbb{R},~x\in V$.
\end{enumerate}
Then we have the following conclusions:
\begin{enumerate}
  \item If $\lim\limits_{n\rightarrow\infty}
  E\|\zeta_{n}^{s}-\zeta^{s}\|^{2}_{H}=0$,
  then $\lim\limits_{n\rightarrow\infty}E\sup\limits_{s\leq\tau\leq t}
  \|X_{n}(\tau)-X(\tau)\|^{2}_{H}=0$ for any $t>s$;
  \item If $\lim\limits_{n\rightarrow\infty}\zeta_{n}^{s}
  =\zeta^{s}$ in probability,
  then $\lim\limits_{n\rightarrow\infty}\sup\limits_{\tau\in[s,t]}
  \|X_{n}(\tau)-X(\tau)\|_{H}=0$ in probability;
  \item If $\lim\limits_{n\rightarrow\infty}
  d_{BL}(\mathcal{L}(\zeta_{n}^{s}),\mathcal{L}(\zeta^{s}))=0$
  in $Pr(H)$, then
  $$\lim\limits_{n\rightarrow\infty}d_{BL}(\mathcal{L}(X_{n}),
  \mathcal{L}(X))=0 \quad {\rm{in}} ~Pr(C([s,\infty),H)).$$
\end{enumerate}
\end{lemma}

\begin{theorem}\label{SCth}
Consider equation \eqref{eqSPDE1}. Suppose that
$2\lambda-L_{G}^{2}\geq0$ and {\rm{(H1)}}--{\rm{(H6)}} hold.
If $\lambda'>0$ or $2\lambda>L_{G}^{2}$, then the unique
$L^{2}$-bounded solution is strongly compatible in distribution.
\end{theorem}

\begin{proof}
It follows from Remark \ref{hulllem} that
$H(A_{2},G)\subset BUC(\R\times V,V_{2}^{*})\times
BUC(\R\times V,L_{2}(U,H))$.
Let $\{t_{n}\}\in\mathfrak M_{(A_{2},G)}$, then there exists
$(\tilde{A_{2}},\tilde{G})\in H(A_{2},G)$ such that
$$
\lim_{n\rightarrow\infty}\sup_{|t|\leq l,\|x\|_{V}\leq r}
\|A_{2}(t+t_{n},x)-\tilde{A_{2}}(t,x)\|_{V_{2}^{*}}=0,
$$
$$
\lim_{n\rightarrow\infty}\sup_{|t|\leq l,\|x\|_{V}\leq r}
\|G(t+t_{n},x)-\tilde{G}(t,x)\|_{L_{2}(U,H)}=0,
$$
for any $l>0$ and $r>0$.
Let $X_{n}$ be the unique $L^{2}$-bounded solution of
\begin{equation*}
\d X(t)=\left(A_{1}(X(t))+A_{2}(t+t_{n},X(t))\right)\d t
+G(t+t_{n},X(t))\d W(t)
\end{equation*}
and $\tilde{X}$ be the unique $L^{2}$-bounded solution of
\begin{equation}\label{SCDLeq}
\d X(t)=\left(A_{1}(X(t))+\tilde{A_{2}}(t,X(t))\right)\d t
+\tilde{G}(t,X(t))\d W(t).
\end{equation}
We now prove that for any $[a,b]\subset\R$,
$\lim\limits_{n\rightarrow\infty}\sup\limits_{t\in[a,b]}
d_{BL}(\mathcal L(X_{n}(t)),\mathcal L(\tilde{X}(t)))=0$.
According to Lemma \ref{conlemma}, we only need to prove that
$\lim\limits_{n\rightarrow\infty}d_{BL}(\mathcal{L}(X_{n}(t)),
\mathcal{L}(\tilde{X}(t)))=0$
in $Pr(H)$ for every $t\in\R$. To this end, it suffices to
show that for every sequence
$\{\gamma_{k}'\}:=\{\gamma_{k}'\}_{k=1}^{\infty}\subset\N$,
there exists a subsequence $\{\gamma_{k}\}$ of $\{\gamma_{k}'\}$
such that
$\lim\limits_{k\rightarrow\infty}d_{BL}(\mathcal{L}(X_{\gamma_{k}}(t)),
\mathcal{L}(\tilde{X}(t)))=0$ in $Pr(H)$
for every $t\in\R$.

Given $r\geq1$, according to the tightness of
$\{\L(X_{\gamma_{k}'}(-r))\}$,
there exists a subsequence $\{\gamma_{k}\}\subset\{\gamma_{k}'\}$
such that $\L(X_{\gamma_{k}}(-r))$ converges weakly to some
probability measure $\mu_{r}$ in $Pr(H)$.
Let $\xi_{r}$ be a random variable with distribution $\mu_{r}$.
Define $Y_{r}(t):=X(t,-r,\xi_{r})$, where
$X(t,-r,\xi_{r})$, $t\in[-r,+\infty)$ is a solution to
the following Cauchy problem
\begin{equation*}
  \left\{
   \begin{aligned}
   &\ \d X(t)=\left(A_{1}(X(t))+\tilde{A_{2}}(t,X(t))\right)\d t
   +\tilde{G}(t,X(t))\d W(t)\\
  &\ X(-r)=\xi_{r}.
   \end{aligned}
   \right.
  \end{equation*}
In view of Lemma \ref{conlemma}, we have
$$\lim\limits_{k\rightarrow\infty}d_{BL}(\L(X_{\gamma_{k}}),\L(Y_{r}))=0
\quad{\rm {in}}~ Pr(C([-r,+\infty),H)).$$
Since $\{\L(X_{\gamma_{k}}(-r-1))\}$ is tight,
going if necessary to a subsequence, we can assume that
$\L(X_{\gamma_{k}}(-r-1))$ converges weakly to some  probability
measure $\mu_{r+1}$ in $Pr(H)$. Let $\xi_{r+1}$ be a random
variable with distribution $\mu_{r+1}$.
In light of Lemma \ref{conlemma}, we have
 $$\lim\limits_{k\rightarrow\infty}d_{BL}
 (\L(X_{\gamma_{k}}),\L(Y_{r+1}))=0
 \quad{\rm {in}} ~ Pr(C([-r-1,+\infty),H)),$$
where $Y_{r+1}(t):=X(t,-r-1,\xi_{r+1})$, $t\in[-r-1,+\infty)$.
Therefore, we have $d_{BL}(\L(Y_{r}),\L(Y_{r+1}))=0$
in $Pr(C([-r,+\infty),H))$. In particular,
$\L(Y_{r}(t))=\L(Y_{r+1}(t))$ for all $t\geq-r$.

Define $\nu(t):=\L(Y_{r}(t))$,  $t\geq-r$.
We use a standard diagonal argument to
extract a subsequence which we still denote by
$\{X_{\gamma_{k}}\}$ satisfying
 $$\lim\limits_{k\rightarrow\infty}
 d_{BL}(\L(X_{\gamma_{k}}(t)),\nu(t))=0 \quad{\rm in} ~Pr(H)$$
for every $t\in\R$. Note that
$\sup\limits_{t\in\R}\int_{H}\|x\|^{2}\nu(t)(\d x)<+\infty$.
And we have $\mathbb P$-a.s.
\begin{equation*}
  Y_{r}(t)=Y_{r}(s)+\int_{s}^{t}\left(A_{1}(Y_{r}(\sigma))
  +\tilde{A_{2}}(\sigma,Y_{r}(\sigma))\right)\d\sigma
  +\int_{s}^{t}\tilde{G}(\sigma,Y_{r}(\sigma))\d W(\sigma),
\end{equation*}
where $t\geq s\geq-r$.
By the uniqueness in law of the solutions for equation
\eqref{SCDLeq}, we have
$\L(Y_{r}(t))=\mu(t,s,\L(Y_{r}(s)))$, $t\geq s\geq-r$,
i.e. $\nu(t)=\mu(t,s,\nu(s))$, $t\geq s$.
In view of Theorem \ref{Boundedth},
we obtain $\nu=\L(\tilde{X})$. Therefore, we have
\begin{equation*}
  \lim\limits_{k\rightarrow\infty}d_{BL}(\L(X_{\gamma_{k}}(t)),
  \L(\tilde{X}(t)))=0 \quad {\rm{in}}~ Pr(H)
\end{equation*}
for every $t\in\R$.

Note that $X(\cdot+t_{n})$ and $X_{n}(\cdot)$ share the same
distribution. It follows from Definition \ref{compatible} and
Lemma \ref{conlemma} that $X(\cdot)$ is strongly compatible
in distribution.
\end{proof}

By Theorems \ref{th1} and \ref{SCth}, we have the following result.
\begin{coro}\label{corL2*}
Under the conditions of Theorem \ref{SCth} the
following statements hold:
\begin{enumerate}
\item
If $A_{2}\in$ $ C(\mathbb R\times V,V_{2}^{*})$ and
$G\in C(\R\times V,L_{2}(U,H))$
are jointly stationary (respectively, $T$-periodic, quasi-periodic
with the spectrum of frequencies $\nu_1,\ldots,\nu_k$, almost
periodic, almost automorphic, Birkhoff recurrent, Lagrange
stable, Levitan almost periodic, almost recurrent, Poisson
stable) in $t\in\R$ uniformly with respect to $x$ on each bouned
subset, then so is the unique solution
$\varphi \in C_{b}(\mathbb R,L^2(\Omega,\mathbb P;H))$ of equation
\eqref{eqSPDE1} in distribution;
\item
If $A_{2}\in$ $ C(\mathbb R\times V,V_{2}^{*})$ and
$G\in C(\R\times V,L_{2}(U,H))$ are
Lagrange stable and jointly pseudo-periodic (respectively,
pseudo-re\-cur\-rent) in $t\in\R$ uniformly with respect to
$x$ on each bouned subset, then equation \eqref{eqSPDE1} has a unique
solution $\varphi \in C_{b}(\mathbb R,L^2(\Omega,\mathbb P;H))$
which is pseudo-periodic (respectively, pseudo-recurrent) in
distribution.
\end{enumerate}
\end{coro}

\section{The second Bogolyubov theorem}
Consider the following stochastic partial differential equation
\begin{equation}\label{eqLSDE1.1}
\d X_{\varepsilon}(t)=\left(A(X_{\varepsilon}(t))
+F\left(\frac{t}{\varepsilon},X_{\varepsilon}(t)\right)\right)\d t
+G\left(\frac{t}{\varepsilon},X_{\varepsilon}(t)\right)\d W(t),
\end{equation}
where $A(x)=A_{1}(x)+A_{2}(x)$,
$A_{i}:V_{i}\rightarrow V_{i}^{*}$, $i=1,2$,
$F\in$ $ C(\mathbb R\times H,H)$, $G\in C(\R\times H,L_{2}(U,H))$
and $0<\varepsilon\leq1$. Here $W$ is a $U$-valued two-sided
cylindrical Wiener process with the identity covariance operator
with respect to a filtered probability space
$(\Omega,\mathcal F,\mathbb P,\mathcal F_{t})$,
where $\mathcal F_{t}:=\sigma\{W(u)-W(v):\ u,v\le t\}$.
In this section, we will omit the index $H$ of $\|\cdot\|_{H}$
and $\langle\cdot,\cdot\rangle_{H}$ if it does not cause confusion.

We employ $\Psi$ to denote the space of all decreasing, positive
bounded functions $\omega:\R_{+}\rightarrow\R_{+}$ with
$\lim\limits_{t\rightarrow+\infty}\omega(t)=0$.
Below we need additional conditions.
\begin{enumerate}
\item[\textbf{(H2$'$)}]
There exist constants $\lambda,\lambda'\geq0$, $L_{F},L_{G},M_{0}>0$,
$r>2$ and $\lambda_{F}\in\R$ such that
\[
_{V^{*}}\langle A(u)-A(v),u-v\rangle_{V}
\leq-\lambda\|u-v\|^{2}-\lambda'\|u-v\|^{r},
\]
\[
\langle F(t,x)-F(t,y),x-y\rangle\leq\lambda_{F}\|x-y\|^{2},
\quad
\|F(t,0)\|\leq M_{0}
\]
\[
\|F(t,x)-F(t,y)\|\leq L_{F}\|x-y\|,\quad
\|G(t,x)-G(t,y)\|_{L_{2}(U,H)}\leq L_{G}\|x-y\|
\]
for any $t\in\R$, $u,v\in V$ and $x,y\in H$;
\item[\textbf{(G1)}]
There exist functions $\omega_{1}\in \Psi$ and $\bar{F}\in C(H,H)$
such that
\begin{equation}\label{eqG3}
\frac{1}{T}\left\|\int_{t}^{t+T}[F(s,x)-\bar{F}(x)]\d s\right\|
\le\omega_{1}(T)(1+\|x\|)\nonumber
\end{equation}
for any $T>0$, $x\in H$ and $t\in\mathbb R$;
 \item[\textbf{(G2)}]
There exist functions $\omega_{2}\in \Psi$ and
$\bar{G}\in C(H,L_{2}(U,H))$ such that
\begin{equation}\label{eqG4}
\frac{1}{T}\int_{t}^{t+T}\left\|G(s,x)-\bar{G}(x)\right\|_{L_{2}(U,H)}^{2}
\d s\le\omega_{2}(T)(1+\|x\|^{2})\nonumber
\end{equation}
for any $T>0$, $x\in H$ and $t\in\mathbb R$.
\end{enumerate}

\begin{remark}\label{oaest}\rm
\begin{enumerate}

\item Note that the estimate of solutions to \eqref{eqLSDE1.1}
      for the integral of time increment on $H$ in Lemma \ref{orstes} is weaker than
      the H\"older continuity but helpful to establish the first Bogolyubov theorem.
      If $F$ is just monotone instead of Lipschitz, we need this estimate on $V$,
      which is crucial to establishing the first Bogolyubov theorem based on the technique of time discretization.
      But we cannot obtain this estimate on $V$ unless there are additional assumptions
      on higher regularity of initial data and coefficients.
      However, this higher regularity condition is too strong to apply to porous media equations, which is one of our examples. On the other hand, the higher regularity of initial data is too strong to establish the second Bogolyubov theorem and global averaging principle, which are our main results in this paper.
      Indeed, the first Bogolyubov theorem and the existence and uniqueness of $L^2$-bounded solutions
      play important roles in establishing the second Bogolyubov theorem and global averaging principle (see Theorems \ref{averth} and \ref{gath}).  But we only establish that this bounded solution belongs to
      $L^\infty(\R;L^2(\Omega,\mathbb P;H))\cap L^\infty(\R;L^2(\Omega,\mathbb P;S))$
      (see Theorem \ref{Boundedth} and Proposition \ref{tightprop}), whose regularity is not higher enough for our purpose.
\item In order to obtain recurrent solutions, the systems \eqref{eqLSDE1.1}
      need to be dissipative; that is, $2\lambda-2\lambda_{F}-L_{G}^{2}\geq0$.
  Since the condition
  $2\lambda-2L_{F}-L_{G}^{2}\geq0$ is stronger
  than $2\lambda-2\lambda_{F}-L_{G}^{2}\geq0$ when $\lambda_F$ is negative,
  we introduce the conditions of monotonicity and Lipschitz continuity of $F$ in (H2$'$) simultaneously.
\item Notice that we can just assume $A_{2}$ is hemicontinuous
      when $A_{2}$ is independent of time $t$. In the following, we
      still say that equation \eqref{eqLSDE1.1} satisfies (H5)
      (respectively, (H6)), if $F$ and $G$ are continuous in
      $t\in\R$ uniformly with respect to $u$ on each bounded
      $Q\subset H$ (respectively, there exist constants
      $c_{4},M_{0}>0$ such that $2_{V^{*}}\langle A(v),T_{n}v
      \rangle_{V}+2\langle F(t,v),T_{n}v\rangle
      +\|G(t,v)\|_{L_{2}(U,H)}\leq-c_{4}\|v\|_{n}^{2}+M_{0}$ for
      all $v\in V$ and $t\in\R$).

\item
      It can be verified that (G1) (respectively, (G2)) implies
      $$\lim\limits_{T\rightarrow\infty}\frac{1}{T}\int_{t}^{t+T}
      F(s,x)\d s=\bar{F}(x)$$
      (respectively,
      $\lim\limits_{T\rightarrow\infty}\frac{1}{T}\int_{t}^{t+T}
      \|G(s,x)-\bar{G}(x)\|_{L_{2}(U,H)}^{2}$\d s=0)
      uniformly with respect to $t\in\R$ and $x$ in any bounded subset on $H$.
%\item If $F$ and $G$ satisfy (H2$'$), (H4) and (G1)--(G2), it can be proved
 %   that $\bar{F}$ and $\bar{G}$ also satisfy
 %   (H2$'$) and (H4)  with the same constants (see Appendix for the proof).
 %   Therefore, under the same conditions,
 %   estimates \eqref{ineq2}, \eqref{peq}, \eqref{gaseq},
 %   \eqref{compeq} uniformly hold for $\varepsilon\in(0,1]$,
 %   $\bar{F}$ and $\bar{G}$.
\end{enumerate}
\end{remark}

Denote by
$F_{\varepsilon}(t,x):=F(\frac{t}{\varepsilon},x)$ and
$G_{\varepsilon}(t,x):=G(\frac{t}{\varepsilon},x)$
for any $t\in\R$, $x\in H$ and $\varepsilon \in(0,1]$.
Equation \eqref{eqLSDE1.1} can be written as
\begin{equation}\label{eqG2.1}
\d X_{\varepsilon}(t)=(A(X_{\varepsilon}(t))
+F_{\varepsilon}(t,X_{\varepsilon}(t)))\d t
+G_{\varepsilon}(t,X_{\varepsilon}(t))\d W(t).
\end{equation}
Along with equations \eqref{eqLSDE1.1}--\eqref{eqG2.1} we
consider the following averaged  equation
\begin{equation}\label{eqG5_1}
\d X(t)=\left(A(X(t))+\bar{F}(X(t))\right)\d t
+\bar{G}(X(t))\d W(t).
\end{equation}

\begin{lemma}\label{FGuni}
If $F$ and $G$ satisfy (H2$'$) and (G1)--(G2), then
$\bar{F}$ and $\bar{G}$ in (G1)--(G2) also satisfy
(H2$'$) with the same constants.
\end{lemma}
\begin{proof}
We only need to prove the conclusion for $\bar{F}$; the case of $\bar{G}$ is similar.
It follows from Remark \ref{oaest} (iv) that
\begin{align*}
\|\bar{F}(0)\|=\left\|\lim\limits_{T\rightarrow\infty}\frac1T\int_0^T F(t,0)\d t\right\|
\leq\lim\limits_{T\rightarrow\infty}\frac1T\int_0^T \|F(t,0)\|\d t\leq M_0.
\end{align*}
For any $x,y\in H$ one sees that
\begin{align*}
&
\left\langle \bar{F}(x)-\bar{F}(y),x-y\right\rangle\\
&
=\left\langle \bar{F}(x)-\frac1T\int_0^T F(t,x)\d t,x-y\right\rangle
+\left\langle \frac1T\int_0^T\left(F(t,x)-F(t,y)\right)\d t,x-y \right\rangle\\
&\quad
+\left\langle\frac1T\int_0^T F(t,y)\d t-\bar{F}(y),x-y\right\rangle\\
&
\leq \frac1T\left\|\int_0^T\left(\bar{F}(x)-F(t,x)\right)\d t\right\|\cdot
\|x-y\|
+\frac1T\int_0^T\left\langle F(t,x)-F(t,y),x-y\right\rangle\d t\\
&\quad
+
\frac1T\left\|\int_0^T
\left(F(t,y)-\bar{F}(y)\right)\d t\right\|\cdot\|x-y\|\\
&
\leq\omega_1(T)(1+\|x\|)\|x-y\|+\lambda_F\|x-y\|^2
+\omega_1(T)(1+\|y\|)\|x-y\|
\end{align*}
and
\begin{align*}
\|\bar{F}(x)-\bar{F}(y)\|
&
\leq\left\|\bar{F}(x)-\frac1T\int_0^TF(t,x)\d t\right\|
+\left\|\frac1T\int_0^T\left(F(t,x)-F(t,y)\right)\d t\right\|\\
&\quad
+\left\|\frac1T\int_0^TF(t,y)\d t-\bar{F}(y)\right\|\\
&
\leq \omega_1(T)(1+\|x\|)+L_F\|x-y\|
+\omega_1(T)(1+\|y\|).
\end{align*}
Letting $T\rightarrow\infty$ in the above two inequalities, we have
\[
\langle\bar{F}(x)-\bar{F}(y),x-y\rangle\leq\lambda_F\|x-y\|^2,\quad
\|\bar{F}(x)-\bar{F}(y)\|\leq L_F\|x-y\|
\]
for all $x,y\in H$ provided $\lim\limits_{T\rightarrow\infty}\omega_1(T)=0$.
\end{proof}

\begin{remark}\label{FGunirm}\rm
It follows from Lemma \ref{FGuni} that
estimates \eqref{ineq2}, \eqref{peq}, \eqref{gaseq},
\eqref{compeq} uniformly hold for $\varepsilon\in(0,1]$, and
$\bar{F}$ and $\bar{G}$.
\end{remark}

Recall that $C_{\alpha}$ mean some positive
constant which depends on $\alpha$. For simplicity,
we just write $C$ when $C$ depends on some parameters of
$\lambda,\lambda',r,\lambda_{F}, L_{F},L_{G},c_{1},c_{2},c_{2}',
\alpha_{1},\alpha_{2},M_{0},c_{3},c_{3}'$ in (H1), (H2$'$)
and (H3).
Let $\delta>0$ be a fixed constant depending on $\varepsilon$.
For any given stochastic process $\phi$,
define a step process $\tilde{\phi}$ such that
$\tilde{\phi}(\sigma)=\phi(s+k\delta)$
for any $\sigma\in[s+k\delta,s+(k+1)\delta)$.
Employing the technique of time discretization, we have
the following estimates.

\begin{lemma}\label{orstes}
Assume {\rm{(H1)}}, {\rm{(H2$'$)}}, {\rm{(H3)--(H4)}}
and {\rm{(G1)--(G2)}} hold.
Let $X_{\varepsilon}(t,s,\zeta_s^\varepsilon),t\geq s$ be the solution of \eqref{eqG2.1}
with the initial value
$X_{\varepsilon}(s,s,\zeta_s^\varepsilon)=\zeta_{s}^{\varepsilon}$ and $\bar{X}(t,s,\zeta_s),t\geq s$ be the
solution of \eqref{eqG5_1} with the initial value
$\bar{X}(s,s,\zeta_s)=\zeta_{s}$. Then we have
\begin{equation}\label{orsteseq}
E\int_{s}^{s+T}\|X_{\varepsilon}(\sigma,s,\zeta_s^\varepsilon)
-\tilde{X}_{\varepsilon}(\sigma,s,\zeta_s^\varepsilon)\|^{2}\d\sigma\leq
C_{T}(1+E\|\zeta_{s}^{\varepsilon}\|^{2})\delta^{\frac{1}{2}}
\end{equation}
and
\begin{equation}\label{avesteseq}
E\int_{s}^{s+T}\|\bar{X}(\sigma,s,\zeta_s)-\tilde{X}(\sigma,s,\zeta_s)\|^{2}\d\sigma
\leq C_{T}(1+E\|\zeta_{s}\|^{2})\delta^{\frac{1}{2}}
\end{equation}
for any $s\in\R$ and $T>0$, where $\tilde{X}:=\tilde{\bar{X}}$.
\end{lemma}
\begin{proof}
For simplicity, let
$X_{\varepsilon}(\sigma):=X_{\varepsilon}(\sigma,s,\zeta_s^\varepsilon)$
and $\bar{X}(\sigma):=\bar{X}(\sigma,s,\zeta_s)$.
By Lemma \ref{pe} we have
\begin{align}\label{orsteseq1}
&
E\int_{s}^{s+T}\|X_{\varepsilon}(\sigma)
-\tilde{X}_{\varepsilon}(\sigma)\|^{2}\d\sigma\\\nonumber
&
=E\int_{s}^{s+\delta}\|X_{\varepsilon}(\sigma)
-\zeta_{s}^{\varepsilon}\|^{2}\d\sigma+E\sum_{k=1}^{T(\delta)-1}
\int_{s+k\delta}^{s+(k+1)\delta}\|X_{\varepsilon}(\sigma)
-X_{\varepsilon}(s+k\delta)\|^{2}\d\sigma\\\nonumber
&\quad
+E\int_{s+T(\delta)\delta}^{s+T}\|X_{\varepsilon}(\sigma)
-X_{\varepsilon}(s+T(\delta)\delta)\|^{2}\d\sigma\\\nonumber
&
\leq C_{T}\left(1+E\|\zeta_{s}^{\varepsilon}\|^{2}\right)\delta
+2E\sum_{k=1}^{T(\delta)-1}\int_{s+k\delta}^{s+(k+1)\delta}
\|X_{\varepsilon}(\sigma)-X_{\varepsilon}(\sigma-\delta)\|^{2}
\d\sigma\\\nonumber
&\quad
+2E\sum_{k=1}^{T(\delta)-1}\int_{s+k\delta}^{s+(k+1)\delta}
\|X_{\varepsilon}(\sigma-\delta)-X_{\varepsilon}(s+k\delta)\|^{2}
\d\sigma\\\nonumber
&
=:C_{T}\left(1+E\|\zeta_{s}^{\varepsilon}\|^{2}\right)\delta
+2\sum_{k=1}^{T(\delta)-1}I_{k}+2\sum_{k=1}^{T(\delta)-1}J_{k}.
\end{align}

Given $k\in[1,T(\delta)-1)$, for any
$\sigma\in[s+k\delta,s+(k+1)\delta)$, by It\^o's formula, (H2$'$),
(H3)--(H4) and Young's inequality we get
\begin{align}\label{orsteseq2}
&
\|X_{\varepsilon}(\sigma)-X_{\varepsilon}(\sigma-\delta)\|^{2}\\\nonumber
&
=\int_{\sigma-\delta}^{\sigma}\left(2_{V^{*}}\langle A(X_{\varepsilon}(u)),
X_{\varepsilon}(u)-X_{\varepsilon}(\sigma-\delta)\rangle_{V}
+2\langle F_{\varepsilon}(u,X_{\varepsilon}(u)),
X_{\varepsilon}(u)-X_{\varepsilon}(\sigma-\delta)\rangle
\right)\d u\\\nonumber
&\quad
+\int_{\sigma-\delta}^{\sigma}
\|G_{\varepsilon}(u,X_{\varepsilon}(u))\|_{L_{2}(U,H)}^{2}\d u
+2\int_{\sigma-\delta}^{\sigma}\langle X_{\varepsilon}(u)
-X_{\varepsilon}(\sigma-\delta),G_{\varepsilon}(u,X_{\varepsilon}(u))
\d W(u)\rangle\\\nonumber
&
\leq\int_{\sigma-\delta}^{\sigma}\bigg(
2\|A_{1}(X_{\varepsilon}(u))\|_{V_{1}^{*}}
\|X_{\varepsilon}(u)-X_{\varepsilon}(\sigma-\delta)\|_{V_{1}}
+2\|A_{2}(X_{\varepsilon}(u))\|_{V_{2}^{*}}
\|X_{\varepsilon}(u)-X_{\varepsilon}(\sigma-\delta)\|_{V_{2}}\\\nonumber
&\qquad
+2\|F_{\varepsilon}(u,X_{\varepsilon}(u))\|
\|X_{\varepsilon}(u)-X_{\varepsilon}(\sigma-\delta)\|
+2L_{G}^{2}\|X_{\varepsilon}(u)\|^{2}+2M_{0}^{2}\bigg)\d u\\\nonumber
&\quad
+2\int_{\sigma-\delta}^{\sigma}\langle X_{\varepsilon}(u)
-X_{\varepsilon}(\sigma-\delta),G_{\varepsilon}(u,X_{\varepsilon}(u))
\d W(u)\rangle\\\nonumber
&
\leq\int_{\sigma-\delta}^{\sigma}\bigg[
2\left(\|X_{\varepsilon}(u)\|^{\alpha_{1}-1}_{V_{1}}+M_{0}\right)
\|X_{\varepsilon}(u)-X_{\varepsilon}(\sigma-\delta)\|_{V_{1}}\\\nonumber
&\qquad
+2\left(\|X_{\varepsilon}(u)\|^{\alpha_{2}-1}_{V_{2}}+M_{0}\right)
\|X_{\varepsilon}(u)-X_{\varepsilon}(\sigma-\delta)\|_{V_{2}}\\\nonumber
&\qquad
+2\left(L_{F}\|X_{\varepsilon}(u)\|+M_{0}\right)
\|X_{\varepsilon}(u)-X_{\varepsilon}(\sigma-\delta)\|
+2L_{G}^{2}\|X_{\varepsilon}(u)\|^{2}+2M_{0}^{2}\bigg]\d u\\\nonumber
&\quad
+2\int_{\sigma-\delta}^{\sigma}\langle X_{\varepsilon}(u)
-X_{\varepsilon}(\sigma-\delta),G_{\varepsilon}(u,X_{\varepsilon}(u))
\d W(u)\rangle\\\nonumber
&
\leq\int_{\sigma-\delta}^{\sigma}\bigg[
2\|X_{\varepsilon}(u)\|^{\alpha_{1}}_{V_{1}}
+2\|X_{\varepsilon}(u)\|^{\alpha_{1}-1}_{V_{1}}
\|X_{\varepsilon}(\sigma-\delta)\|_{V_{1}}
+2M_{0}\|X_{\varepsilon}(u)\|_{V_{1}}
+2M_{0}\|X_{\varepsilon}(\sigma-\delta)\|_{V_{1}}\\\nonumber
&\qquad
+2\|X_{\varepsilon}(u)\|^{\alpha_{2}}_{V_{2}}
+2\|X_{\varepsilon}(u)\|^{\alpha_{2}-1}_{V_{2}}
\|X_{\varepsilon}(\sigma-\delta)\|_{V_{2}}
+2M_{0}\|X_{\varepsilon}(u)\|_{V_{2}}
+2M_{0}\|X_{\varepsilon}(\sigma-\delta)\|_{V_{2}}\\\nonumber
&\qquad
+\left(2L_{F}+L_{F}^{2}+2L_{G}^{2}+1\right)\|X_{\varepsilon}(u)\|^{2}
+2\|X_{\varepsilon}(\sigma-\delta)\|^{2}+4M_{0}^{2}\bigg]
\d u\\\nonumber
&\quad
+2\int_{\sigma-\delta}^{\sigma}\langle X_{\varepsilon}(u)
-X_{\varepsilon}(\sigma-\delta),G_{\varepsilon}(u,X_{\varepsilon}(u))
\d W(u)\rangle\\\nonumber
&
\leq\int_{\sigma-\delta}^{\sigma}
\bigg[4\|X_{\varepsilon}(u)\|^{\alpha_{1}}_{V_{1}}
+\frac{4}{\alpha_{1}}
\|X_{\varepsilon}(\sigma-\delta)\|^{\alpha_{1}}_{V_{1}}
+4\|X_{\varepsilon}(u)\|^{\alpha_{2}}_{V_{2}}
+\frac{4}{\alpha_{2}}
\|X_{\varepsilon}(\sigma-\delta)\|^{\alpha_{2}}_{V_{2}}
+\frac{4(\alpha_{1}-1)}{\alpha_{1}}M_{0}^{\frac{\alpha_{1}}
{\alpha_{1}-1}}\\\nonumber
&\qquad
+\frac{4(\alpha_{2}-1)}{\alpha_{2}}M_{0}^{\frac{\alpha_{2}}
{\alpha_{2}-1}}+\left(2L_{F}+L_{F}^{2}+2L_{G}^{2}+1\right)
\|X_{\varepsilon}(u)\|^{2}+2\|X_{\varepsilon}(\sigma-\delta)\|^{2}
+4M_{0}^{2}\bigg]\d u\\\nonumber
&\quad
+2\int_{\sigma-\delta}^{\sigma}\langle X_{\varepsilon}(u)
-X_{\varepsilon}(\sigma-\delta),G_{\varepsilon}(u,X_{\varepsilon}(u))
\d W(u)\rangle.\\\nonumber
\end{align}
Set $s^{k\delta}:=s+k\delta$ for all $s\in\R,k\geq0$.
Then we have
\begin{align}\label{orsteseq3}
I_{k}
&
:=E\int_{s^{k\delta}}^{s^{(k+1)\delta}}\|X_{\varepsilon}(\sigma)
-X_{\varepsilon}(\sigma-\delta)\|^{2}\d\sigma\\\nonumber
&
\leq E\int_{s^{k\delta}}^{s^{(k+1)\delta}}\bigg\{\int_{\sigma-\delta}^{\sigma}
\bigg[4\|X_{\varepsilon}(u)\|^{\alpha_{1}}_{V_{1}}
+\frac{4}{\alpha_{1}}
\|X_{\varepsilon}(\sigma-\delta)\|^{\alpha_{1}}_{V_{1}}
+4\|X_{\varepsilon}(u)\|^{\alpha_{2}}_{V_{2}}
+\frac{4}{\alpha_{2}}
\|X_{\varepsilon}(\sigma-\delta)\|^{\alpha_{2}}_{V_{2}}\\\nonumber
&\qquad
+\frac{4(\alpha_{1}-1)}{\alpha_{1}}M_{0}^{\frac{\alpha_{1}}
{\alpha_{1}-1}}
+\frac{4(\alpha_{2}-1)}{\alpha_{2}}M_{0}^{\frac{\alpha_{2}}
{\alpha_{2}-1}}+\left(2L_{F}+L_{F}^{2}+2L_{G}^{2}+1\right)
\|X_{\varepsilon}(u)\|^{2}\\\nonumber
&\qquad
+2\|X_{\varepsilon}(\sigma-\delta)\|^{2}+4M_{0}^{2}\bigg]\d u
+2\int_{\sigma-\delta}^{\sigma}\langle X_{\varepsilon}(u)
-X_{\varepsilon}(\sigma-\delta),G_{\varepsilon}(u,X_{\varepsilon}(u))
\d W(u)\rangle\bigg\}\d\sigma\\\nonumber
&
\leq E\int_{s^{(k-1)\delta}}^{s^{(k+1)\delta}}\int_{u}^{u+\delta}\Big[
4\|X_{\varepsilon}(u)\|^{\alpha_{1}}_{V_{1}}
+4\|X_{\varepsilon}(u)\|^{\alpha_{2}}_{V_{2}}
+C
+C\|X_{\varepsilon}(u)\|^{2}\Big]
\d\sigma\d u\\\nonumber
&\quad
+E\int_{s^{k\delta}}^{s^{(k+1)\delta}}\delta\left(\frac{4}{\alpha_{1}}
\|X_{\varepsilon}(\sigma-\delta)\|^{\alpha_{1}}_{V_{1}}
+\frac{4}{\alpha_{2}}
\|X_{\varepsilon}(\sigma-\delta)\|^{\alpha_{2}}_{V_{2}}
+2\|X_{\varepsilon}(\sigma-\delta)\|^{2}\right)\d\sigma
+I_{k}^{2}\\\nonumber
&
\leq E\int_{s^{(k-1)\delta}}^{s^{(k+1)\delta}}\delta\Big[
4\|X_{\varepsilon}(u)\|^{\alpha_{1}}_{V_{1}}
+4\|X_{\varepsilon}(u)\|^{\alpha_{2}}_{V_{2}}
+C+C\|X_{\varepsilon}(u)\|^{2}\Big]\d u\\\nonumber
&\quad
+E\int_{s^{k\delta}}^{s^{(k+1)\delta}}\delta\left(\frac{4}{\alpha_{1}}
\|X_{\varepsilon}(\sigma-\delta)\|^{\alpha_{1}}_{V_{1}}
+\frac{4}{\alpha_{2}}
\|X_{\varepsilon}(\sigma-\delta)\|^{\alpha_{2}}_{V_{2}}
+2\|X_{\varepsilon}(\sigma-\delta)\|^{2}\right)\d\sigma
+I_{k}^{2}\\\nonumber
&
\leq \delta C E
\int_{s^{(k-1)\delta}}^{s^{(k+1)\delta}}\left(
\|X_{\varepsilon}(u)\|^{\alpha_{1}}_{V_{1}}
+\|X_{\varepsilon}(u)\|^{\alpha_{2}}_{V_{2}}
+\|X_{\varepsilon}(u)\|^{2}+1\right)\d u+I_{k}^{2}.
\end{align}
Now we estimate $I_{k}^{2}$. In view of Burkholder-Davis-Gundy
inequality, (H2$'$), (H4) and Young's inequality, we obtain
\begin{align}\label{orsteseq4}
I_{k}^{2}
&
:=2E\int_{s^{k\delta}}^{s^{(k+1)\delta}}\int_{\sigma-\delta}^{\sigma}
\langle X_{\varepsilon}(u)-X_{\varepsilon}(\sigma-\delta),
G_{\varepsilon}(u,X_{\varepsilon}(u))
\d W(u)\rangle\d\sigma\\\nonumber
&
\leq6\int_{s^{k\delta}}^{s^{(k+1)\delta}}E\left(\int_{\sigma-\delta}^{\sigma}
\|G_{\varepsilon}(u,X_{\varepsilon}(u))\|_{L_{2}(U,H)}^{2}
\|X_{\varepsilon}(u)-X_{\varepsilon}(\sigma-\delta)\|^{2}\d u
\right)^{\frac{1}{2}}\d\sigma\\\nonumber
&
\leq6\int_{s^{k\delta}}^{s^{(k+1)\delta}}E\left(\int_{\sigma-\delta}^{\sigma}
\left(2L_{G}^{2}\|X_{\varepsilon}(u)\|^{2}+2M_{0}^{2}\right)
\|X_{\varepsilon}(u)-X_{\varepsilon}(\sigma-\delta)\|^{2}\d u
\right)^{\frac{1}{2}}\d\sigma\\\nonumber
&
\leq C\int_{s^{k\delta}}^{s^{(k+1)\delta}}
E\left[\left(\sup_{s\leq t\leq s+T}
\|X_\eps(t)\|^2+1\right)^{\frac{1}{2}}\left(
\int_{\sigma-\delta}^{\sigma}
\|X_{\eps}(u)
-X_{\eps}(\sigma-\delta)\|^{2}\d u\right)^{\frac{1}{2}}
\right]\d\sigma\\\nonumber
&
\leq\delta^{\frac{1}{2}}C
\left(E\sup_{s\leq t\leq s+T}
\|X_{\eps}(u)\|^{2}+1\right)^{\frac{1}{2}}
\left(\int_{s^{k\delta}}^{s^{(k+1)\delta}}
\int_{\sigma-\delta}^{\sigma}
E\left(\|X_{\eps}(u)\|^{2}
+\|X_{\eps}(\sigma-\delta)\|^{2}\right)\d u\d\sigma
\right)^{\frac{1}{2}}\\\nonumber
&
\leq\delta^{\frac{1}{2}}C
\left(E\|\zeta^\eps_s\|^{2}+1\right)
\left(
\int_{s^{k\delta}}^{s^{(k+1)\delta}}
\int_{\sigma-\delta}^{\sigma}E
\left(\|X_{\eps}(u)\|^{2}
+\|X_{\eps}(\sigma-\delta)\|^{2}\right)\d u\d\sigma
\right)^{\frac{1}{2}}\\\nonumber
&
\leq\delta C
\left(E\|\zeta^\eps_s\|^{2}+1\right)\left(E
\int_{s^{(k-1)\delta}}^{s^{(k+1)\delta}}
\|X_{\eps}(u)\|^{2}\d u\right)^{\frac{1}{2}}.
\end{align}
Therefore
\eqref{orsteseq3} and \eqref{orsteseq4} yield
\begin{align}\label{orsteseq5}
I_{k}
&
\leq \delta C
E\int_{s^{(k-1)\delta}}^{s^{(k+1)\delta}}
\left(\|X_{\varepsilon}(u)\|^{\alpha_{1}}_{V_{1}}
+\|X_{\varepsilon}(u)\|^{\alpha_{2}}_{V_{2}}
+\|X_{\varepsilon}(u)\|^{2}+1\right)\d u\\\nonumber
&\quad
+\delta C
\left(E\|\zeta^\eps_s\|^{2}+1\right)\left(E
\int_{s^{(k-1)\delta}}^{s^{(k+1)\delta}}
\|X_{\eps}(u)\|^{2}\d u\right)^{\frac{1}{2}}.
\end{align}
By \eqref{peq} and Remark \ref{FGunirm} we get
\begin{align}\label{orsteseq6}
2\sum_{k=1}^{T(\delta)-1}I_{k}
&
\leq\delta C E\int_{s}^{s+T}
\left(\|X_{\varepsilon}(u)\|^{\alpha_{1}}_{V_{1}}
+\|X_{\varepsilon}(u)\|^{\alpha_{2}}_{V_{2}}
+\|X_{\varepsilon}(u)\|^{2}+1\right)\d u\\\nonumber
&\quad
+\delta C
\left(E\|\zeta^\eps_s\|^{2}+1\right)\sum_{k=1}^{T(\delta)-1}
\left(E\int_{s^{(k-1)\delta}}^{s^{(k+1)\delta}}
\|X_{\eps}(u)\|^{2}\d u\right)^{\frac{1}{2}}\\\nonumber
&
\leq\delta C E\int_{s}^{s+T}
\left(\|X_{\varepsilon}(u)\|^{\alpha_{1}}_{V_{1}}
+\|X_{\varepsilon}(u)\|^{\alpha_{2}}_{V_{2}}
+\|X_{\varepsilon}(u)\|^{2}+1\right)\d u\\\nonumber
&\quad
+\delta^{\frac12}C
\left(E\|\zeta^\eps_s\|^{2}+1\right)\left(
\int_s^{s+T}E\|X_\eps(t)\|^2\d t\right)^{\frac12}\\\nonumber
&
\leq C_{T}\left(1+E\|\zeta_{s}^{\varepsilon}\|^{2}\right)
\delta^{\frac{1}{2}}.
\end{align}
Similarly, we have
\begin{align}\label{orsteseq7}
2\sum_{k=1}^{T(\delta)-1}J_{k}
&
\leq C_{T}\left(1+E\|\zeta_{s}^{\varepsilon}\|^{2}\right)
\delta^{\frac{1}{2}}.
\end{align}
Combining \eqref{orsteseq1}, \eqref{orsteseq6} and \eqref{orsteseq7},
we obtain
\begin{equation*}
E\int_{s}^{s+T}\|X_{\varepsilon}(\sigma)-\tilde{X}_{\varepsilon}
(\sigma)\|^{2}\d\sigma\leq C_{T}(1
+E\|\zeta_{s}^{\varepsilon}\|^{2})\delta^{\frac{1}{2}}.
\end{equation*}

It follows from the same steps as in the proof of
\eqref{orsteseq} that
\begin{equation*}
E\int_{s}^{s+T}\|\bar{X}(\sigma)-\tilde{X}(\sigma)\|^{2}\d\sigma
\leq C_{T}(1+E\|\zeta_{s}\|^{2})\delta^{\frac{1}{2}}.
\end{equation*}
\end{proof}

Now we establish the following first Bogolyubov theorem.
\begin{theorem}\label{avethf}
Suppose that {\rm{(G1)--(G2)}}, {\rm{(H1)}}, {\rm{(H2$'$)}} and
{\rm{(H3)--(H4)}} hold. For any $s\in\R$, let $X_{\varepsilon}(t,s,\zeta_s^\varepsilon),t\geq s$
be the solution of the following Cauchy problem
\begin{equation*}
 \left\{
   \begin{aligned}
   &\ \d X(t)=\left(A(X(t))+F_{\varepsilon}(t,X(t))\right)\d t
   +G_{\varepsilon}(t,X(t))\d W(t)\\
  &\ X(s)=\zeta_{s}^{\varepsilon}
   \end{aligned}
   \right.
\end{equation*}
and $\bar{X}(t,s,\zeta_s),t\geq s$ be the solution of the following Cauchy problem
\begin{equation*}
 \left\{
   \begin{aligned}
   &\ \d X(t)=\left(A(X(t))+\bar{F}(X(t))\right)\d t
   +\bar{G}(X(t))\d W(t)\\
  &\ X(s)=\zeta_{s}.
   \end{aligned}
   \right.
\end{equation*}
Assume further that
$\lim\limits_{\varepsilon\rightarrow0}
 E\|\zeta^{\varepsilon}_{s}-\zeta_{s}\|^{2}=0$.
Then
$$
\lim_{\varepsilon\rightarrow0} E\sup_{s\leq t\leq s+T}
\|X_{\varepsilon}(t,s,\zeta_s^\varepsilon)-\bar{X}(t,s,\zeta_s)\|^{2}=0
$$
for any $T>0$.
\end{theorem}
\begin{proof}
Set
$X_{\varepsilon}(\sigma):=X_{\varepsilon}(\sigma,s,\zeta_s^\varepsilon)$
and $\bar{X}(\sigma):=\bar{X}(\sigma,s,\zeta_s)$ for all $\sigma\geq s$.
In view of It\^o's formula and (H2$'$), we have
\begin{align}\label{avethfeq1}
&
\|X_{\varepsilon}(t)-\bar{X}(t)\|^{2}\\\nonumber
&
=\|\zeta^{\varepsilon}_{s}-\zeta_{s}\|^{2}+\int_{s}^{t}
\bigg(2_{V^{*}}\langle
A(X_{\varepsilon}(\sigma))-A(\bar{X}(\sigma)),X_{\varepsilon}(\sigma)
-\bar{X}(\sigma)\rangle_{V}\\\nonumber
&\qquad
+2\langle F_{\varepsilon}(\sigma,X_{\varepsilon}(\sigma))
-\bar{F}(\bar{X}(\sigma)),X_{\varepsilon}(\sigma)-\bar{X}(\sigma)\rangle
+\|G_{\varepsilon}(\sigma,X_{\varepsilon}(\sigma))
-\bar{G}(\bar{X}(\sigma))\|_{L_{2}(U,H)}^{2}\bigg)\d\sigma\\\nonumber
&\quad
+2\int_{s}^{t}\langle X_{\varepsilon}(\sigma)-\bar{X}(\sigma),\left(
G_{\varepsilon}(\sigma,X_{\varepsilon}(\sigma))
-\bar{G}(\bar{X}(\sigma))\right)\d W(\sigma)\rangle\\\nonumber
&
\leq\|\zeta^{\varepsilon}_{s}-\zeta_{s}\|^{2}
+2\int_{s}^{t}\langle X_{\varepsilon}(\sigma)-\bar{X}(\sigma),\left(
G_{\varepsilon}(\sigma,X_{\varepsilon}(\sigma))
-\bar{G}(\bar{X}(\sigma))\right)\d W(\sigma)\rangle\\\nonumber
&\quad
+\int_{s}^{t}\bigg(2
\langle F_{\varepsilon}(\sigma,X_{\varepsilon}(\sigma))
-\bar{F}(\bar{X}(\sigma)),X_{\varepsilon}(\sigma)-\bar{X}(\sigma)
\rangle
+\|G_{\varepsilon}(\sigma,X_{\varepsilon}(\sigma))
-\bar{G}(\bar{X}(\sigma))\|_{L_{2}(U,H)}^{2}\bigg)\d\sigma.
\end{align}
Therefore, by Burkholder-Davis-Gundy inequality and
Young's inequality we get
\begin{align}\label{avethfeq2}
&
E\sup_{s\leq t\leq s+T}
\|X_{\varepsilon}(t)-\bar{X}(t)\|^{2}\\\nonumber
&
\leq E\|\zeta^{\varepsilon}_{s}-\zeta_{s}\|^{2}
+E\sup_{s\leq t\leq s+T}\int_{s}^{t}2\langle
F_{\varepsilon}(\sigma,X_{\varepsilon}(\sigma))
-\bar{F}(\bar{X}(\sigma)),X_{\varepsilon}(\sigma)-\bar{X}(\sigma)
\rangle\d\sigma\\\nonumber
&\quad
+E\int_{s}^{s+T}\|G_{\varepsilon}(\sigma,X_{\varepsilon}(\sigma))
-\bar{G}(\bar{X}(\sigma))\|_{L_{2}(U,H)}^{2}\d\sigma\\\nonumber
&\quad
+6E\left(\int_{s}^{s+T}\|G_{\varepsilon}(\sigma,X_{\varepsilon}(\sigma))
-\bar{G}(\bar{X}(\sigma))\|_{L_{2}(U,H)}^{2}
\|X_{\varepsilon}(\sigma)-\bar{X}(\sigma)\|^{2}\d\sigma
\right)^{\frac{1}{2}}\\\nonumber
&
\leq E\|\zeta^{\varepsilon}_{s}-\zeta_{s}\|^{2}
+E\sup_{s\leq t\leq s+T}\int_{s}^{t}2\langle
F_{\varepsilon}(\sigma,X_{\varepsilon}(\sigma))
-\bar{F}(\bar{X}(\sigma)),X_{\varepsilon}(\sigma)-\bar{X}(\sigma)
\rangle\d\sigma\\\nonumber
&\quad
+\frac{1}{2}E\sup_{s\leq t\leq s+T}
\|X_{\varepsilon}(t)-\bar{X}(t)\|^{2}+19E\int_{s}^{s+T}\|
G_{\varepsilon}(\sigma,X_{\varepsilon}(\sigma))
-\bar{G}(\bar{X}(\sigma))\|_{L_{2}(U,H)}^{2}\d\sigma.
\end{align}
Then we obtain
\begin{align}\label{avethfeq3}
&
E\sup_{s\leq t\leq s+T}
\|X_{\varepsilon}(t)-\bar{X}(t)\|^{2}\\\nonumber
&
\leq2E\|\zeta^{\varepsilon}_{s}-\zeta_{s}\|^{2}
+4E\sup_{s\leq t\leq s+T}\int_{s}^{t}\langle
F_{\varepsilon}(\sigma,X_{\varepsilon}(\sigma))
-\bar{F}(\bar{X}(\sigma)),X_{\varepsilon}(\sigma)-\bar{X}(\sigma)
\rangle\d\sigma\\\nonumber
&\quad
+38E\int_{s}^{s+T}\|G_{\varepsilon}(\sigma,X_{\varepsilon}(\sigma))
-\bar{G}(\bar{X}(\sigma))\|_{L_{2}(U,H)}^{2}\d\sigma\\\nonumber
&
=:2E\|\zeta^{\varepsilon}_{s}-\zeta_{s}\|^{2}
+\mathbb I_{1}+\mathbb I_{2}.
\end{align}

For $\mathbb I_{1}$,
\begin{align}\label{avethfeq4}
\mathbb I_{1}
&
:=4E\sup_{s\leq t\leq s+T}\int_{s}^{t}\langle
F_{\varepsilon}(\sigma,X_{\varepsilon}(\sigma))
-\bar{F}(\bar{X}(\sigma)),X_{\varepsilon}(\sigma)-\bar{X}(\sigma)
\rangle\d\sigma\\\nonumber
&
\leq4E\int_{s}^{s+T}\|F_{\varepsilon}(\sigma,X_{\varepsilon}(\sigma))
-F_{\varepsilon}(\sigma,\bar{X}(\sigma))\|\|X_{\varepsilon}(\sigma)
-\bar{X}(\sigma)\|\d\sigma\\\nonumber
&\quad
+4E\sup_{s\leq t\leq s+T}\int_{s}^{t}\langle
F_{\varepsilon}(\sigma,\bar{X}(\sigma))-\bar{F}(\bar{X}(\sigma)),
X_{\varepsilon}(\sigma)-\bar{X}(\sigma)\rangle\d\sigma\\\nonumber
&
\leq4E\int_{s}^{s+T}L_{F}\|X_{\varepsilon}(\sigma)-\bar{X}(\sigma)\|^{2}
\d\sigma\\\nonumber
&\quad
+4E\sup_{s\leq t\leq s+T}\int_{s}^{t}\langle
F_{\varepsilon}(\sigma,\bar{X}(\sigma))-\bar{F}(\bar{X}(\sigma)),
X_{\varepsilon}(\sigma)-\tilde{X}_{\varepsilon}(\sigma)
\rangle\d\sigma\\\nonumber
&\quad
+4E\sup_{s\leq t\leq s+T}\int_{s}^{t}\langle
F_{\varepsilon}(\sigma,\bar{X}(\sigma))-\bar{F}(\bar{X}(\sigma)),
\tilde{X}_{\varepsilon}(\sigma)-\tilde{X}(\sigma)
\rangle\d\sigma\\\nonumber
&\quad
+4E\sup_{s\leq t\leq s+T}\int_{s}^{t}\langle
F_{\varepsilon}(\sigma,\bar{X}(\sigma))-\bar{F}(\bar{X}(\sigma)),
\tilde{X}(\sigma)-\bar{X}(\sigma)\rangle\d\sigma\\\nonumber
&
\leq4E\int_{s}^{s+T}L_{F}\|X_{\varepsilon}(\sigma)-\bar{X}(\sigma)\|^{2}
\d\sigma+\mathbb I_{1}^{2}+\mathbb I_{1}^{3}
+\mathbb I_{1}^{4}.\\\nonumber
\end{align}
For $\mathbb I_{1}^{2}$, by (H2$'$), H\"older's inequality,
\eqref{peq}, Remark \ref{FGunirm} and \eqref{orsteseq}  we have
\begin{align}\label{avethfeq5}
\mathbb I_{1}^{2}
&
:=4E\sup_{s\leq t\leq s+T}\int_{s}^{t}\langle
F_{\varepsilon}(\sigma,\bar{X}(\sigma))-\bar{F}(\bar{X}(\sigma)),
X_{\varepsilon}(\sigma)-\tilde{X}_{\varepsilon}(\sigma)
\rangle\d\sigma\\\nonumber
&
\leq4E\int_{s}^{s+T}\|F_{\varepsilon}(\sigma,\bar{X}(\sigma))
-\bar{F}(\bar{X}(\sigma))\|\|X_{\varepsilon}(\sigma)
-\tilde{X}_{\varepsilon}(\sigma)\|\d\sigma\\\nonumber
&
\leq4\left(E\int_{s}^{s+T}\left(2L_{F}\|\bar{X}(\sigma)\|
+2M_{0}\right)^{2}\d\sigma\right)^{\frac{1}{2}}\left(E\int_{s}^{s+T}
\|X_{\varepsilon}(\sigma)-\tilde{X}_{\varepsilon}(\sigma)\|^{2}
\d\sigma\right)^{\frac{1}{2}}\\\nonumber
&
\leq C_{T}\left(1+E\|\zeta_{s}\|^{2}\right)\delta^{\frac{1}{4}}.
\end{align}
For $\mathbb I_{1}^{4}$,
similar to $\mathbb I_{1}^{2}$, employing (H2$'$), H\"older's inequality,
\eqref{peq}, Remark \ref{FGunirm} and \eqref{avesteseq} we get
\begin{align}\label{avethfeq6}
\mathbb I_{1}^{4}
&
:=4E\sup_{s\leq t\leq s+T}\int_{s}^{t}\langle
F_{\varepsilon}(\sigma,\bar{X}(\sigma))-\bar{F}(\bar{X}(\sigma)),
\tilde{X}(\sigma)-\bar{X}(\sigma)\rangle\d\sigma\\\nonumber
&
\leq C_{T}\left(1+E\|\zeta_{s}\|^{2}\right)\delta^{\frac{1}{4}}.
\end{align}
For $\mathbb I_{1}^{3}$, in view of (H2$'$), H\"older's inequality,
\eqref{peq} and Remark \ref{FGunirm}, we have
\begin{align}\label{avethfeq7}
\mathbb I_{1}^{3}
&
:=4E\sup_{s\leq t\leq s+T}\int_{s}^{t}\langle
F_{\varepsilon}(\sigma,\bar{X}(\sigma))-\bar{F}(\bar{X}(\sigma)),
\tilde{X}_{\varepsilon}(\sigma)-\tilde{X}(\sigma)
\rangle\d\sigma\\\nonumber
&
=4E\sup_{s\leq t\leq s+T}\int_{s}^{t}\langle
F_{\varepsilon}(\sigma,\bar{X}(\sigma))
-F_{\varepsilon}(\sigma,\tilde{X}(\sigma)),
\tilde{X}_{\varepsilon}(\sigma)-\tilde{X}(\sigma)
\rangle\d\sigma\\\nonumber
&\quad
+4E\sup_{s\leq t\leq s+T}\int_{s}^{t}\langle
F_{\varepsilon}(\sigma,\tilde{X}(\sigma))-\bar{F}(\tilde{X}(\sigma)),
\tilde{X}_{\varepsilon}(\sigma)-\tilde{X}(\sigma)
\rangle\d\sigma\\\nonumber
&\quad
+4E\sup_{s\leq t\leq s+T}\int_{s}^{t}\langle
\bar{F}(\tilde{X}(\sigma))-\bar{F}(\bar{X}(\sigma)),
\tilde{X}_{\varepsilon}(\sigma)-\tilde{X}(\sigma)
\rangle\d\sigma\\\nonumber
&
\leq4E\int_{s}^{s+T}\|F_{\varepsilon}(\sigma,\bar{X}(\sigma))
-F_{\varepsilon}(\sigma,\tilde{X}(\sigma))\|
\|\tilde{X}_{\varepsilon}(\sigma)-\tilde{X}(\sigma)\|\d\sigma\\\nonumber
&\quad
+4E\sup_{s\leq t\leq s+T}\int_{s}^{t}\langle
F_{\varepsilon}(\sigma,\tilde{X}(\sigma))-\bar{F}(\tilde{X}(\sigma)),
\tilde{X}_{\varepsilon}(\sigma)-\tilde{X}(\sigma)
\rangle\d\sigma\\\nonumber
&\quad
+4E\int_{s}^{s+T}\|\bar{F}(\tilde{X}(\sigma))-\bar{F}(\bar{X}(\sigma))\|
\|\tilde{X}_{\varepsilon}(\sigma)-\tilde{X}(\sigma)\|\d\sigma\\\nonumber
&
\leq8E\int_{s}^{s+T}L_{F}\|\tilde{X}(\sigma)-\bar{X}(\sigma)\|
\|\tilde{X}_{\varepsilon}(\sigma)-\tilde{X}(\sigma)\|\d\sigma\\\nonumber
&\quad
+4E\sup_{s\leq t\leq s+T}\int_{s}^{t}\langle
F_{\varepsilon}(\sigma,\tilde{X}(\sigma))-\bar{F}(\tilde{X}(\sigma)),
\tilde{X}_{\varepsilon}(\sigma)-\tilde{X}(\sigma)
\rangle\d\sigma\\\nonumber
&
\leq8L_{F}\left(E\int_{s}^{s+T}\|\tilde{X}(\sigma)-\bar{X}(\sigma)\|^{2}
\d\sigma\right)^{\frac{1}{2}}\left(E\int_{s}^{s+T}
\|\tilde{X}_{\varepsilon}(\sigma)-\tilde{X}(\sigma)\|^{2}\d\sigma
\right)^{\frac{1}{2}}\\\nonumber
&\quad
+4E\sup_{s\leq t\leq s+T}\int_{s}^{t}\langle
F_{\varepsilon}(\sigma,\tilde{X}(\sigma))-\bar{F}(\tilde{X}(\sigma)),
\tilde{X}_{\varepsilon}(\sigma)-\tilde{X}(\sigma)
\rangle\d\sigma\\\nonumber
&
\leq C_{T}\left(1+E\|\zeta_{s}\|^{2}\right)\delta^{\frac{1}{4}}
+\mathbb I_{1}^{3,2}.
\end{align}

Define $t(s,\delta):=s+\left[\frac{t-s}{\delta}\right]\delta$,
where $\left[\frac{t-s}{\delta}\right]$ is the integer part of
$\frac{t-s}{\delta}$. Now we estimate
$\mathbb I_{1}^{3,2}:=4E\sup\limits_{s\leq t\leq s+T}
\int_{s}^{t}\langle F_{\varepsilon}(\sigma,\tilde{X}(\sigma))
-\bar{F}(\tilde{X}(\sigma)),
\tilde{X}_{\varepsilon}(\sigma)-\tilde{X}(\sigma)\rangle\d\sigma$ by
(H2$'$), (G1), \eqref{peq} and Remark \ref{FGunirm}:
\begin{align}\label{avethfeq8}
&
4E\sup_{s\leq t\leq s+T}\int_{s}^{t}\langle
F_{\varepsilon}(\sigma,\tilde{X}(\sigma))-\bar{F}(\tilde{X}(\sigma)),
\tilde{X}_{\varepsilon}(\sigma)-\tilde{X}(\sigma)
\rangle\d\sigma\\\nonumber
&
=4E\sup_{s\leq t\leq s+T}\Bigg\{\sum_{k=0}^{\left[\frac{t-s}{\delta}\right]-1}
\int_{s^{k\delta}}^{s^{(k+1)\delta}}\langle
F_{\varepsilon}(\sigma,\bar{X}(s^{k\delta}))-\bar{F}(\bar{X}(s^{k\delta})),
X_{\varepsilon}(s^{k\delta})-\bar{X}(s^{k\delta})
\rangle\d\sigma\\\nonumber
&\quad
+\int_{t(s,\delta)}^{t}\langle
F_{\varepsilon}(\sigma,\bar{X}(t(s,\delta)))
-\bar{F}(\bar{X}(t(s,\delta))),
X_{\varepsilon}(t(s,\delta))
-\bar{X}(t(s,\delta))
\rangle\d\sigma\Bigg\}\\\nonumber
&
\leq4E\sup_{s\leq t\leq s+T}\Bigg\{
\sum_{k=0}^{\left[\frac{t-s}{\delta}\right]-1}\left\langle
\int_{s^{k\delta}}^{s^{(k+1)\delta}}\left(F_{\varepsilon}(\sigma,
\bar{X}(s^{k\delta}))-\bar{F}(\bar{X}(s^{k\delta}))\right)\d\sigma,
X_{\varepsilon}(s^{k\delta})-\bar{X}(s^{k\delta})\right\rangle\\\nonumber
&\quad
+\int_{t(s,\delta)}^{t}\|
F_{\varepsilon}(\sigma,\bar{X}(t(s,\delta)))
-\bar{F}(\bar{X}(t(s,\delta)))\|
\|X_{\varepsilon}(t(s,\delta))
-\bar{X}(t(s,\delta))\|\d\sigma\Bigg\}\\\nonumber
&
\leq4E\sup_{s\leq t\leq s+T}\Bigg\{\sum_{k=0}
^{\left[\frac{t-s}{\delta}\right]-1}\left\|\int_{s^{k\delta}}^{s^{(k+1)\delta}}
\left(F_{\varepsilon}(\sigma,\bar{X}(s^{k\delta}))
-\bar{F}(\bar{X}(s^{k\delta}))\right)\d\sigma\right\|\\\nonumber
&\qquad
\times
\|X_{\varepsilon}(s^{k\delta})-\bar{X}(s^{k\delta})\|\Bigg\}
+C_{T}\left(1+E\|\zeta_{s}\|^{2}\right)\delta\\\nonumber
&
\leq\frac{4T}{\delta}\max_{0\leq k\leq T(\delta)-1}\left(E
\left\|\int_{s^{k\delta}}^{s^{(k+1)\delta}}\left(
F_{\varepsilon}(\sigma,\bar{X}(s^{k\delta}))
-\bar{F}(\bar{X}(s^{k\delta}))\right)\d\sigma\right\|^{2}\right)^{\frac{1}{2}}
C_{T}\left(1+E\|\zeta_{s}\|^{2}\right)\\\nonumber
&\quad
+C_{T}\left(1+E\|\zeta_{s}\|^{2}\right)\delta\\\nonumber
&
\leq\frac{4T}{\delta}\max_{0\leq k\leq T(\delta)-1}\delta
\omega_{1}\left(\frac{\delta}{\varepsilon}\right)\left(E\left(1
+\|\bar{X}(s^{k\delta})\|\right)^{2}\right)^{\frac{1}{2}}
C_{T}\left(1+E\|\zeta_{s}\|^{2}\right)
+C_{T}\left(1+E\|\zeta_{s}\|^{2}\right)\delta\\\nonumber
&
\leq C_{T}\left(1+E\|\zeta_{s}\|^{2}\right)\left(
\omega_{1}\left(\frac{\delta}{\varepsilon}\right)+\delta\right).
\end{align}
Combining \eqref{avethfeq7} and \eqref{avethfeq8}, we deduce
\begin{equation}\label{avethfeq9}
\mathbb I_{1}^{3}\leq C_{T}\left(1+E\|\zeta_{s}\|^{2}\right)\left(
\delta^{\frac{1}{4}}+\delta+
\omega_{1}\left(\frac{\delta}{\varepsilon}\right)\right).
\end{equation}
Therefore, \eqref{avethfeq4}--\eqref{avethfeq6} and \eqref{avethfeq9}
yield
\begin{equation}\label{avethfeq10}
\mathbb I_{1}
\leq4L_{F}\int_{s}^{s+T}E\sup_{s\leq u\leq\sigma}\|X_{\varepsilon}(u)
-\bar{X}(u)\|^{2}\d\sigma+C_{T}\left(1+E\|\zeta_{s}\|^{2}\right)
\left(\delta^{\frac{1}{4}}+\delta+
\omega_{1}\left(\frac{\delta}{\varepsilon}\right)\right).
\end{equation}

Now we estimate $\mathbb I_{2}$.
\begin{align}\label{avethfeq11}
\mathbb I_{2}
&
:=38E\int_{s}^{s+T}\|G_{\varepsilon}(\sigma,X_{\varepsilon}(\sigma))
-\bar{G}(\bar{X}(\sigma))\|_{L_{2}(U,H)}^{2}\d\sigma\\\nonumber
&
\leq76E\int_{s}^{s+T}\|G_{\varepsilon}(\sigma,X_{\varepsilon}(\sigma))
-G_{\varepsilon}(\sigma,\bar{X}(\sigma))\|_{L_{2}(U,H)}^{2}
\d\sigma\\\nonumber
&\quad
+76E\int_{s}^{s+T}\|G_{\varepsilon}(\sigma,\bar{X}(\sigma))
-\bar{G}(\bar{X}(\sigma))\|_{L_{2}(U,H)}^{2}\d\sigma\\\nonumber
&
\leq76L_{G}^{2}E\int_{s}^{s+T}\|X_{\varepsilon}(\sigma)
-\bar{X}(\sigma)\|^{2}\d\sigma\\\nonumber
&\quad
+76E\int_{s}^{s+T}\|G_{\varepsilon}(\sigma,\bar{X}(\sigma))
-\bar{G}(\bar{X}(\sigma))\|_{L_{2}(U,H)}^{2}\d\sigma\\\nonumber
&
\leq76L_{G}^{2}\int_{s}^{s+T}E\sup_{s\leq u\leq\sigma}
\|X_{\varepsilon}(u)-\bar{X}(u)\|^{2}\d\sigma+\mathbb I_{2}^{2}.
\end{align}
Denote by $T(\delta):=\left[\frac{T}{\delta}\right]$.
For $\mathbb I_{2}^{2}$, it follows from (H2$'$), (G1),
\eqref{avesteseq}, \eqref{peq} and Remark \ref{FGunirm} that
\begin{align}\label{avethfeq12}
\mathbb I_{2}^{2}
&
:=76E\int_{s}^{s+T}\|G_{\varepsilon}(\sigma,\bar{X}(\sigma))
-\bar{G}(\bar{X}(\sigma))\|_{L_{2}(U,H)}^{2}\d\sigma\\\nonumber
&
\leq228E\int_{s}^{s+T}\|G_{\varepsilon}(\sigma,\bar{X}(\sigma))
-G_{\varepsilon}(\sigma,\tilde{X}(\sigma))\|_{L_{2}(U,H)}^{2}
\d\sigma\\\nonumber
&\quad
+228E\int_{s}^{s+T}\|G_{\varepsilon}(\sigma,\tilde{X}(\sigma))
-\bar{G}(\tilde{X}(\sigma))\|_{L_{2}(U,H)}^{2}\d\sigma\\\nonumber
&\quad
+228E\int_{s}^{s+T}\|\bar{G}(\tilde{X}(\sigma))
-\bar{G}(\bar{X}(\sigma))\|_{L_{2}(U,H)}^{2}\d\sigma\\\nonumber
&
\leq456E\int_{s}^{s+T}L_{G}^{2}\|\bar{X}(\sigma)-\tilde{X}(\sigma)\|^{2}
\d\sigma+228E\int_{s}^{s+T}\|G_{\varepsilon}(\sigma,\tilde{X}(\sigma))
-\bar{G}(\tilde{X}(\sigma))\|_{L_{2}(U,H)}^{2}\d\sigma\\\nonumber
&
\leq C_{T}\left(1+E\|\zeta_{s}\|^{2}\right)\delta^{\frac{1}{2}}
+228E\sum_{k=0}^{T(\delta)-1}\int_{s^{k\delta}}^{s^{(k+1)\delta}}
\|G_{\varepsilon}(\sigma,\bar{X}(s^{k\delta})
-\bar{G}(\bar{X}(s^{k\delta})\|_{L_{2}(U,H)}^{2}\d\sigma\\\nonumber
&\quad
+228E\int_{s+T(\delta)\delta}^{s+T}
\|G_{\varepsilon}(\sigma,\bar{X}(s+T(\delta)\delta))
-\bar{G}(\bar{X}(s+T(\delta)\delta))\|_{L_{2}(U,H)}^{2}\d\sigma\\\nonumber
&
\leq C_{T}\left(1+E\|\zeta_{s}\|^{2}\right)\delta^{\frac{1}{2}}+228
\sum_{k=0}^{T(\delta)-1}\delta\omega_{2}\left(\frac{\delta}{\varepsilon}
\right)E\left(1+\|\bar{X}(s+k\delta)\|^{2}\right)\\\nonumber
&\quad
+CE\int_{s+T(\delta)\delta}^{s+T}\left(
\|\bar{X}(s+T(\delta)\delta)\|^{2}+1\right)\d\sigma\\\nonumber
&
\leq C_{T}\left(1+E\|\zeta_{s}\|^{2}\right)\left(\delta^{\frac{1}{2}}
+\delta+\omega_{2}\left(\frac{\delta}{\varepsilon}\right)\right).
\end{align}
Therefore,
\begin{align}\label{avethfeq13}
\mathbb I_{2}\leq 76L_{G}^{2}\int_{s}^{s+T}E\sup_{s\leq u\leq\sigma}
\|X_{\varepsilon}(u)-\bar{X}(u)\|^{2}\d\sigma
+C_{T}\left(1+E\|\zeta_{s}\|^{2}\right)\left(\delta^{\frac{1}{2}}
+\delta+\omega_{2}\left(\frac{\delta}{\varepsilon}\right)\right).
\end{align}

Combining \eqref{avethfeq3}, \eqref{avethfeq10} and
\eqref{avethfeq13}, we get
\begin{align}\label{avethfeq14}
&
E\sup_{s\leq t\leq s+T}
\|X_{\varepsilon}(t)-\bar{X}(t)\|^{2}\\\nonumber
&
\leq2E\|\zeta^{\varepsilon}_{s}-\zeta_{s}\|^{2}
+\left(4L_{F}+76L_{G}^{2}\right)\int_{s}^{s+T}
E\sup_{s\leq u\leq \sigma}\|X_{\varepsilon}(u)-\bar{X}(u)\|^{2}\d
\sigma\\\nonumber
&\quad
+C_{T}\left(1+E\|\zeta_{s}\|^{2}\right)\left(\delta^{\frac{1}{4}}
+\omega_{1}\left(\frac{\delta}{\varepsilon}\right)
+\omega_{2}\left(\frac{\delta}{\varepsilon}\right)\right).
\end{align}
It follows from Gronwall's lemma that
\begin{align}\label{avetheq015}
E\sup_{s\leq t\leq s+T}
\|X_{\varepsilon}(t)-\bar{X}(t)\|^{2}
&
\leq\bigg[2E\|\zeta^{\varepsilon}_{s}-\zeta_{s}\|^{2}
+C_{T}\left(1+E\|\zeta_{s}\|^{2}\right)\Big(\delta^{\frac{1}{4}}
+\omega_{1}\left(\frac{\delta}{\varepsilon}\right)\\\nonumber
&\qquad
+\omega_{2}\left(\frac{\delta}{\varepsilon}\right)\Big)\bigg]
\exp\left\{\left(4L_{F}+76L_{G}^{2}\right)T\right\},
\end{align}
which implies
\begin{align}\label{avethfeq15}
E\sup_{s\leq t\leq s+T}
\|X_{\varepsilon}(t)-\bar{X}(t)\|^{2}\leq C_T
\left(E\|\zeta^{\varepsilon}_{s}-\zeta_{s}\|^{2}
+\varepsilon^{\frac18}
+\omega_{1}\left(\frac{1}{\sqrt{\varepsilon}}\right)
+\omega_{2}\left(\frac{1}{\sqrt{\varepsilon}}\right)\right)
\end{align}
provided $\delta=\sqrt{\varepsilon}$.
Letting
$\varepsilon\rightarrow0$, we obtain
\[
\lim_{\varepsilon\rightarrow0}E\sup_{s\leq t\leq s+T}
\|X_{\varepsilon}(t)-\bar{X}(t)\|^{2}=0.
\]
\end{proof}

\begin{remark}\rm
\begin{enumerate}
  \item For simplicity, we take $\delta=\sqrt{\eps}$ in
   \eqref{avetheq015} to obtain \eqref{avethfeq15},
  which gives a convergence rate for the first Bogolyubov theorem.
  If we take $\delta=\psi(\eps)$ satisfying $\psi(\eps)\rightarrow0$
  and $\frac{\psi(\eps)}{\eps}\rightarrow\infty$ as $\eps\rightarrow0$ such that
  $\delta^{\frac14}=\omega_1(\frac{\delta}{\eps})=\omega_2(\frac{\delta}{\eps})$,
  then we obtain a better convergence rate. But we are not sure if our method can give
  the optimal rate.
  \item As mentioned in Introduction, there are three types of
averaging principle and most existing works (except for \cite{CL2020,KMF_2015})
on {\em stochastic} averaging focus on the first Bogolyubov theorem.
But to the best of our knowledge, the above result is new and hence interesting on its own rights.
Meanwhile, it is helpful for us to establish the second
Bogolyubov theorem and global averaging principle in what follows.
\end{enumerate}
\end{remark}

With the help of Theorem \ref{avethf}, we can now establish
the second Bogolyubov theorem.
\begin{theorem}\label{averth}
Suppose that conditions {\rm{(G1)--(G2)}}, {\rm{(H1)}},
{\rm{(H2$'$)}} and {\rm{(H3)--(H6)}} hold. Assume further that
$2\lambda-2\lambda_{F}-L_{G}^{2}\geq0$. If $\lambda'>0$
or $2\lambda-2\lambda_{F}-L_{G}^{2}>0$
then for any $0<\varepsilon \le 1$
\begin{enumerate}
\item equation \eqref{eqG2.1} has a unique solution
      $X_{\varepsilon}\in C_{b}(\mathbb R,L^{2}
      (\Omega,\mathbb P;H))$;
\item the $L^{2}$-bounded solution $X_{\varepsilon}$ of
      \eqref{eqG2.1} is strongly compatible in distribution, i.e.
      $\mathfrak M_{(F_{\varepsilon},G_{\varepsilon})}\subseteq
      \tilde{\mathfrak M}_{X_{\varepsilon}}$, and
      \[
      \lim_{\varepsilon\rightarrow0}d_{BL}(\mathcal L
      (X_{\varepsilon}),\mathcal L(\bar{X}))=0 \qquad
      {\rm in}~Pr(C(\R,H)),
      \]
      where $\bar{X}$ is the unique stationary solution of
      averaged equation \eqref{eqG5_1}.
\end{enumerate}
\end{theorem}
\begin{proof}
(i) follows from Theorem \ref{Boundedth}.

(ii)
By Theorem \ref{SCth} the bounded solution $X_{\varepsilon}$ of
equation \eqref{eqG2.1} is strongly compatible in distribution,
i.e. $\mathfrak M_{(F_{\varepsilon},G_{\varepsilon})}\subseteq
\tilde{\mathfrak M}_{X_{\varepsilon}}$, for any
$0<\varepsilon\leq1$.

Now we prove that
$\lim\limits_{\varepsilon\rightarrow0}d_{BL}(\L(X_{\varepsilon}(t)),
\L(\bar{X}(t)))=0$ in $Pr(H)$ for any $t\in\R$. Take a sequence
$\{\varepsilon_{n}\}_{n=1}^{\infty}\subset(0,1]$ such that
$\varepsilon_{n}\rightarrow0$ as $n\rightarrow\infty$.
Similar to Proposition \ref{tightprop}, we have
$\sup\limits_{t\in\R,\varepsilon\in(0,1]} E
\|X_{\varepsilon}(t)\|^{2}_{S}<\infty$. It follows from Chebychev's
inequality and the compact imbedding $S\subset H$ that
$\{\mathcal L(X_{\varepsilon_{n}}(t))\}_{n=1}^{\infty}$
is tight for all $t\in\R$. For every $r\geq1$,
according to the tightness of
$\{\mathcal L(X_{\varepsilon_{n}}(-r))\}_{n=1}^{\infty}$,
there exists a subsequence
$\{\varepsilon_{n_{k}}\}\subset\{\varepsilon_{n}\}$ such that
$\mathcal L\left(X_{\varepsilon_{n_{k}}}(-r)\right)$
weakly converges to $\mu_{r}$
in $Pr(H)$. Due to the Skorohod representation theorem,
there exists a sequence of random
variables $\hat{\psi}^{k}(-r)$ and $\hat{\zeta_{r}}$
with laws of $\L\left(X_{\varepsilon_{n_{k}}}(-r)\right)$ and $\mu_{r}$
respectively, defined on another probability space
$(\hat{\Omega},\hat{\mathcal F},\hat{\mathbb P})$, such that
$\hat{\psi}^{k}(-r)\rightarrow\hat{\zeta}_{r}$
$\hat{\mathbb P}-{\rm{a.s.}}$
It follows from \eqref{ineq2} and Remark \ref{FGunirm}
that there exists $p>1$ such that
\begin{equation*}
\hat{ E}\left\|\hat{\psi}^{k}(-r)\right\|^{2p}
=\int_{H}\|x\|^{2p}\mathcal L(\hat{\psi}^{k}(-r))(\d x)
=\int_{H}\|x\|^{2p}
\mathcal L(X_{\varepsilon_{n_{k}}}(-r))(\d x)<\infty.
\end{equation*}
By the Vitali $L^{P}$ convergence criterion, we have
$\lim\limits_{k\rightarrow\infty}\hat{E}
\left\|\hat{\psi}^{k}(-r)-\hat{\zeta}_{r}\right\|^{2}=0$.

Let $\hat{\psi}^{k}$ be the solution of the following Cauchy problem
\begin{equation*}
\left\{
\begin{aligned}
&
\d X(t)=\left(A(X(t))+F_{\varepsilon_{n_{k}}}(t,X(t))\right)\d t
+G_{\varepsilon_{n_{k}}}(t,X(t))\d \hat{W}(t)\\
&
X(-r)=\hat{\psi}^{k}(-r)
\end{aligned}
\right.
\end{equation*}
and $\hat{Y}_{r}$ be the solution of the following Cauchy problem
\begin{equation*}
\left\{
\begin{aligned}
&
\d X(t)=\left(A(X(t))+\bar{F}(X(t))\right)\d t
+\bar{G}(X(t))\d \hat{W}(t)\\
&
X(-r)=\hat{\zeta}_{r},
\end{aligned}
\right.
\end{equation*}
where $\hat{W}$ is a cylindrical Wiener process with the identity
covariance operator on
$(\hat{\Omega},\hat{\mathcal F},\hat{\mathbb P})$.
In view of Theorem \ref{avethf}, we get
$$
\lim_{k\rightarrow\infty}\hat{ E}\sup_{-r\leq s\leq t}
\left\|\hat{\psi}^{k}(s)-\hat{Y}_{r}(s)\right\|^{2}=0
$$
for any $t\geq -r$.

Let $\zeta_{r}$ be a random variable defined on
$(\Omega,\mathcal F,\mathbb P)$ such that
$\mathcal L(\zeta_{r})=\mu_{r}$, and $Y_{r}$ be the solution of
\begin{equation*}
\d X(t)=\left(A(X(t))+\bar{F}(X(t))\right)\d t
+\bar{G}(X(t))\d W(t)
\end{equation*}
with initial value $Y_{r}(-r)=\zeta_{r}$.
Since the law of the solutions for equation \eqref{eqG2.1}
(respectively, equation \eqref{eqG5_1}) is unique,
$\mathcal L\left(\hat{\psi}^{k}\right)
=\mathcal L\left(X_{\varepsilon_{n_{k}}}\right)$
and $\mathcal L(\hat{Y_{r}})=\mathcal L(Y_{r})$ in
$Pr(C([-r,+\infty),H))$.
Then we have
\begin{equation}\label{avertheqiv9}
\lim_{k\rightarrow\infty}d_{BL}\left(
\mathcal L(X_{\varepsilon_{n_{k}}}),
\mathcal L(Y_{r})\right)=0 \qquad {\rm{in}}\quad Pr(C([-r,+\infty),H)).
\end{equation}

It follows from the tightness of
$\left\{\mathcal L\left(X_{\varepsilon_{n_{k}}}(-r-1)\right)\right\}$
that there exists a subsequence
$\{\varepsilon_{n_{k_{j}}}\}\subset\{\varepsilon_{n_{k}}\}$
such that $\mathcal L\left(X_{\varepsilon_{n_{k_{j}}}}(-r-1)\right)$
weakly converges to $\mu_{r+1}$. We can find a random variable
$\zeta_{r+1}$ on $(\Omega,\mathcal F,\mathbb P)$ such that
$\mathcal L(\zeta_{r+1})=\mu_{r+1}$.
Let $Y_{r+1}$ be the solution of
\begin{equation*}
\d X(t)=\left(A(X(t))+\bar{F}(X(t))\right)\d t
+\bar{G}(X(t))\d W(t)
\end{equation*}
with initial value $Y_{r+1}(-r-1)=\zeta_{r+1}$.
Similar to the procedure of calculating \eqref{avertheqiv9}, we get
$$
\lim_{j\rightarrow\infty}d_{BL}\left(
\mathcal L(X_{\varepsilon_{n_{k_{j}}}}),\mathcal L(Y_{r+1})\right)=0
\qquad {\rm{in}} \quad Pr(C([-r-1,+\infty),H)).
$$
Therefore, we have $d_{BL}(\mathcal L(Y_{r}),\mathcal L(Y_{r+1}))=0$
in $Pr(C([-r,+\infty),H))$. In particular,
$\mathcal L(Y_{r}(t))=\mathcal L(Y_{r+1}(t))$ for all $t\geq-r$.

Define $\nu(t):=\mathcal L(Y_{r}(t))$, $t\geq-r$. We can extract a
subsequence which we still denote by
$\{X_{\varepsilon_{n_{k_{j}}}}\}$ satisfying
$\lim\limits_{j\rightarrow\infty}d_{BL}\left(
\L(X_{\varepsilon_{n_{k_{j}}}}(t)),\nu(t)\right)=0$ in $Pr(H)$
for every $t\in\R$. In view of Theorem \ref{Boundedth},
we obtain that $\nu$ is the law of  the $L^{2}$-bounded solution of
\eqref{eqG5_1}. Therefore we have
$$
\lim_{j\rightarrow\infty}d_{BL}\left(\L(X_{\varepsilon_{n_{k_{j}}}}(t)),
\L(\bar{X}(t))\right)=0 \quad {\rm{in}}\quad Pr(H)
$$
for every $t\in\R$. By the arbitrariness of
$\{\varepsilon_{n}\}_{n=1}^{\infty}\subset(0,1]$, we have
$$
\lim_{\varepsilon\rightarrow0}d_{BL}\left(\L(X_{\varepsilon}(t)),
\L(\bar{X}(t))\right)=0 \qquad {\rm{in}} \quad Pr(H).
$$

Now we show that $\lim\limits_{\varepsilon \to 0}d_{BL}\left(
\mathcal L(X_{\varepsilon}),\mathcal L(\bar{X})\right)=0$
in $Pr(C(\R,H))$. For any $[a,b]\subset\R$, we have
$\L(X_{\varepsilon}(a))$ converges weakly to $\mathcal L(X(a))$
in $Pr(H)$. In light of Skorohod representation theorem,
there exist random variables $\hat{\hat{\psi}}_{\varepsilon}(a)$ and
$\hat{\hat{\psi}}(a)$ defined on another probability space
$(\hat{\hat{\Omega}},\hat{\hat{\mathcal F}},\hat{\hat{\mathbb P}})$ satisfying
$\lim\limits_{\varepsilon\rightarrow0}\hat{\hat{\psi}}_{\varepsilon}(a)
=\hat{\hat{\psi}}(a)$ $\hat{\hat{\mathbb P}}$-a.s., where
$\L\left(\hat{\hat{\psi}}_{\varepsilon}(a)\right)
=\L\left(X_{\varepsilon}(a)\right)$ and
$\L\left(\hat{\hat{\psi}}(a)\right)=\mathcal L\left(\bar{X}(a)\right)$.
Similar to the procedure of calculating \eqref{avertheqiv9}, we have
\begin{equation}\label{avertheqiv15}
\lim_{\varepsilon\rightarrow0}d_{BL}
\left(\mathcal L(X_{\varepsilon}),\mathcal L(\bar{X})\right)=0
\qquad {\rm{in}}\quad Pr(C([a,b],H)).
\end{equation}

The proof is complete.
\end{proof}

\begin{coro}\label{averthcoro}
Under the conditions of Theorem \ref{averth}
the following statements hold:
\begin{enumerate}
\item
If  $F\in$ $ C(\mathbb R\times H,H)$ and
$G\in C(\R\times H,L_{2}(U,H))$ are jointly stationary
(respectively, $T$-periodic, quasi-periodic with the
spectrum of frequencies $\nu_1,\ldots,\nu_k$, almost periodic,
almost automorphic, Birkhoff recurrent, Lagrange stable,
Levitan almost periodic, almost recurrent, Poisson stable) in $t$
uniformly with respect to $x$ on each bounded subset,
then equation \eqref{eqG2.1} has a unique solution
$X_{\varepsilon} \in C_{b}(\mathbb R,L^2(\Omega,\mathbb P;H))$
which is stationary (respectively, $T$-periodic,
quasi-periodic with the spectrum of frequencies
$\nu_1,\ldots,\nu_k$, almost periodic, almost
automorphic, Birkhoff recurrent, Lagrange stable, Levitan almost
periodic, almost recurrent, Poisson stable) in distribution;
\item
If  $F\in$ $ C(\mathbb R\times H,H)$ and
$G\in C(\R\times H,L_{2}(U,H))$ are Lagrange stable
and jointly pseudo-periodic (respectively, pseudo-recurrent)
in $t$ uniformly with respect to $x$ on each bounded subset, then
equation \eqref{eqG2.1} has a unique solution
$X_{\varepsilon} \in C_{b}(\mathbb R,L^2(\Omega,\mathbb P;H))$
which is pseudo-periodic (respectively, pseudo-recurrent)
in distribution;
\item
$$
\lim\limits_{\varepsilon \to 0}
d_{BL}(\mathcal L(X_{\varepsilon}),\mathcal L(\bar{X}))=0
\quad {\rm{in}}~Pr(C(\R,H)),
$$
where $\bar{X}$ is the unique stationary solution of
averaged equation \eqref{eqG5_1}.
\end{enumerate}
\end{coro}

\section{Global averaging principle in weak sense}
We recall firstly that some known definitions and lemma in dynamical
systems (see, e.g. \cite{CJLL2013, Ch2015, CV2002, KR2011, Sel}
for more details). Let $\left(X, d_{X}\right)$ and
$\left(P, d_{P}\right)$ be metric spaces.

\begin{definition}\rm
A {\em nonautonomous dynamical system}
$\left(\sigma,\varphi\right)$ (in short, $\varphi$) consists
of two ingredients:
\begin{enumerate}
  \item A {\em dynamical system} $\sigma$ on $P$ with time set
        $\mathbb T=\mathbb Z$ or $\R$, i.e.
        \begin{enumerate}
          \item [(1)] $\sigma_{0}(\cdot)=Id_{P}$,
          \item [(2)] $\sigma_{t+s}(p)=\sigma_{t}
          (\sigma_{s}(p))$ for all $t,s\in\mathbb T$
          and $p\in P$,
          \item [(3)] the mapping
          $(t,p)\mapsto\sigma_{t}(p)$ is continuous.
        \end{enumerate}
        If $\mathbb T=\R$, $\sigma$ is called {\em flow} on $P$; if
        $\mathbb T=\R^{+}$, $\sigma$ is called {\em semiflow} on $P$.
  \item A {\em cocycle} $\varphi:\mathbb T^{+}\times P\times X
        \rightarrow X$ satisfies
        \begin{enumerate}
          \item [(1)] $\varphi(0,p,x)=x$ for all
          $(p,x)\in P\times X$,
          \item [(2)] $\varphi(t+s,p,x)=\varphi(t,\sigma_{s}(p),
          \varphi(s,p,x))$ for all $s,t\in\mathbb T^{+}$
          and $(p,x)\in P\times X$,
          \item [(3)] the mapping
          $(t,p,x)\mapsto\varphi(t,p,x)$ is continuous.
        \end{enumerate}
\end{enumerate}
$P$ is called the {\em base} or {\em parameter space} and
$X$ is the {\em fiber} or {\em state space}.
For convenience, we also write $\sigma_{t}(p)$ as $\sigma_{t}p$.
\end{definition}

\begin{definition}\rm
Let $(\sigma,\varphi)$ be a nonautonomous dynamical system with
base space $P$ and state space $X$. The {\em skew product semiflow}
$\Pi:\mathbb T^{+}\times P\times X\rightarrow P\times X$ is a
semiflow of the form:
\[
\Pi(t,(p,x)):=\left(\sigma_{t}p,\varphi(t,p,x)\right).
\]
\end{definition}

\begin{definition}\rm
Define $\mathfrak{X}:=P\times X$. A nonempty compact subset
$\mathfrak{A}$ of $\mathfrak{X}$ is called {\em global attractor}
for skew product semiflow $\Pi$, if
\begin{enumerate}
  \item $\Pi(t,\mathfrak{A})=\mathfrak{A}$ for all
  $t\in\mathbb T^{+}$,
  \item $\lim\limits_{t\rightarrow+\infty}{\rm dist}_{\mathfrak{X}}
  \left(\Pi(t,D),\mathfrak{A}\right)=0$ for every nonempty bounded
  subset $D$ of $\mathfrak{X}$,
\end{enumerate}
where ${\rm dist}_{\mathfrak{X}}(A,B)$ is the Hausdorff semi-metric
between sets $A$ and $B$, i.e.
${\rm dist}_{\mathfrak{X}}(A,B):=\sup\limits_{x\in A}d(x,B)$  with
$d(x,B):=\inf\limits_{y\in B}d_{\mathfrak{X}}(x,y)$. Here
$d_{\mathfrak{X}}\left((p_{1},x_{1}),(p_{2},x_{2})\right)=
d_{P}(p_{1},p_{2})+d_{X}(x_{1},x_{2})$ for all $(p_{1},x_{1}),
(p_{2},x_{2})\in P\times X$.
\end{definition}

\begin{lemma}[see, e.g. \cite{CV2002}]\label{conta}
Let $\{S(t)\}_{t\geq0}$ be a  semiflow in a complete metric space
$\mathcal X$ having a compact attracting set $K\subset\mathcal X$,
i.e.
\[
\lim_{t\rightarrow+\infty}{\rm{dist_{\mathcal X}}}(S(t)B,K)=0
\]
for all bounded set $B\subset\mathcal X$. Then $\{S(t)\}_{t\geq0}$
has a global attractor $\mathcal A:=\omega(K)$. Where $\omega(K)$
is the $\omega$-limit set of $K$, i.e.
$\omega(K):=\cap_{t\geq0}\overline{\cup_{s\geq t}S(s)K}$.
\end{lemma}

\begin{definition}\rm
A family $D:=\{D_{p}:p\in P\}$ of subsets of $X$ is called a
{\em non-autonomous set}. If every {\em fiber} $D_{p}$
is compact, then $D=\{D_{p}:p\in P\}$ is called
{\em non-autonomous compact set}.
\end{definition}

\begin{definition}[see, e.g. \cite{CV2002}]\label{gadef}\rm
A compact set $\mathcal A\subset X$ is called the
{\em uniform attractor (with respect to $p\in P$)} of cocycle
$\varphi$ if the following conditions are fulfilled:
\begin{enumerate}
  \item The set $\mathcal A$ is uniformly attracting, i.e.
        \[
        \lim_{t\rightarrow+\infty}\sup_{p\in P}{\rm dist}_{X}
        \left(\varphi(t,p,B),\mathcal A\right)=0
        \]
        for every bounded set $B\subset X$.
  \item If $\mathcal A_{1}$ is another closed uniformly attracting
        set, then $\mathcal A\subset \mathcal A_{1}$.
\end{enumerate}
\end{definition}

\begin{remark}\rm
It follows from Definition \ref{gadef} (ii) that
the uniform attractor is unique.
\end{remark}

Denote by $\mathbb F:=(F,G)\in BUC(\R\times H,H)\times
BUC(\R\times H,L_{2}(U,H))$. Recall that
$F^{\tau}(t,x)=F(t+\tau,x)$ for all $(t,x)\in \R\times H$,
$$H(\mathbb F)=\overline{\left\{\mathbb F^{\tau}
=\left(F^{\tau},G^{\tau}\right):\tau\in\R\right\}}\subset
BUC(\R\times H,H)\times BUC(\R\times H,L_{2}(U,H)),$$
and $\left(H(\mathbb F),\R,\sigma\right)$ is a shift dynamical
system. Here $\sigma:\R\times H(\mathbb F)\rightarrow H(\mathbb F),
(\tau,\mathbb F)\mapsto \mathbb F^{\tau}$.

Let $X(t,s,x),t\geq s$ be the solution of equation
\begin{equation}\label{gaSPDE}
\d X(t)=\left(A(X(t))+F(t,X(t))\right)\d t+G(t,X(t))\d W(t)
\end{equation}
with initial condition $X(s,s,x)=x$.
Define $P_{\mathbb F}(s,x,t,\d y):=\mathbb P\circ\left(
X(t,s,x)\right)^{-1}(\d y)$. Then we can associate a mapping
$P^{*}(t,\mathbb F,\cdot):Pr(H)\rightarrow Pr(H)$ defined by
\[
P^{*}(t,\mathbb F,\mu)(B):=\int_{H}P_{\mathbb F}(0,x,t,B)\mu(\d x)
\]
for all $\mu\in Pr(H)$ and $B\in\mathcal{B}(H)$. We write $Pr_{2}(H)$
to mean the space of probability measures $\mu\in Pr(H)$ such that
$
\int_{H}\|z\|^{2}\mu(\d z)<\infty.
$
We say that $B\subset Pr_{2}(H)$ is {\em bounded} if there exists
a constant $r>0$ such that
$
\int_{H}\|z\|^{2}\mu(\d z)\leq r^{2}
$
for all $\mu\in B$. In the following, we define
\[
B_{r}:=\left\{\mu\in Pr_{2}(H):\int_{H}\|z\|^{2}\mu(\d z)
\leq r^{2}\right\}
\]
and
\[
\mathcal O_{\rho}(B):=\{\mu\in Pr_{2}(H):d(\mu,B)<\rho\}
\]
for all $r,\rho>0$,
where $d(\mu,B):=\inf\limits_{\nu\in B}d_{BL}(\mu,\nu)$.

\begin{lemma}\label{cocylem}
Consider equation \eqref{gaSPDE}. Assume that conditions
{\rm{(H1)}}, {\rm{(H2$'$)}} and {\rm(H3)--(H5)} hold. Then
$P^{*}$ is a cocycle on $\left(H(\mathbb F),\R,\sigma\right)$ with
fiber $Pr_{2}(H)$.
\end{lemma}
\begin{proof}
It follows from Lemma \ref{conlemma} that $P^{*}$ is a continuous
mapping from $\R^{+}\times H(\mathbb F)\times Pr_{2}(H)$ into
$Pr_{2}(H)$. For any $\mu\in Pr_{2}(H)$, $t,\tau\in\R^{+}$ and
$\tilde{\mathbb F}\in H(\mathbb F)$, according to the
uniqueness in law of the solutions for equation
\eqref{gaSPDE}, we have
$P^{*}(t+\tau,\tilde{\mathbb F},\mu)=P^{*}\left(t,\sigma_{\tau}
\tilde{\mathbb F},P^{*}(\tau,\tilde{\mathbb F},\mu)\right)$.
And by the definition of $P^{*}$ we have
$P^{*}(0,\tilde{\mathbb F},\cdot)
=Id_{Pr_{2}(H)}$ for all $\tilde{\mathbb F}\in H(\mathbb F)$.
\end{proof}

\begin{coro}
Under conditions of Lemma \ref{cocylem},
the mapping given by
\[
\Pi :\R^{+}\times H(\mathbb F)\times Pr_{2}(H)\rightarrow
H(\mathbb F)\times Pr_{2}(H),
\]
\[
\Pi(t,(\tilde{\mathbb F},\mu)):=\left(\sigma_{t}\tilde{\mathbb F},
P^{*}(t,\tilde{\mathbb F},\mu)\right)
\]
is a continuous skew-product semiflow.
\end{coro}

For any given $\tilde{\mathbb F}\in H(\mathbb F)$,
suppose that (H1), (H2$'$), (H3)--(H4) hold and
$2\lambda-2\lambda_{F}-L_{G}^{2}\geq0$. If $\lambda'>0$
or $2\lambda-2\lambda_{F}-L_{G}^{2}>0$, then equation
\eqref{gaSPDE} has a unique $L^{2}$-bounded solution
$X_{\tilde{\mathbb F}}$ with the distribution
$\L(X_{\tilde{\mathbb F}}(t))=:\mu_{\tilde{\mathbb F}}(t), t\in\R$.
In the following,
we denote by $\mathfrak{X}:=H(\mathbb F)\times Pr_{2}(H)$.

\begin{lemma}\label{avehullem}
Let
$$\bar{F}(x):=\lim\limits_{T\rightarrow\infty}\frac{1}{T}
\int_{t}^{t+T}F(s,x)\d s
\quad {\rm and}\quad
\lim\limits_{T\rightarrow\infty}\frac{1}{T}
\int_{t}^{t+T}\|G(s,x)-\bar{G}(x)\|_{L_2(U,H)}^{2}\d s=0$$ uniformly with
respect to $t\in\R$. Assume that $F$ and $G$ satisfy ${\rm(G1)}$--${\rm(G2)}$.
%with $\omega_1$ and $\omega_2$ respectively.
If $H(\mathbb F)$ is compact, then for
any $\tilde{\mathbb F}=(\tilde{F},\tilde{G})\in H(\mathbb F)$ we have
\begin{equation}\label{avehull}
\frac{1}{T}\left\|\int_t^{t+T}\left(\tilde{F}(s,x)-\bar{F}(x)\right)\d s\right\|
\leq\omega_1(T)(1+\|x\|)
\end{equation}
and
\begin{equation}\label{avehull2}
\frac{1}{T}\int_t^{t+T}\|\tilde{G}(s,x)-\bar{G}(x)\|_{L_2(U,H)}^{2}
\d s\leq\omega_2(T)(1+\|x\|^2)
\end{equation}
for all $T>0$, $x\in H$ and $t\in\R$.
\end{lemma}
\begin{proof}
Given
$\tilde{F}\in H(F)$, there exists $\{t_{n}\}\subset\R$ such that
\[
\lim_{n\rightarrow\infty}\sup_{|t|\leq l,\|x\|\leq r}
\|\tilde{F}(t,x)-F(t+t_{n},x)\|=0
\]
for all $l,r>0$. Then we have
\begin{align}\label{app01}
&
\frac{1}{T}\left\|\int_{t}^{t+T}\left(\tilde{F}(s,x)-\bar{F}(x)\right)\d s\right\|\\\nonumber
&
\leq\frac{1}{T}\left\|\int_{t}^{t+T}\left(\tilde{F}(s,x)-F(s+t_n,x)\right)\d s\right\|
+\frac{1}{T}\left\|\int_{t}^{t+T}\left(F(s+t_n,x)-\bar{F}(x)\right)\d s\right\|\\\nonumber
&
\leq\frac{1}{T}\left\|\int_{t}^{t+T}\left(\tilde{F}(s,x)-F(s+t_n,x)\right)\d s\right\|
+\omega_1(T)(1+\|x\|).
\end{align}
Letting $n\rightarrow\infty$ in \eqref{app01}, by Lebesgue dominated convergence theorem, we get
\begin{align*}
\frac{1}{T}\left\|\int_{t}^{t+T}\left(\tilde{F}(s,x)-\bar{F}(x)\right)\d s\right\|
\leq\omega_1(T)(1+\|x\|).
\end{align*}

The proof of \eqref{avehull2} is similar.
\end{proof}

\begin{remark}\rm\label{uesteq}
It follows from Remark \ref{hulllem} and Lemma \ref{avehullem} that estimates \eqref{ineq2},
\eqref{ANSCP}, \eqref{peq}, \eqref{gaseq}, \eqref{compeq} and
\eqref{avethfeq15} hold uniformly for all
$\tilde{\mathbb F}\in H(\mathbb F)$
and $\varepsilon\in(0,1]$.
\end{remark}

\begin{prop}\label{invlemm}
Consider equation \eqref{gaSPDE}.
Assume that conditions {\rm(H1)}, {\rm(H2$'$)}, {\rm(H3)--(H6)} hold,
and $2\lambda-2\lambda_{F}-L_{G}^{2}\geq0$. Suppose
further that $\lambda'>0$ or $2\lambda-2\lambda_{F}-L_{G}^{2}>0$.
Then we have the following results.
\begin{enumerate}
  \item Define
  $\mathfrak{A}_{\tilde{\mathbb F}}
  :=\overline{\left\{\mu_{\tilde{\mathbb F}}(t)
  \in Pr_{2}(H):t\in \R\right\}}$. Then
  $$
  P^{*}(t,\tilde{\mathbb F},\mathfrak{A}_{\tilde{\mathbb F}})
  =\mathfrak{A}_{\sigma_{t}\tilde{\mathbb F}}
  $$
  for all $t\in \R^{+}$ and $\tilde{\mathbb F}\in H(\mathbb F)$.
  \item If $H(\mathbb F)$ is compact, then
  the skew product semiflow $\Pi$ admits a global attractor
  $\mathfrak{A}:=\omega\left(H(\mathbb F)\times\overline{
  \cup_{\tilde{\mathbb F}\in H(\mathbb F)}
  \mathfrak{A}_{\tilde{\mathbb F}}}\right)$.
  Moreover, $\Pi_{2}\mathfrak{A}$ is the uniform
  attractor of cocycle $P^{*}$. Here
  $\Pi_{2}(\tilde{\mathbb F},\mu):=\mu$ for all
  $(\tilde{\mathbb F},\mu)\in H(\mathbb F)\times
  Pr_{2}(H)$.
\end{enumerate}
\end{prop}

\begin{proof}
(i) Given $t\in\R^{+}$ and $\tilde{\mathbb F}\in H(\mathbb F)$, let
$X_{\sigma_{t}\tilde{\mathbb F}}$ be the unique $L^{2}$-bounded
solution of equation
\[
\d X(s)=\left(A(X(s))+\tilde{F}(s+t,X(s))\right)\d s
+\tilde{G}(s+t,X(s))\d W(s).
\]
Note that $\L(X_{\tilde{\mathbb F}}(s+t))=\L(X_{\sigma_{t}
\tilde{\mathbb F}}(s))$ for all $s\in\R$. Consequently,
$P^{*}(t,\tilde{\mathbb F},\mathfrak{A}_{\tilde{\mathbb F}})
=\mathfrak{A}_{\sigma_{t}\tilde{\mathbb F}}$.

(ii)
As mentioned in Remark \ref{uesteq}, \eqref{compeq} in
Proposition \ref{tightprop} holds uniformly for all
$\tilde{\mathbb F}\in H(\mathbb F)$. Namely,
\[
\sup_{\tilde{\mathbb F}\in H(\mathbb F)}\sup_{t\in\R}\int_S
\|z\|^2\mu_{\tilde{\mathbb F}}(t)(\d z)<\infty.
\]
Then there exists a constant $R>0$ such that
\[
\bigcup_{\tilde{\mathbb F}\in H(\mathbb F)}\mathfrak{A}_{
\tilde{\mathbb F}}\subset\left\{\mu\in Pr_{2}(H):\int_{S}\|z\|^{2}
\mu(\d z)<R^{2}\right\}.
\]
According to the Chebychev's inequality and the compactness of the
inclusion $S\subset H$, $\overline{\cup_{\tilde{\mathbb F}\in
H(\mathbb F)}\mathfrak{A}_{\tilde{\mathbb F}}}$ is compact in
$Pr(H)$.

Let $r>0$ be an arbitrary constant. For any $\mu\in B_{r}$, take a
random variable $\xi$ such that $\L(\xi)=\mu$. Let $Y(t,\xi), t\geq0$
be the solution to
\[
Y(t,\xi)=\xi+\int_{0}^{t}\left(A(Y(s,\xi))+\tilde{F}(s,Y(s,\xi))
\right)\d s+\int_{0}^{t}\tilde{G}(s,Y(s,\xi))\d W(s).
\]
In view of Theorem \ref{gasms}, we have
\begin{equation*}
E\|Y(t,\xi)-X_{\tilde{\mathbb F}}(t)\|^{2}
\leq \begin{cases}
    E\|\xi-X_{\tilde{\mathbb F}}(s)\|_{H}^{2}\wedge \left\{\lambda'(r-2)
   (t-s)\right\}^{-\frac{2}{r-2}}, &\text{if $\lambda'>0$}\\
	{\rm{e}}^{-(2\lambda-2\lambda_{F}-L_{G}^{2})(t-s)}
    E\|\xi-X_{\tilde{\mathbb F}}(s)\|_{H}^{2},
   &\text{if $2\lambda-2\lambda_{F}-L_{G}^{2}>0$}.
  \end{cases}
\end{equation*}
Therefore, $\lim\limits_{t\rightarrow+\infty}\sup\limits_{\tilde
{\mathbb F}\in H(\mathbb F)}{\rm dist}_{Pr_{2}(H)}
\left(P^{*}(t,\tilde{\mathbb F},\mu),
\overline{\cup_{\tilde{\mathbb F}\in
H(\mathbb F)}\mathfrak{A}_{\tilde{\mathbb F}}}\right)=0$
uniformly with respect to $\mu\in B_{r}$, i.e.
$\overline{\cup_{\tilde{\mathbb F}\in H(\mathbb F)}
\mathfrak{A}_{\tilde{\mathbb F}}}$
is a compact uniformly attracting set. Obviously, $H(\mathbb F)\times
\overline{\cup_{\tilde{\mathbb F}\in H(\mathbb F)}
\mathfrak{A}_{\tilde{\mathbb F}}}$
is a compact attracting set for $\Pi$. By Lemma \ref{conta}, $\Pi$
admits a global attractor $\mathfrak{A}:=\omega\left(H(\mathbb F)
\times\overline{\cup_{\tilde{\mathbb F}\in H(\mathbb F)}
\mathfrak{A}_{\tilde{\mathbb F}}}\right)$.

Let us now prove that $\Pi_{2}\mathfrak{A}$ is the uniform attractor
of cocycle $P^{*}$. Let $B\subset Pr_{2}(H)$ be bounded, then
$H(\mathbb F)\times B$ is bounded in $H(\mathbb F)\times Pr_{2}(H)$.
Therefore,
\begin{align*}
{\rm dist_{Pr_{2}(H)}}(P^{*}(t,\tilde{\mathbb F},B),\Pi_{2}\mathfrak{A})
&
\leq{\rm dist_{\mathfrak{X}}}(H(\mathbb F)\times
P^{*}(t,\tilde{\mathbb F},B),\mathfrak{A})\\
&
={\rm dist_{\mathfrak{X}}}(\Pi(t,H(\mathbb F)\times B),\mathfrak{A})
\rightarrow 0 \quad {\rm as} \quad t\rightarrow+\infty.
\end{align*}
Next we verify the minimality property. Denote by
$\omega_{H(\mathbb F)}(B):=\cap_{t\geq0}\overline{\cup_{
\tilde{\mathbb F}\in H(\mathbb F)}\cup_{s\geq t}
P^{*}(s,\tilde{\mathbb F},B)}$. Then $\mu\in
\omega_{H(\mathbb F)}(B)$ if and only if there exist $\{\nu_{n}\}
\subset B$, $\{\mathbb F_{n}\}\subset H(\mathbb F)$ and
$\{t_{n}\}\subset\R_{+}$ such that $t_{n}\rightarrow+\infty$ and
$P^{*}(t_{n},\mathbb F_{n},\nu_{n})\rightarrow\mu$ as
$n\rightarrow+\infty$. Let $\mathcal A_{1}$ be a closed uniformly
attracting set. Then we show that $\omega_{H(\mathbb F)}
(\Pi_{2}\mathfrak{A})\subset\mathcal A_{1}$. Indeed, if this is
false, i.e. $\omega_{H(\mathbb F)}(\Pi_{2}\mathfrak{A})\not\subset
\mathcal A_{1}$. Take $\mu\in\omega_{H(\mathbb F)}
(\Pi_{2}\mathfrak{A})\setminus\mathcal A_{1}$, there exist
$\{\nu_{n}\}\subset \Pi_{2}\mathfrak{A}$,
$\{\mathbb F_{n}\}\subset H(\mathbb F)$ and
$\{t_{n}\}\subset\R_{+}$ such that $t_{n}\rightarrow+\infty$ and
$P^{*}(t_{n},\mathbb F_{n},\nu_{n})\rightarrow\mu$ as
$n\rightarrow+\infty$. Hence we have
\begin{align*}
0<d(\mu,\mathcal A_{1})
&
\leq\lim_{n\rightarrow+\infty}d(P^{*}(t_{n},\mathbb
F_{n},\nu_{n}),\mathcal A_{1})\\
&
\leq\lim_{n\rightarrow+\infty}{\rm dist_{Pr_{2}(H)}}
(P^{*}(t_{n},\mathbb F_{n},\Pi_{2}\mathfrak{A}),\mathcal A_{1})\\
&
\leq\lim_{n\rightarrow+\infty}\sup_{\tilde{\mathbb F}\in
H(\mathbb F)}{\rm dist_{Pr_{2}(H)}}(P^{*}(t_{n},\tilde{\mathbb F},
\Pi_{2}\mathfrak{A}),\mathcal A_{1})=0,
\end{align*}
a contradiction. On the other hand, for any
$(\tilde{\mathbb F},\mu)\in\omega
\left(H(\mathbb F)\times\Pi_{2}\mathfrak{A}\right)=\mathfrak{A}$,
there exist $\{\nu_{n}\}\subset\Pi_{2}\mathfrak{A}$,
$\{\mathbb F_{n}\}\subset H(\mathbb F)$, $\{t_{n}\}\subset\R_{+}$
such that $P^{*}(t_{n},\mathbb F_{n},\nu_{n})\rightarrow\mu$ and
$\sigma_{t_{n}}\mathbb F_{n}\rightarrow\tilde{\mathbb F}$ as
$n\rightarrow+\infty$. Then $\mu\in\omega_{H(\mathbb F)}
\left(\Pi_{2}\mathfrak{A}\right)$. Therefore,
$\Pi_{2}\mathfrak{A}\subset\mathcal{A}_{1}$.

The proof is complete.
\end{proof}

\begin{remark}\rm
It is known that $H(\mathbb F)$ is compact provided $\mathbb F$
is Birkhoff recurrent, which includes periodic, quasi-periodic,
almost periodic, almost automorphic as special cases.
\end{remark}

Next we prove the global averaging principle
for strongly monotone SPDEs.
\begin{theorem}\label{gath}
Suppose that $2\lambda-2\lambda_{F}-L_{G}^{2}\geq0$, {\rm(G1)--(G2)},
{\rm(H1)}, {\rm(H2$'$)} and {\rm(H3)--(H6)} hold.
Assume further that
$\lambda'>0$ or $2\lambda-2\lambda_{F}-L_{G}^{2}>0$.
If $H(\mathbb F)$ is compact, then
\begin{enumerate}
  \item the cocycle $P_{\varepsilon}^{*}$ associated with SPDE
        \eqref{eqG2.1} has a uniform attractor
        $\mathcal{A}^{\varepsilon}$ for any $0<\varepsilon\leq1$;
  \item the cocycle $\bar{P}^{*}$ associated with SPDE
        \eqref{eqG5_1} has a uniform attractor
        $\bar{\mathcal{A}}$, which is a singleton set;
  \item for arbitrary large $R_{1}$ and small $\rho>0$ there exist
        $\varepsilon_{0}=\varepsilon_{0}(R_{1},\rho)$ and
        $T=T(R_{1},\rho)$ such that for all $\varepsilon
        \leq\varepsilon_{0},~t\geq T$ and $\tilde{\mathbb F}\in H(\mathbb F)$
        \begin{equation}\label{gathbaine}
        P_{\varepsilon}^{*}(t,\tilde{\mathbb F},B_{R_{1}})
        \subset\mathcal{O}_{\rho}\left(\bar{\mathcal{A}}\right).
        \end{equation}
        In particular,
        \begin{equation}\label{gathlim}
        \lim_{\varepsilon\rightarrow0}{\rm dist}_{Pr_{2}(H)}\left(
        \mathcal{A}^{\varepsilon},\bar{\mathcal{A}}\right)=0.
        \end{equation}
\end{enumerate}
\end{theorem}

\begin{proof}
(i)--(ii)
It follows from Proposition \ref{invlemm} that
$P^{*}_{\varepsilon}$ and $\bar{P}^{*}$ admit uniform attractors,
and $\bar{\mathcal{A}}=\{\L(\bar{X}(0))\}\in Pr_{2}(H)$.
Here $\bar{X}(t),t\in\R$ is the unique
stationary solution to averaged equation \eqref{eqG5_1}.

(iii) It follows from Theorem \ref{gasms} that there exists $\delta$,
$0<\delta<\frac{\rho}{2}$ such that
\begin{equation}\label{apf1}
\bar{P}^{*}\left(t,\mathcal O_{\delta}(\bar{\mathcal{A}})\right)
\subset\mathcal O_{\frac{\rho}{2}}(\bar{\mathcal{A}})
\end{equation}
for all $t\geq0$. Fix $R_{1}$ large enough.
In view of \eqref{ineq2}, there exists
$T_{0}>0$ such that
\begin{equation}\label{spf}
P_{\varepsilon}^{*}(t,\tilde{\mathbb F},B_{R_{1}})\subset B_{R_{1}}
\end{equation}
for all $t\geq T_{0}$. Since $\bar{\mathcal{A}}$ is attractor,
we can choose $T_{1}=T_{1}(R_{1},\rho)$ so large such that
\begin{equation}\label{apf2}
\bar{P}^{*}(t,B_{R_{1}})\subset
\mathcal O_{\frac{\delta}{2}}(\bar{\mathcal{A}})
\end{equation}
for all $t\geq T_{1}$.  Denote by
$T:=\max\{T_{0},T_{1}\}$. Employing \eqref{avethfeq15}, we have
\begin{equation}\label{dpf}
\sup_{0\leq t\leq T}d\left(P_{\varepsilon}^{*}(t,
\tilde{\mathbb F},\mu),\bar{P}^{*}(t,\mu)\right)
<\eta(T,R_{1})(\varepsilon)
\end{equation}
for all $\mu\in B_{R_{1}}$ and $\tilde{\mathbb F}\in H(\mathbb F)$,
where $\eta(T,R_{1})(\varepsilon)\rightarrow0$
as $\varepsilon\rightarrow0$. Then, there exists $\varepsilon_{0}=
\varepsilon_{0}(T,R_{1})$ such that
$\eta(T,R_{1})(\varepsilon)<\frac{\delta}{2}$ for all
$\varepsilon\leq\varepsilon_{0}$.

For any $\mu\in B_{R_{1}}$, in view of \eqref{spf}--\eqref{dpf},
we have
\begin{equation*}
P_{\varepsilon}^{*}(T,\tilde{\mathbb F},\mu)\in \mathcal O_{\delta}
(\bar{\mathcal{A}})\cap B_{R_{1}}
\end{equation*}
for all $\varepsilon\leq\varepsilon_{0}$. It can be verified that
$P_{\varepsilon}^{*}(t,\tilde{\mathbb F},\mu)\in \mathcal O_{\rho}
(\bar{\mathcal{A}})$ for all $t\geq T$ and
$\varepsilon\leq\varepsilon_{0}$. To this end, define
$\mu_{1}^{\varepsilon}:=P_{\varepsilon}^{*}(T,\tilde{\mathbb F},\mu)$.
Then $\bar{P}^{*}(t,\mu_{1}^{\varepsilon})\in
\mathcal O_{\frac{\rho}{2}}(\bar{\mathcal{A}})$ and
$P_{\varepsilon}^{*}(t+T,\tilde{\mathbb F},\mu)
=P_{\varepsilon}^{*}(t,\sigma_{T}\tilde{\mathbb F},\mu_{1}^{\varepsilon})$
for all $t\geq0$. Therefore, according to \eqref{apf2}--\eqref{dpf},
we get
\[
P_{\varepsilon}^{*}(2T,\tilde{\mathbb F},\mu)\in\mathcal O_{\delta}
\left(\bar{\mathcal{A}}\right)\cap B_{R_{1}}
\]
and
\[
P_{\varepsilon}^{*}(t+T,\tilde{\mathbb F},\mu)\in
\mathcal O_{\frac{\rho}{2}+\frac{\delta}{2}}(\bar{\mathcal{A}})
\subset \mathcal O_{\rho}(\bar{\mathcal{A}})
\]
for all $t\in[0,T]$. Repeating the above procedure, we have
\[
P_{\varepsilon}^{*}(t,\tilde{\mathbb F},\mu)\in
\mathcal O_{\rho}(\bar{\mathcal{A}})
\]
for all $t\geq T$ and $\varepsilon\leq\varepsilon_{0}$.

Take $R_{1}$ large enough so that $\mathcal{A}^{\varepsilon}\subset
B_{R_{1}}$, then \eqref{gathlim} follows from \eqref{gathbaine} and
Definition \ref{gadef}.
\end{proof}

\section{Applications}
In this section, we illustrate our theoretical results
by two examples. We mainly consider the additive or linear
multiplicative noise in these examples for brevity. Let
$\Lambda\subset\R^{n},n\in\mathbb N$ be an open bounded subset and
$0<\varepsilon\leq1$. Denote by $f^{+}(t):=\max\{f(t),0\}$
for all $t\in\R$ and $\lambda_{*}$ the first eigenvalue of $-\Delta$
with the Dirichlet boundary condition.

\subsection{Stochastic reaction diffusion equations}
Consider the equation
\begin{equation}\label{SCIeq}
 \d u=\left(\Delta u-au|u|^{p-2}+\phi(t/\varepsilon)u+g(t/\varepsilon)\right)
 \d t+\kappa u\d W(t),
\end{equation}
where $W(\cdot)$ is a two-sided standard real-valued Wiener
process, $p\in[2,+\infty)$ and $g\in C_{b}(\R,H_{0}^{1,2}(\Lambda))$.
Here $a>0$ and $\kappa\in\R$ are constants.
We define $V_{1}:=H_{0}^{1,2}(\Lambda)$, $V_{2}:=L^{p}(\Lambda)$,
$H:=L^{2}(\Lambda)$,  $V:=V_{1}\cap V_{2}$ and
\[
A_{1}(u):=\Delta u,\quad A_{2}(u):=-au|u|^{p-2},\quad
F(t,u):=\phi(t)u+g(t), \quad G(t,u)=\kappa u.
\]
Assume that $\lambda_{*}-|\phi^{+}|_{\infty}-\frac{\kappa^{2}}{2}>0$,
then we have the following theorem.

\begin{theorem}\label{CIth}
\begin{itemize}
  \item [(1)] There exists a unique $L^{2}$-bounded
      solution $X_{\varepsilon}(\cdot)$ to equation \eqref{SCIeq}
      which is globally asymptotically stable in square-mean sense
      for any $0<\varepsilon\leq1$.
  \item[(2)] If $\phi$ is almost automorphic and $g$ is almost
      periodic, then the $L^{2}$-bounded solution
      $X_{\varepsilon}(\cdot)$ is almost automorphic in
      distribution.
  \item[(3)] Let $\bar{X}$ be the unique stationary solution of
      the following averaged equation
  \begin{equation}\label{GLeqave}
  \d u=\left(\Delta u-a|u|^{p-2}u+\bar{\phi}u+\bar{g}\right)\d t
   +\kappa u\d W(t),
  \end{equation}
  where
  $\bar{\phi}=\lim\limits_{T\rightarrow\infty}\frac{1}{T}
  \int_{t}^{t+T}\phi(s)\d s$ and
  $\bar{g}=\lim\limits_{T\rightarrow\infty}\frac{1}{T}
  \int_{t}^{t+T}g(s)\d s$
  uniformly for all $t\in\R$.
  Then
  \[
  \lim\limits_{\varepsilon \to 0}d_{BL}
  (\mathcal L(X_{\varepsilon}),\mathcal L(\bar{X}))=0
  \quad {\rm{in}}~Pr(C(\R,L^{2}(\Lambda))).
  \]
  \item[(4)] The cocycle $P^{*}_{\varepsilon}$ generated by
  equation \eqref{SCIeq} has a uniform attractor
  $\mathcal{A}^{\varepsilon}$, and
  $$\lim_{\varepsilon\rightarrow0}{\rm dist}_{Pr_{2}(H)}\left(
  \mathcal{A}^{\varepsilon},\bar{\mathcal{A}}\right)=0.$$
  Here $\bar{\mathcal{A}}:=\L(\bar{X}(0))$ is the attractor
  for $\bar{P}^{*}$and $H:=L^{2}(\Lambda)$.
\end{itemize}
\end{theorem}

\begin{proof}
(1)--(2) It suffices to show that conditions of Theorem \ref{gasms}
and Corollary \ref{corL2*} hold.

(H1) $A_{1}$ is obviously hemicontinuous.
We now prove that $A_{2}$ is hemicontinuous.
Let $u$, $v$, $w\in V$. For $\theta\in\mathbb{R}$,
without loss of generality, we assume $|\theta|\leq1$, then we have
\begin{align}\label{appien1}
   & ~_{V_{2}^{*}}\langle A_{2}(u+\theta v)-A_{2}(u),w\rangle_{V_{2}} \\\nonumber
   & =\int_{\Lambda}\left(-\left(u(\xi)+\theta v(\xi)\right)
   |u(\xi)+\theta v(\xi)|^{p-2}w(\xi)+u(\xi)|u(\xi)|^{p-2}
   w(\xi)\right)\d\xi\\\nonumber
   & \leq\int_{\Lambda}\left(4\left(|u(\xi)|^{p-1}
   +|v(\xi)|^{p-1}\right)|w(\xi)|
   +|u(\xi)|^{p-1}|w(\xi)|\right)\d\xi<\infty.
\end{align}
The last inequality holds since $u$, $v$, $w\in L^{p}(\Lambda)$. Then $~_{V_{2}^{*}}\langle A_{2}(u+\theta v)-A_{2}(u),w\rangle_{V_{2}}$
converges to zero as $\theta\rightarrow0$ by Lebesgue's dominated
convergence theorem. So, (H1) holds.

(H2$'$) For all $u,v\in V$ and $t\in\R$
\begin{align*}
   ~_{V_{1}^{*}}\langle A_{1}(u)-A_{1}(v),u-v\rangle_{V_{1}}
   \leq-\lambda_{*}\|u-v\|_{H}^{2},
\end{align*}
\begin{equation*}
  ~_{V_{2}^{*}}\langle A_{2}(u)-A_{2}(v),u-v\rangle_{V_{2}}
  =-a\int_{\Lambda}\left(u(\xi)|u(\xi)|^{p-2}-v(\xi)|v(\xi)|^{p-2}
  \right)\left(u(\xi)-v(\xi)\right)\d\xi\leq0,
\end{equation*}
\[\langle F(t,u)-F(t,v),u-v\rangle
\leq|\phi^{+}|_{\infty}\|u-v\|_{H}^{2},\quad
\|F(t,0)\|_{H}\leq\sup_{t\in\R}\|g(t)\|_{H_{0}^{1,2}(\Lambda)},
\]
\[
\|F(t,u)-F(t,v)\|_{H}\leq|\phi|_{\infty}\|u-v\|_{H}\quad
{\rm{and}}\quad \|\kappa u-\kappa v\|^{2}\leq\kappa^{2}\|u-v\|^{2}.
\]
So (H2$'$) holds with $\lambda=\lambda_{*}$,  $\lambda'=0$,
$\lambda_{F}=|\phi^{+}|_{\infty}$, $L_{F}=|\phi|_{\infty}$
and $L_{G}=|\kappa|$.

(H3) For all $v\in V$, $t\in\mathbb{R}$ we have
\begin{equation*}
  ~_{V_{1}^{*}}\langle A_{1}(v),v\rangle_{V_{1}}
=-\int_{\Lambda}|\nabla v(\xi)|^{2}\d\xi
=\|v\|_{H}^{2}-\|v\|^{2}_{V_{1}},
\end{equation*}
\begin{equation*}
 ~_{V_{2}^{*}}\langle A_{2}(v),v\rangle_{V_{2}}
 = -a\int_{\Lambda}|v(\xi)|^{p}\d\xi =-a\|v\|^{p}_{V_{2}}.
\end{equation*}
Then (H3) holds with $\alpha_{1}=2$, $\alpha_{2}=p$.

(H4) For all $u$, $v\in V$, $t\in\R$ we have
\begin{equation*}
\left|_{V_{1}^{*}}\langle A_{1}(u),v\rangle_{V_{1}}\right|
\leq\|\nabla u\|_{H}\|\nabla v\|_{H}
\leq\|u\|_{V_{1}}\|v\|_{V_{1}},
\end{equation*}
\begin{equation*}
   \left|_{V_{2}^{*}}\langle A_{2}(u),v\rangle_{V_{2}}\right|
   =\left|a\int_{\Lambda}-u(\xi)|u(\xi)|^{p-2}v(\xi)\d\xi\right|
    \leq a\|u\|_{V_{2}}^{p-1}\|v\|_{V_{2}}.
\end{equation*}
Therefore, we get
$\|A_{1}(u)\|_{V_{1}^{*}}\leq\|u\|_{V_{1}}$ and
$\|A_{2}(u)\|_{V_{2}^{*}}\leq a\|u\|^{p-1}_{V_{2}}$.

(2) In order to prove the almost automorphic property of the
$L^2$-bounded solution, it suffices to show that (H5)--(H6) holds.
To this end, let $S:=H_{0}^{1,2}(\Lambda)$, we define
$T_{n}=-\Delta\left(I-\frac{\Delta}{n}\right)^{-1}
=n(I-(I-\frac{\Delta}{n})^{-1})$.
Note that $T_{n}$ are continuous on $W_{0}^{1,2}(\Lambda)$.
Since the heat semigroup
$\{P_{t}\}_{t\geq0}$ (generated by $\Delta$) is contractive on
$L^{p}(\Lambda)$, $p>1$ and
$(I-\frac{\Delta}{n})^{-1}u=\int_{0}^{\infty}{\rm{e}}^{-t}
P_{\frac{t}{n}}u\d t$, $T_{n}$ are continuous on $L^{p}(\Lambda)$.

For all $u\in V$, $t\in\R$ we have
\[
  ~_{V_{1}^{*}}\langle\Delta u,T_{n}u\rangle_{V_{1}}
   \leq-\lambda_{*}\|u\|_{n}^{2}
\quad {\rm and}\quad
\phi(t)\langle u,T_{n}u\rangle_{H}=\phi(t)\|u\|_{n}^{2}
\leq |\phi^{+}|_{\infty}\|u\|_{n}^{2}.
\]
In view of the contractivity of $\{P_{t}\}_{t\geq0}$ on
$L^{p}(\Lambda)$, we have
\begin{align*}
_{V_{2}^{*}}\langle A_{2}(u),T_{n}u\rangle_{V_{2}}
&
=\langle-a|u|^{p-2}u,nu-n\left(I-\frac{\Delta}{n}\right)^{-1}
u\rangle\\
&
=n\int_{0}^{\infty}{\rm{e}}^{-t}\left(\int_{\Lambda}
-au(\xi)|u(\xi)|^{p-2}\left(u(\xi)
-P_{\frac{t}{n}}u(\xi)\right)d\xi\right)\d t\leq0.
\end{align*}
Then we obtain
\begin{align*}
2_{V^{*}}\langle A(t,u),T_{n}u\rangle_{V}
+2\langle F(t,u),T_{n}u\rangle
\leq-2\left(\lambda_{*}-|\phi^{+}|_{\infty}-\varepsilon\right)
\|u\|_{n}^{2}+C_{\varepsilon}\sup_{t\in\R}\|g(t)\|_{S}^{2}.
\end{align*}
That is, (H6) holds. And (H5) is obviuos.

(3)--(4) follows from Corollary \ref{averthcoro} and
Theorem \ref{gath}.
\end{proof}

\begin{remark}\rm
\begin{enumerate}
\item We mention that the first Bogolyubov theorem was also studied for reaction-diffusion equations with
polynomial nonlinearities by Cerrai \cite{Cerr2011} and Gao \cite{Gao}.
\item
Note that equation \eqref{SCIeq} is the real Ginzburg-Landau
equation when $p=4$.
\end{enumerate}
\end{remark}

\subsection{Stochastic generalized porous media equations}
Consider the equation
\begin{equation}\label{spme}
  \d u=\left(\Delta(|u|^{p-2}u+a u)
  +\phi\left(t/\varepsilon\right)u\right)\d t+G\d W(t),
\end{equation}
where $W(\cdot)$ is a two-sided cylindrical $Q$-Wiener process
with $Q=I$ on $L^{p}(\Lambda)$, $p>2$ and $a\geq0$,
$G\in L_{2}(L^{p}(\Lambda))$.
And there exists a constant $C_{1}>0$ such that $\phi(t)<-C_{1}$
for all $t\in\R$. We define
$$V:=L^{p}(\Lambda)\subset H:=W_{0}^{-1,2}(\Lambda)\subset V^{*}.$$

\begin{theorem}
\begin{enumerate}
\item[(1)] There exists a unique $L^{2}$-bounded solution
    $X_{\varepsilon}(\cdot)$ to equation \eqref{spme}, which is
  globally asymptotically stable in square-mean sense.
\item[(2)] The $L^{2}$-bounded
    solution $X_{\varepsilon}(\cdot)$ is almost periodic
    in distribution provided $\phi$ is almost periodic.
  \item[(3)] Let $\bar{X}$ be the unique stationary solution of
      averaged equation
  \begin{equation}\label{GLeqave}
  \d u=\left(-\Delta\left(|u|^{p-2}u+au\right)
  +\bar{\phi}u\right)\d t+G\d W(t),
  \end{equation}
  where
  $\bar{\phi}=\lim\limits_{T\rightarrow\infty}\frac{1}{T}
  \int_{t}^{t+T}\phi(s)\d s$
  uniformly for all $t\in\R$.
  Then
  \[
  \lim\limits_{\varepsilon \to 0}d_{BL}
  (\mathcal L(X_{\varepsilon}),\mathcal L(\bar{X}))=0
  \quad {\rm{in}}~Pr(C(\R,H)).
  \]
  \item[(4)] The cocycle $P^{*}_{\varepsilon}$ generated by
  equation \eqref{spme} has a uniform attractor
  $\mathcal{A}^{\varepsilon}$, and
  $$\lim_{\varepsilon\rightarrow0}{\rm dist}_{Pr_{2}(H)}\left(
  \mathcal{A}^{\varepsilon},\bar{\mathcal{A}}\right)=0.$$
  Here $\bar{\mathcal{A}}:=\L(\bar{X}(0))$ is the attractor
  for $\bar{P}^{*}$.
\end{enumerate}
\end{theorem}

\begin{proof}
Let $A(u):=\Delta\left(|u|^{p-2}u+au\right)$, $F(t,u):=\phi(t)u$ for
all $u\in V$ and $t\in\R$.
Fix $u\in V$, for all $v\in V$ we denote
\begin{equation*}
  ~_{V^{*}}\langle A(u),v\rangle_{V}
  :=-\int_{\Lambda}u(\xi)|u(\xi)|^{p-2}v(\xi)\d\xi
  -a\int_{\Lambda} u(\xi)v(\xi)\d\xi.
\end{equation*}
We first show that $A: V\rightarrow V^{*}$ is well-defined.
Indeed, employing H\"older's inequality and Young's inequality, we get
\begin{equation*}
\left|_{V^{*}}\langle A(u),v\rangle_{V}\right|
\leq \|u\|_{L^{p}}^{p-1}\|v\|_{L^{p}}
+a\left(C_{p}\|u\|_{L^{p}}^{p-1}
+C_{p}\left(|\Lambda|\right)^{\frac{p-1}{p}}\right)\|v\|_{L^{p}}
\end{equation*}
for all $u,v\in V$, where $C_{p}$ is a constant depending only on
$p$. Therefore, $A: V\rightarrow V^{*}$ is well-defined and
\begin{equation}\label{AP3ine1}
  \|A(u)\|_{V^{*}}\leq\left(1+aC_{p}\right)\|u\|_{L^{p}}^{p-1}
  +C_{p}\left(|\Lambda|\right)^{\frac{p-1}{p}}a.
\end{equation}

Similar to the proof of Theorem \ref{CIth},
it suffices to show that (H1), (H2$'$) and (H3)--(H6) hold.
Note that (H1), (H5) hold and $\lambda_{F}=0$,
$L_{F}=|\phi|_{\infty}$ in (H2$'$).

(H2$'$) For all $u$, $v\in V$, $t\in\R$ we have
\begin{align*}
~_{V^{*}}\langle A(u)-A(v),u-v\rangle_{V}
&
=-\langle u|u|^{p-2}-v|v|^{p-2},u-v\rangle_{L^{2}}
-a\|u-v\|_{L^{2}}^{2}\\
&
\leq-2^{2-p}\|u-v\|_{L^{2}}^{p}-a\|u-v\|_{L^{2}}^{2}\\
&
\leq-2^{2-p}\|u-v\|_{H}^{p}-a\|u-v\|_{H}^{2}.
\end{align*}
Therefore, (H2) holds with $r=p$, $\lambda'=2^{2-p}$
and $\lambda=a$.

(H3)  Note that for all $u\in V$, $t\in\R$
\begin{equation*}
  ~_{V^{*}}\langle A(u),u\rangle_{V}=-\int_{\Lambda}u(\xi)
  |u(\xi)|^{p-2}u(\xi)\d\xi-a\int_{\Lambda}
  u(\xi)u(\xi)\d\xi\leq-\|u\|_{V}^{p}.
\end{equation*}
That is, (H3) holds with $\alpha=p$.

(H4) holds by (\ref{AP3ine1}) with $\alpha=p$.

(H6) Let $S=L^{2}(\Lambda)$ and $\Delta$ be the Laplace operator
on $L^{2}(\Lambda)$ with the Dirichlet boundary condition.
Define $T_{n}=-\Delta\left(I-\frac{\Delta}{n}\right)^{-1}
=n\left(I-(I-\frac{\Delta}{n})^{-1}\right)$.
Then we obtain
\begin{align*}
&
~_{V^{*}}\langle\Delta(u|u|^{p-2}+au)+\phi(t)u,
-\Delta(I-\frac{\Delta}{n})^{-1}u\rangle_{V} \\
&
=-\langle u|u|^{p-2}+au,nu-n\int_{0}^{\infty}{\rm{e}}^{-t}
P_{\frac{t}{n}}u\d t\rangle_{L^{2}}+\phi(t)\|u\|_{n}^{2}\\
&
\leq -C_{1}\|u\|_{n}^{2}.
\end{align*}
That is, (H6) holds.
\end{proof}

\section*{Acknowledgements}
The authors sincerely thank the anonymous referees and the handling editor for their
very careful reading and valuable comments, which lead to significant
improvement of the paper. The authors are also grateful to Professor Michael R\"ockner for helpful
discussions and valuable suggestions during the revision of the paper. This work is partially
supported by NSFC Grants 11871132, 11925102, Dalian High-level Talent Innovation Project (Grant 2020RD09),
and Xinghai Jieqing fund from Dalian University of Technology.

\end{document}